\documentclass{article}

\usepackage[margin=1.3in]{geometry}
\usepackage{amsthm}
\usepackage[all, pdf, cmtip]{xy}
\usepackage{amssymb,amsfonts,amsmath}
\usepackage{amscd}
\usepackage[english]{babel}
\usepackage{import}
\usepackage{mathtools}
\usepackage[colorlinks=true]{hyperref}
\usepackage{comment}

\newtheorem{theorem}{Theorem}[section]
\newtheorem{corollary}[theorem]{Corollary}
\newtheorem{proposition}[theorem]{Proposition}
\newtheorem{List}[theorem]{List}
\newtheorem{conjecture}[theorem]{Conjecture}
\theoremstyle{definition}
\newtheorem{definition}[theorem]{Definition}

\theoremstyle{plain}
\newtheorem{lemma}[theorem]{Lemma}
\theoremstyle{remark}
\newtheorem{remark}[theorem]{Remark}
\newtheorem{example}[theorem]{Example}

\newcommand{\m}{\mathcal}

\def\RR{\mathbb{R}}
\def\CC{\mathbb{C}}

\def\QQ{\mathbb{Q}}
\def\FF{\mathbb{F}}
\def\ZZ{\mathbb{Z}}

\def\NN{\mathbb{N}}
\def\Z2{\ZZ/2\ZZ}

\def\({\left(}
\def\){\right)}

\DeclareMathOperator\GL{GL}

\DeclareMathOperator\HH{H}

\DeclareMathOperator\Gal{Gal}
\DeclareMathOperator{\IH}{IH}
\DeclareMathOperator{\KH}{KH}

\DeclareMathOperator\Her{Her}
\DeclareMathOperator\Quad{Quad}
\DeclareMathOperator\Rep{Rep}
\DeclareMathOperator\Res{Res}
\DeclareMathOperator\Ind{Ind}
\DeclareMathOperator\Span{Span}
\DeclareMathOperator\rank{rank}
\DeclareMathOperator\Stb{Stb}
\newcommand{\frk}{\mathfrak}

\newcommand\Cref[1]{{Corollary~\ref{#1}}}

\newcommand{\low}{_}

\newcommand{\KerH}[3]{\KH^1({#1},{#2},{#3})}
\newcommand{\ImH}[3]{\IH^1({#1},{#2},{#3})}
\newcommand{\bd}{\textbf}

\makeatother
\newcommand\IIarr{\ar@<-.4ex>[r] \ar@<.4ex>[r]}
\newcommand\IIard{\ar@<-.4ex>[d] \ar@<.4ex>[d]}
\newcommand{\IIIarr}{\ar@<-.8ex>[r] \ar[r] \ar@<.8ex>[r]}
\newcommand{\IIIard}{\ar@<-.8ex>[d] \ar[d] \ar@<.8ex>[d]}
\newcommand{\IVarr}{\ar@<-1.2ex>[r] \ar@<-.4ex>[r] \ar@<.4ex>[r] \ar@<1.2ex>[r]}
\newcommand{\IVard}{\ar@<-1.2ex>[d] \ar@<-.4ex>[d] \ar@<.4ex>[d] \ar@<1.2ex>[d]}
\newcommand{\COH}[3]{\HH^{#1}(#2,#3)}
\newcommand {\SES}[3]{\begin {CD} 1 @>>> #1 @>>> #2 @>>> #3 @>>> 1 \\ \end{CD}}
\newcommand {\LES}[4]{\begin {CD} 1 @>>> #1^{#4} @>>> #2^{#4} @>>> #3^{#4} @>\delta>> \COH{1}{#4}{#1} @>>> \COH{1}{#4}{#2} @>>> \COH{1}{#4}{#3}  \\ \end{CD}}

\makeatletter
\newcommand{\colim@}[2]{%
\vtop{\m@th\ialign{##\cr
	\hfil$#1\operator@font Hocolim$\hfil\cr
	\noalign{\nointerlineskip\kern1.5\ex@}#2\cr
	\noalign{\nointerlineskip\kern-\ex@}\cr}}%
}
\newcommand{\colim}{%
\mathop{\mathpalette\colim@{\rightarrowfill@\textstyle}}\nmlimits@
}
\makeatother
\begin{document}

%\makeatother\

\title{On the Stability and Gelfand Property of Symmetric Pairs}
\author{Shachar Carmeli}
\date{\today}

% \tracingall

\maketitle

\begin{abstract}
A symmetric pair of reductive groups $(G,H,\theta)$ is called stable, if every closed double coset of $H$ in $G$ is preserved by the anti-involution $g\mapsto \theta(g^{-1})$.

In this paper, we develop a method to verify the stability of symmetric pairs over
local fields of characteristic 0 (Archimedean and $p$-adic), using non-abelian group cohomology.
Combining our method with results of Aizenbud and Gourevitch, we classify the Gelfand pairs among the pairs
\begin{align*}
&(SL_n(F), (GL_k(F) \times GL_{n - k}(F)) \cap SL_n(F)),
(U(B_1 \oplus B_2),U(B_1) \times U(B_2)),\\
&(GL_n(F),O(B)), (GL_n(F),U(B)),
(GL_{2n}(F), GL_n(E)),(SL_{2n}(F), SL_n(E)),
\end{align*}
and the pair $(O(B_1 \oplus B_2),O(B_1) \times O(B_2))$ in the real case.

\end{abstract}
\tableofcontents
\section{Introduction}
\label{section: introduction}

A pair consisting of a reductive group $G$ and a subgroup $H$, both defined over a local field $F$ is called a Gelfand pair, if the space of $H$-invariant continuous functionals on every irreducible smooth admissible\footnote{If $F$ is Archimedean then admissible means smooth admissible Frechet representation of moderate growth. } representation of $G$ is at most one-dimensional. In this paper we consider the case when $H$ is a symmetric subgroup of $G$, i.e. the group of fixed points of an involution of $G$.

It is classically known that if $F$ is Archimedean, $G$ is a connected Lie group and $H$ is compact then $(G,H)$ is a Gelfand pair. If $G$ is not compact then this is usually not the case. However, many symmetric pairs of non-compact reductive groups were shown to be Gelfand pairs in
\cite{vD86,Fli91,BvD94,JR96,AGRS10,AG09,AG10,SZ12,AS12,Aiz13,Zha10}. These papers used the Gelfand-Kazhdan criterion and its various generalizations. The Gelfand-Khazdan criterion reduces the verification of the Gelfand property to a statement on equivariant distributions. In \cite{AG09,AG10} this distributional criterion is reduced, for many symmetric pairs, to a geometric statement that in \cite{AG09,AG10} is called "goodness", and we call \emph{stability}.
We say that a symmetric pair $(G,H)$ is \textbf{stable} if every closed double coset $HgH\subset G$ is stable under the anti-involution $g \mapsto \theta(g^{-1})$, where $\theta$ is the involution of $G$ such that $H=G^{\theta}$.

In this paper we settle the stability problem in a systematic way, producing a practical method to decide the
stability for pairs $(G,H)$.
Using this method and the results of \cite{HW93}, we show that the stability of a pair $(G,H)$ 
imply the uniqueness of open orbits of $H$ in every parabolic quotient of $G$. We call a pair with the latter property a \textbf{p-stable pair}. 
This result, combined with representation-theoretic
results due to Blanck-Delorme and van den Ban in \cite{BD08} and \cite{vdB88} respectively, shows that both the Gelfand property and the stability of a pair imply
its p-stability. We conjecture that in the real case or in the case where $G$ is quasi-split, stability and p-stability are equivalent. However, for \emph{p-adic} pairs p-stability is strictly
weaker than stability (See Example \ref{example: example 1}).

Our results for the representation theory of symmetric pairs are summarized in the following theorem
(See Theorem \ref{theorem: main theorem of paper}).

\begin{theorem}
\label{theorem: main theorem}
Let $F$ be a local field of characteristic $0$ different from $\CC$, and $E/F$ a quadratic extension. 
\begin{itemize}
\item The pair $(SL_{k+l}(F),S(GL_k(F)\times GL_{l}(F)))$ is a Gelfand pair if and only if $k
\ne l$. 
\item Let $F=\RR$. The pair $(O_{p_1+p_2,q_1+q_2}(\RR),O_{p_1,q_1}(\RR)\times O_{p_2,q_2}(\RR))$ is a Gelfand pair if and only if at least one of $p_1,p_2,q_1,q_2$ vanish. The same holds for 
$(U_{p_1+p_2,q_1+q_2}(\RR),U_{p_1,q_1}(\RR)\times U_{p_2,q_2}(\RR))$
\item Let $F$ be non-Archimedean and let $B=B_1 \oplus B_2$ be a non-degenerate Hermitian form over $E$. The pair $(U(B),U(B_1)\times U(B_2))$ is a Gelfand pair if and only if one of the forms $B_i$ has rank 1.
\item The pair $(SL_n(E),SL_n(F))$ is a Gelfand pair if and only if $n$ is odd. 
\item The pair $(SL_{2n}(F),SL_n(E))$ is a Gelfand pair if and only if either $n=1$ or $F=\RR$.
\end{itemize}
\end{theorem}
Note that symmetric pairs over over $\CC$ are automatically stable ,e.g. by \cite[Corollary 7.1.7]{AG09}, and therefore we do not consider them in this paper. 

\subsection{Outline of the Method}

We shall briefly describe our stability verification method. 
Given a group $G$ with an involution $\theta$, there is a naturally associated cohomology pointed set $\HH^1(\theta,G)$. This is simply the collection of $G$-orbits of all elements in $G$ such that $\theta(x)=x^{-1}$, by the action $\delta_g(x)=gx\theta(g)^{-1}$. 
The method for verifying the stability of a symmetric pair is based on a construction of obstruction classes leaving in such cohomology sets. A "candidate" for a counter example of the stability of a symmetric pair $(G,H,\theta)$ is a closed double coset $HgH$. We start by showing that the question whether $H\theta(g)^{-1}H=HgH$ depends only on the corresponding "coboundary" $g\theta(g)^{-1}$. It turns out that it is much easier to consider \emph{all} semi-simple elements $r\in G$ such that $\theta(r)=r^{-1}$, and then add the constraint that they represent the trivial cohomology class in $\HH^1(\theta,G)$. 

For every semi-simple element $r\in G$ such that $r=g\theta(g)^{-1}$ (i.e. $r$ represent the base-point of $\HH^1(\theta,G)$), we show that this element contradicts the stability of $(G,H,\theta)$ if and only if the corresponding cohomology class $[r]$ represented by $r$ is trivial already in $\HH^1(\theta,Z_G(r))$. In other words, $r=g\theta(g)^{-1}$ for $g$ which commutes with $r$. 
At this point, we take a direct approach. For every $r\in G$ we compute the map $\HH^1(\theta,Z_G(r))\to \HH^1(\theta,G)$ and check if the preimage of the base-point contains classes other than the basepoint. However, there is still something to say regarding the \emph{order} in which we do it. 

The easiest way to construct examples of semi-simple elements $r$ such that $\theta(r)=r^{-1}$, is to consider tori $A\subseteq G$ on which $\theta$ acts by inversion. For technical reasons we consider such tori which are split over $F$. If $r\in A$ for such a torus, then $Z_G(A)\subseteq Z_G(r)$, so if $[r]$ is trivial in $\HH^1(\theta,Z_G(A))$ it can not serve as a counter example for stability. Actually, we show that the opposite is also true: if $\alpha$ is in the kernel of the map $\HH^1(\theta,Z_G(A))\to \HH^1(\theta,G)$, then a "generic" representative of $\alpha$ in $\HH^1(\theta,Z_G(A))$ gives a counter example to the stability of $(G,H,\theta)$.

It turns out that there are only finitely many $H$-conjugacy classes of maximal such tori $A$, and hence that computing all the maps $\HH^1(\theta,Z_G(A))\to \HH^1(\theta,G)$ for various $A$-s is a relatively easy task. This allows to \emph{falsify} the stability of many pairs. For the remaining cases, our methods require to proceed directly and compute all the obstruction classes $[r]$ mentioned before. 

Even though we originally considered the criterion based on split tori mainly as a computational trick, it turns out that it has a geometric meaning of its own. We discovered that the cohomology classes coming from centralizers of maximal tori $A$ as above, obstruct the uniqueness of open orbits of $H$ in parabolic quotients of $G$. This way, using the close relation between the cohomological obstructions for stability and p-stability (i.e. the uniqueness of open $H$-orbits in parabolic quotients), we show that stability of $(G,H,\theta)$ is stronger than its p-stability.

\subsection{Structure of the Paper}

In section 2 we introduce notations, partially standard and partially specific for this article,
to be used in the theory of symmetric pairs.

In section 3 we recall the theory of non-abelian group cohomology,
which is the main tool we use for our treatment of symmetric pairs.

Recall that a symmetric pair $(G,H,\theta)$ is called stable
if every closed $H$ double-coset in $G$ is stabilized by the
anti-involution $\sigma: g \mapsto \theta(g)^{-1}$.

In section 4 we show how, using only calculation of cohomology groups, one can determine the stability of a symmetric pair.
We construct, for each semi-simple $r \in G$ of the form $g \theta(g)^{-1}$,
a cohomological obstruction $[r] \in \HH^1(\theta,G)$ for the stability of the 
of the double coset $HgH$ by the anti-involusion $\sigma$. We show that this obstruction vanishes 
if and only if $\sigma(HgH) = HgH$.  

In section 5 we consider other types of stability, in particular p-stability, and write cohomological obstructions for them.
We also state and prove several relations between these stability properties.
We show that stability implies p-stability, and give counter example for the converse.

In section 6 we use the obstructions introduced earlier in the paper to varify the stability, p-stability 
and other stability-related properties of various symmetric pairs over local fields. 
The complete list of cases we considered is summarized in table \ref{table:stable pairs} in this section.

In section 7 we prove that the Gelfand property of a symmetric pair implies its p-stability. 
Using this result, the stability results of section 6 and the methods developed in \cite{AG09} 
we prove Theorem \ref{theorem: main theorem}.

\subsection {Acknowledgments}

First, I want to deeply thank my advisor Dmitry Gourevitch for guiding me in this project.
I want to thank also some other people for their valuable comments and references.
Among them are Jeffrey Adams, Avraham Aizenbud, Joseph Bernstein, Michael Borovoi, Gal Dor,  Aloysius G. Helminck, Lev Radzivilovsky and many others.

The author is partially supported by the Adams Fellowship of the 
Israeli Academy of Science and Humanities, 
ISF grant 249/17 and ERC StG grant 637912.

\section{Definitions and Notation}
\label{section: definitions and notation}

We shall use the following standard notation in this paper:

\begin {itemize}
\item For a group $G$ and a subset $X \subseteq G$ we let $Z_{G}(X) = \{g \in G |: gx = xg \quad \forall x \in X\}$ be the centralizer of $X$.
\item If $G$ acts on $X$ and $x \in X$, we let $\Stb_{G}(x) = \{g \in G |: g(x) = x\}.$
\item $F$ will denote a local field of characteristic 0.
\item $\bar{F}$ is the algebraic closure of $F$.
\item $\textbf{H}, \textbf{G}, \text{e.t.c.}$ will denote reductive algebraic groups defined over $F$. The rational points of an algebraic group $\bd{G}$ will be denoted by the corresponding thin letter, namely $\bd{G}(F)=G$. 
\item We denote by $\Gal_F$ the absolute Galois group of $F$. For a field extension $E/F$ we let $\Gal_{E/G}$ be the relative Galois group.
\item For a torus $\textbf{A}$ we let $X^*(\textbf{A})$ (resp. $X_*(\textbf{A})$) denote the lattice of characters (resp. co-characters) of $\textbf{A}$. Namely, $X^*(\bd{A})=Hom(\bd{A},\mathbb{G}_m)$ and $X_*(\bd{A})=Hom(\mathbb{G}_m,\bd{A})$. 
\item For a character $\psi: \textbf{A} \to \textbf{G}_m$ and a one-parameter subgroup $\alpha : \textbf{G}_m \to \textbf{A}$, let $<\psi, \alpha>$ be the unique integer
such that $t^{<\psi, \alpha>} = \psi(\alpha(t))$. If $\textbf{A}$ is defined over $F$, this pairing turn $X^*(\textbf{A})$ and $X_*(\textbf{A})$ into dual $\Gal_F$-modules.
\item For a linear algebraic group $\textbf{P}$ and a torus $\textbf{A} \subseteq \textbf{P}$, let $\Phi(\textbf{A}, \textbf{P}) \subseteq X^*(\textbf{A})$ be the set of roots of
$\textbf{A}$ in $\textbf{P}$.
\item If $\Phi \subseteq V$ is a symmetric finite set of vectors in a real linear space then by a \textbf{collection of positive roots} in $\Phi$
we mean an intersection of $\Phi$ with an half-space containing no element of $\Phi$. This is compatible with the notion of positive roots
from the theory of root systems.
\end{itemize}

Let $G = \textbf{G}(F)$ be a reductive algebraic group over a local field $F$, and let $\theta$ be an involution of $G$.
We consider $G$ and all its subspaces and quotients with their usual topology arising from the topology of $F$,
sometimes called the $t$-topology.
We call the triple $(G, H, \theta)$ where $\theta: G\to G$ is an involution and $H=G^\theta$ a \textbf{symmetric pair} when
$\textbf{G}/ \textbf{H}$ is connected in the Zariski topology and $\textbf{G}$ is connected modulo its center.

The following notions will be frequently used:
\begin {itemize}
\item $\sigma(g) = \theta(g^{-1})$ is the anti-involution associated with $\theta$.
\item $H = G^\theta = \{g \in G |: \theta(g) = g\}$ will be referred to as the \textbf {orthogonal part} of G.
\item $G^\sigma=\{g\in G : \theta(g)=g^{-1}\}$ will be referred to as the \textbf{symmetric part} of G.
\item $s (g) = g \sigma (g)$ will be called the \textbf{symmetrization} map.
%\item $\bar{s} (g) =\sigma (g) g$ will be called \textbf{anti-symmetrization}.
\end{itemize}

Let $\frk{g}$ be the Lie algebra of $G$. The differential $\theta$ acts on $\frk{g}$ as an involution. For simplicity we denote the action on the Lie algebra by $\theta$ as well. the linear space $\frk{g}$ naturally decomposes into a direct sum $\frk{g} = \frk{h} \oplus  \frk{g}^{\sigma}$, where $\frk{h}=Lie(H)$. The spaces $\frk{h}$ and $\frk{g}^{\sigma}$ can be identified with the tangent spaces of $H$ and $G^{\sigma}$ respectively.

The main notion we consider in this paper is the following.
\begin {definition} 
A symmetric pair $(G,H,\theta)$ is called \textbf{stable} if every closed double coset of $H$ in $G$
is stable under $\sigma$. In other words, for each $g \in G$ such that $HgH$ is closed in $G$,
\[\sigma (HgH) = HgH.\]
\end {definition}

\begin {remark}
This notion is referred to as a \textbf{good} pair in \cite{AG09}. We chose the word stable to indicate that the notion is related
to the stability of the double cosets under the corresponding anti-involution.
\end {remark}
\begin{remark}
Here again we consider $G$ with the t-topology coming from the topology of $F$. 	
\end{remark}

\begin{definition}
	Let $(G,H,\theta)$ be a symmetric pair. An element $g\in G$ is called stable, if $\sigma(HgH)=HgH\subseteq G.$
\end{definition}
Hence, the pair $(G,H,\theta)$ is stable
if and only if every $g \in G$ with closed $H\times H$-orbit is stable.

We list now few examples of symmetric pairs.

\begin {example}
Let $G=\GL_n(\RR)$, $\theta\left(g\right) = (g^{-1})^t$.

In this case $\sigma\left(g\right) = g^t$,  $ H = O_n(\RR)$, $G^\sigma$ is the set of symmetric matrices, $\frk{h}= \frk{so}_n$,
and $\frk{g}^\sigma$ is the set of symmetric matrices in $\frk{gl}_n$. The symmetrization map is given by
$s(g) = gg^t$.
\end {example}

This example is the source for our terminology for symmetric pairs. 

\begin {example}
\label {example: quadratic pairs}
Let $\textbf{G}$ be a reductive group and let $E/F$ be a quadratic extension. Let $c$ denote the unique non-trivial element of $\Gal_{E/F}$. Then $(\textbf{G}(E),G,c)$,
when considered as a pair over $F$,  is a symmetric pair.
\end {example}

Recall that the adjoint group of $G$ is defined as the quotient $Ad_G := G/Z(G)$ of $G$ by its center $Z(G)$, and for each $x \in G$ the corresponding element $Ad_x \in Ad_G$ acts by conjugation on $G$, as well as on the Lie algebra $\frk{g}$.

\begin {example}
\label {example: inner pairs}
Let $G$ be a reductive group defined over $F$ and let $h \in G$ be an element such that $Ad_h$ has order $2$. Then $(G, Z_G(h), Ad_h)$
is a symmetric pair. The example works also in the case $Ad_h$ is of order $2$ in $Ad_G$.
\end {example}

\begin {example}
\label {example: diagonal pairs}
Let $G$ be a reductive group, and let $\Delta(G)= \{(g,g) |: g \in G\} \subseteq G$ be the
diagonal subgroup of $G \times G$. Then $(G \times G, \Delta(G), (x,y) \mapsto (y,x))$ is a symmetric pair.
\end {example}

\begin{example}
\label{example: riemanian pairs}
Let $G$ be a reductive Lie group and $\theta$ a Cartan involution of $G$. Then $H$ is a maximal compact subgroup, and $(G,H,\theta)$ is called a \textbf{Riemannian symmetric pair}.
\end{example}

\section{Preliminaries on Group Cohomology}
\label{section: preliminaries on group cohomology}

The language of non-abelian first group cohomology is extremely useful in order to verify whether a given pair is stable or not.
In this section we briefly remind some basics of the theory.
We skip the proofs and the details.
A complete treatment of the subject can be found in \cite[\S5]{Ser97}.

\subsection {Definition of Group Cohomology}

Let $L$ be a group. 
A group $G$ together with an $L$-action will be refered to as an $L$-group.
To every $L$-group $G$ we can consider $G^L$, the group of $L$-fixed points in $G$.
The first cohomology of $L$ with coefficients in $G$ is a measure for the lack of
right-exactness of the fixed-points functor. It is a pointed set denoted $\HH^1(L,G)$.

Recall that, for a group $L$ acting on 
$G$, one defines a 1-cocycle of $L$ with value in 
$G$ to be a function 
$L \to G$, $l \mapsto a_l$ such that 
\[a_{lt} = a_l l(a_t).\]  
The set of 1-cocycles of $L$ with coefficients in $G$
will be denoted $Z^1(L,G)$. 

The group $G$ acts from the right on $Z^1(L,G)$ by 
\[(\delta_g(a))_l := g^{-1} a_l l(g).\]
We call this action the \textbf{coboundary action}. 
The quotient $Z^1(L,G)/G$ of this action is, by definition, 
$\HH^1(L,G)$. It is a pointed set, with base-point 
the (orbit of) the coboundaries. 
Clearly, $\HH^1(L,G)$ is covariant functor of $G$ and a contravariant 
functor of $L$. 
If $f: G \to G'$ we denote by $f_*$ the induced map on cohomology. 
If $K \subseteq G$ is an inclusion of $L$-groups, we shall denote by $i_K^G$ that inclusion. 

\begin{definition}
Let $K \subseteq G$ be an $L$-subgroup of an $L$ group $G$.
We will denote
\begin{eqnarray*}
&\KerH{L}{K}{G} := Ker((i_K^G)_*), \\
&\ImH{L}{K}{G} := Im((i_K^G)_*).
\end{eqnarray*}
\end{definition}
Recall that the kernel of a map $f: (X,x_0) \to (Y,y_0)$ of pointed sets is by definition
$f^{-1}(y_0)$.

\subsection {Some Properties of non-abelian Group Cohomology}

The first and most foundamental property of 
non-abelian group cohomology is the existence of long 
exact sequence of cohomologies associated with a 
short exact sequence of $L$-groups.

Let
\[\SES{H}{G}{K}\]
be an exact sequence of $L$-groups. There is an induced 
sequence
\[\LES{H}{G}{K}{L}\]
which is an exact sequence of pointed sets.

Exactness for sequences of pointed sets is weaker than the corresponding property for sequences of groups.
The existence of long exact sequence can be strengthened in two ways.
Firstly, one can replace the group $K$ in the sequence with a pointed $L$-set on which $G$ acts transitively (not preserving the base-point), and we still get an exact sequence of the form
\[\HH^0(L,H) \to \HH^0(L,G) \to \HH^0(L,K) \to \HH^1(L,H) \to \HH^1(L,G)\] 
where $H$ is the stabilizer of the base-point of $X$.

Secondly, for an injection $f: H \to G$ of $L$-groups one can describe all the fibers of the map $f_*: \HH^1(L,K) \to \HH^1(L,G)$,
and not only the pre-image of the trivial cocycle. The description of the fibers is given by the twisting operation,
to be described in the next subsection.

The following classical consequence of the exact sequence will be 
extensively used.
\begin{proposition}[{\cite[Corollary 4.6]{Ber10}}]
\label {proposition: descent lemma}
Let $G$ be an $L$-group acting transitively on an $L$-set $X$ possessing an $L$-fixed point $x_0$. Let $K = \Stb_G(x_0)$. Then $\delta$ defines
a bijection \[X^L/G^L \cong \KerH{L}{K}{G}.\]
\end {proposition}
 .

\subsection {Twisting}
An important feature of non-abelian group cohomology is the twisting operation. Let $G$ be an $L$-group acting on an $L$-set $X$.
Let $a$ be a 1-cocycle of $L$ with coefficients in $G$. Then, using $a$, it is possible to twist the action of $L$ on $X$, as well as on
$G$, to obtain a new pair $(G,X)$ of an $L$-group acting on an $L$-set. It is done in the following way. Define a new
action of $L$ on $X$, denoted $x \mapsto l*x$, by \[l*(x) := a\low{l}(l(x)).\]
Similarly, define a new action of $L$ on $G$ by \[l*(g) := a\low{l}\cdot l(g) \cdot a\low{l}^{-1}.\]
The pair $(G, X)$ with the new actions is the twisting of $X$ and $G$ by $a$. In order to distinguish it from the original pair,
we denote it by $(\tau_a(G), \tau_a(X))$.
A routine calculation shows that this operation really gives an action of 
$L$ on $G$.
One of the most important properties of the twisting operation, is that it
induces an equivalence on first cohomology sets. 
More precisely, given a cocycle $a_l$ of $L$ with coefficients in $G$, 
the map 
$b_l \mapsto b_l a_l$ gives a bijection 
$\tau_{a} : \HH^1(L,\tau_a(G)) \to \HH^1(L,G)$, mapping the 
trivial cocycle to the class of $a_l$. 
This twisting operation is, moreover, compatible with long exact sequences in the obvious sense. 

The twisting operation is important for two reasons.
Firstly, it allows one to compute cohomologies of a twisted pair by mean of the cohomology of the original pair.
Secondly, using twisting we can describe completely the fibers of 
$(i_K^G)_* : \HH^1(L,K) \to \HH^1(L,G)$ for $K \subseteq G$. 
Indeed, if $a_l$ is a cocycle of $L$ with coefficients in 
$K$, then the fiber $(i_K^G)_*^{-1}((i_K^G)_* a_l)$ correspond via twisting 
by $a_l$ to the kernel of the map $\HH^1(L,\tau_a (K)) \to \HH^1(L,\tau_a(G))$. 

\subsection {Galois Cohomology}
Suppose that $L$ is the Galois group of a separable extension $E/F$, and that $\textbf{G}$ is an algebraic group
defined over $F$. Then $\COH{1}{L}{G(E)}$ is called the \textbf{Galois cohomology} of the extension $E/F$ with coefficients in $G$,
and denoted $\COH{1}{E/F}{G}$.
If $E = F^{sep}$ is the separable closure of $F$, we even write it as $\COH{1}{F}{G}$.
The following computation is basic for many other computations of cohomologies.
\begin{remark}
Note that in general Galois cohomology is defined as the continuous cohomology group of the profinite Galois group, but we consider only finite extensions here anyway, so this issue is irrelevant for us.
\end{remark}

\begin{proposition} [{Hilbert's 90 Theorem}]
\label {proposition: hilbert 90}
Let $E/K$ be a Galois extension. Then \[\COH{1}{E/F}{GL_n} = 1.\]
\end{proposition}

The following consequence of Hilbert's 90 Theorem is useful for us as well.
\begin {proposition}
\label {proposition: vanishing theorem for centralizers}
Let $E/F$ be a separable extension, and let $A \in M_{n \times n}(F)$.
Then \[\HH^1(E/F, Z_{\GL_n}(A)) = 1.\]
\end {proposition}

\begin {proof}
Let $\textbf{X}$ denote the conjugacy class of $A$ in $M_{n \times n}(E)$. Then Theorem \ref{proposition: descent lemma} shows that
\[\KerH{E/F}{Z_{GL_n}(A)}{GL_n} \cong \textbf{X}(F) / GL_n(F)\]
where $GL_n(F)$ acts on $\textbf{X}(F)$ by conjugation.

By Hilbert's 90 Theorem,
\[\COH{1}{E/F}{GL_n} = 1\]
and therefore
\[\HH^1(E/F, Z_{GL_n}(A)) \cong \textbf{X}(F) / GL_n(F).\]
But it is a classical theorem in linear algebra that
if two $F$-matrices are conjugate over $E$, they are also conjugate over $F$. We deduce that $GL_n(F)$ acts transitively on
$\textbf{X}(F)$.
\end {proof}

\section{Cohomology and Stability}
\label{section: cohomology and stability}

In this section we shll introduce the method for varifying stability of symmetric pairs using non-abelian cohomology. 
 
\begin {definition}

Let $G$ be a topological group acting continuously on a topological space $X$. An element $x \in X$ is called $G$-semi-simple
if the orbit $Gx \subset X$ is closed.
\end {definition}
When a reductive group $G$ over an algebraically closed field acts on itself by conjugation, this definition is consistent with the usual definition
of semi-simplicity.

The stability of a symmetric pair $(G,H,\theta)$ can now be restated as follows:
The pair $(G,H,\theta)$ is stable if and only if every $H\times H$-semi-simple element of $G$ is stable.

\subsection {Interpretation of the Cohomology of a Symmetric Pair}
\label{subsection: Interpretation of cohomology}

In this section we recall an explicit description of the cohomology set $\HH^1(\ZZ/2\ZZ, G)$ with $\ZZ/2\ZZ$ acting by the identification $\ZZ/2\ZZ = \{Id_G,\theta\}$.
This treatment of involutions using cohomology can be found in \cite[\S II.5]{BJ05}.
To emphesize the dependence on the involution $\theta$, we denote the cohomology of this $\ZZ/2\ZZ$ action by $\HH^1(\theta,G)$.

\begin {lemma}
\label {lemma: interpretation of cohomology}
Let $(G,H,\theta)$ be a symmetric pair. Let $G$ acts on $G^\sigma$ by \[\delta_g(r) = g^{-1} \cdot r \cdot \theta(g).\]
Then there is a natural bijection $\HH^1(\theta,G) \equiv S/\delta_{G}$, given by $a \mapsto a_{\theta}$.
\end{lemma}

\begin{proof}
Let $a_{\theta}$ be a cocycle. The cocycle condition gives 4 equations satisfied by $a$:
\begin{align}
&a_{Id} = a_{Id^2} = a_{Id} \cdot Id(a_{Id}) = a_{Id} \cdot a_{Id} \\
&a_{_\theta} = a_{Id\cdot \theta} = a_{{Id}} \cdot Id(a_{\theta}) = a_{{Id}} \cdot a_{{\theta}} \\
&a_{\theta} = a_{{ \theta \cdot Id}} = a_{{\theta}} \cdot \theta(a_{{Id}}) \\
&a_{{Id}} = a_{{ \theta \cdot \theta}} = a_{{\theta}} \cdot \theta(a_{{\theta}})
\end{align}

Condition (1) is equivalent to $a_{Id} = e_{G}$. Substituting this into (2) and (3) we see that they are consequences of (1).
Finally, condition (4) is equivalent to $a_{\theta} \in G^\sigma$.
Next, we describe the coboundary action on these cocycles. This is given by
\[\delta_g(a)_{Id} =  g^{-1} \cdot a_{Id} \cdot g = g^{-1} \cdot g = e_{G}\]
\[\delta_g(a)_{\theta} =  g^{-1} \cdot a_{\theta} \cdot \theta(g)\]
\end{proof}

\begin {corollary}
Let $r \in G^\sigma$. Then $r$ is a symmetrization of an element in $G$ if and only if the corresponding cohomology class $[r] \in \HH^1(\theta,G)$
is trivial.
\end {corollary}

According to Lemma \ref{lemma: interpretation of cohomology}, 
$\HH^1(\theta, G)$ has a natural topology, the quotient topology 
from the projection $G^\sigma \to \HH^1(\theta, G)$. We shall prove that the induced topology on $\HH^1(\theta, G)$ is discrete. Equivalently,
the $G$-orbits in $G^\sigma$ are all open. This seem to be a well known result, and will play a central role in our approach to stability.

In order to prove this, we use the open mapping theorem for the action.

\begin {definition}
\label{definition: submersive action}
Let $G$ be an analytic F-group acting on an analytic manifold $X$. Then $G$ acts \textbf{submersively} if for every $x \in X$,
the action map \[\phi_x : G \rightarrow X\] given by $\phi_x(g) :=  g(x)$ is submersive.
\end {definition}

Equivalently, the action is submersive if the
global vector fields on $X$ induced by $\frk{g}$ span the tangent space of $X$ at every point.

By the open mapping theorem for analytic manifolds over $F$, if $G$ acts submersively on $X$ then every orbit of $G$ in $X$ is open.
Therefore, in order to prove that $\HH^1(\theta, G)$ is discrete, it will be sufficient to prove that $G$ act submersively on $Z^1(\theta, G) \cong G^\sigma$, where the action on the right object is induced from the $\delta$-action of $G$ on cocycles.

\begin{definition}
	Let $(G,H,\theta)$ be a symmetric pair. Denote by $G^\sigma_0$ the image of $s:G\to G^\sigma$. 
\end{definition}
In other words, $G^\sigma_0$ is the orbit of the identity element by the $\delta$-action.

\begin {theorem}
\label {theorem: submersive action}
Let $(G,H,\theta)$ be a symmetric pair. Then the action of $G$ on $G^\sigma$ given by $\delta$ is submersive.
\end {theorem}

\begin{proof}
We have to prove that, at any point $r \in G^\delta$, the action map $\phi_r(g) = g^{-1} r \theta(g)$ is submersive.
Consider first the special case $r = e_{G}$. At this point, we get \[\phi_{e_{G}}(g) = s(g^{-1}).\] We have a direct decomposition $\frk{g} = \frk{h} \oplus \frk{g}^\sigma$.
The map $ds$ is given by \[ds(X + Y) = 2Y,\] for $X \in \frk{h}, Y \in \frk{g}^\sigma$. This map is clearly onto $T_{e_{G}}G^\sigma \cong \frk{g}^\sigma$.
Since the action of $G$ on $G^\delta_0$ is transitive, this shows that the boundary action of $G$ is submersive at every point of $G^\sigma_0$.
For a point $r \in G^\sigma - G^\sigma_0$, we use twisting. Let $r \in G^\sigma$, and let $G' = \tau_{r}(G)$. We add $'$ to things related to $G'$, so that $(G^\sigma_0)'$ will be cocycles of
the twisted action and $\phi '$ the action map for the twisted action. Let $Id: G' \to G$ be the identity map. Consider the commutative square

\[
\begin {CD}
(G^\sigma)' @>\cdot r >>                         G^\sigma \\
@ A\phi'_eAA                                                                    @A\phi_rAA \\
                G'                      @>Id>>                                                  G \\
\end{CD}
\]

As the two horizontal maps are isomorphisms, the submersivity of $\phi_r$ follows from that of $\phi'_e$ which is the special case $r = e$ for the $\mathbb{Z}/2\mathbb{Z}$ group $G'$.
%\qed
\end{proof}

\begin{corollary}
The orbits of $G$ in $G^\sigma$ are open in the t-topology.
\end{corollary}

\begin {corollary}
\label{corollary: cohomology is discrete}
$\COH{1}{\theta}{G}$ is discrete for every symmetric pair $(G,H,\theta)$.
\end{corollary}

\subsection {Equivalent Condition to Stability}

Let $(G, H, \theta)$ be a symmetric pair. We would like to describe stable elements of $G$ in terms of the conjugacy classes of their symmetrizations. 
The map $s$ is submersive everywhere, hence open. We can view it
as the action map of the transitive left action of $G$ on $G^\sigma_0$ given by $g \mapsto \delta_{g^{-1}}$. The map $s$ is invariant
under multiplication from the right by $H$, as if $\theta(h) = h$, then $s(gh) = ghh^{-1}\sigma(g) = s(g)$. Therefore,
$s$ induces a submersive $G$-equivariant map $s:G/H \rightarrow G^\sigma_0$.

\begin {lemma}
\label {lemma: injectivity of symmetrization on orbits}
The symmetrization map defines an isomorphism of analytic $F$-manifolds
\[s:G/H \stackrel{~}{\longrightarrow}G^\sigma_0.\]
\end{lemma}

\begin {proof}
As $G^\delta_0$ is a homogenous space for $G$, it is sufficient to prove that $H$ is the stabilizer of $e$ in $G^\sigma$. But
\[s(g) = e \Leftrightarrow g\sigma(g) = e \Leftrightarrow g = \theta(g) \Leftrightarrow g \in H\]
\end{proof}

\begin{remark}
	Note that in general both sides are not the $F$-points of an algebraic variety, e.g. in the case $G=GL_n$ and $\theta(x)=(x^{t})^{-1}$. However, they form an open subset of the $F$-points of an algebraic variety, namely of $(\bd{G}/\bd{H})(F)$.  
\end{remark}

\begin{proposition}
	\label{proposition: reducing stability to conjugacy problem}
	An element $g \in G$ is stable, if and only if $s(g)$ and $s(\sigma(g))$ are conjugate by an element of $H$.
\end{proposition}

\begin{proof}
The action of $H$ on $G/H$ by multiplication from the left is transformed by $s$ to conjugation. It follows from
Lemma \ref{lemma: injectivity of symmetrization on orbits} that $H\backslash G/H \cong G^\sigma_0/Ad_H$. The action of $\sigma$ on
the double-cosets space is transformed by $s$ to the action $Ad_H(s(g)) \mapsto Ad_H(s(\sigma(g)))$.
The result follows easily from these observations. 
\end{proof}

This means that we can check the stability of an element by looking at the $H$-conjugacy class of
its symmetrization. Note that in particular, the stability of $g$ depends only on the $H$-conjugacy class of $s(g)$.

In view of this result, and using slightly ambiguous, yet convenient language, we call a semi-simple element $r \in G^\sigma_0$ stable, 
if $r = s(g)$ for a stable element $g\in G$. In other words, if $r=s(g)$ is conjugate to $s(\sigma(g))$. 
Using this terminology, 
$g$ is stable if and only if $s(g)$ is. From now on we shall focus on the stability of elements of $G^\sigma$, not to be confused with their stability as elements of $G$.   

As the map $s$ is a homeomorphism from $G/H$ to $G^\sigma_0$, an element $g\in G$ is $H\times H$ semi-simple
if and only if $s(g)\in G^\sigma_0$ is $Ad_H$ semi-simple.
In  \cite[Theorem 7.2.1]{AG09},
it is proved that the $Ad_{H}$ semi-simplicity of $r \in G^\sigma$ implies its $Ad_{G}$ semi-simplicity. The following is a slight improvement of this result. 

\begin{lemma}
\label{lemma: orbits of Ad_H}
Let $(G,H,\theta)$ be a symmetric pair. The orbits of $Ad_H$ in $Ad_{G}(r) \cap G^\sigma$ are open and closed. 
\end{lemma}
In other words, the quotient $(Ad_{G}(r) \cap G^\sigma)/Ad_{H}$
is discrete.
\begin{proof}
We shall define a continuous injection
\[\mu: (Ad_{G}(r) \cap G^\sigma)/Ad_{H} \hookrightarrow \HH^1(\theta, Z_{G}(r)).\] Given this, the result follows from the discreteness of $\HH^1(\theta, Z_{G}(r))$, which in turn follows from Corollary \ref{corollary: cohomology is discrete}. 
Let $g \in G$, and assume that $grg^{-1} \in G^\sigma$. Then $\theta(g)\theta(r) \theta(g)^{-1} = g r^{-1} g^{-1}$.
On the other hand,
\[\theta(g)\theta(r) \theta(g)^{-1} = \theta(g) r^{-1} \theta(g)^{-1}.\]
Taking inverses,
we get that
\[grg^{-1} = \theta(g) r \theta(g^{-1}).\]
Therefore, the coboundary $s(g^{-1}) = g^{-1}\theta(g)$ belongs to $Z_G(r)$. Set \[\mu(grg^{-1})=[s(g^{-1})]\in \HH^1(\theta,Z_G(r)).\]
The map $\mu$ does not depend on the choice of $g$. Indeed,
if $grg^{-1} = g'rg'^{-1}$, then $g'^{-1}g \in Z_{G}(r)$. So
\[s(g^{-1}) = s((gg'^{-1}g')^{-1}) = \delta(g'^{-1}g)(s(g^{-1}))\]
which is equivalent to $s(g^{-1})$ as a cocycle of $Z_{G}(r)$.
It remains to verify the injectivity of $\mu$, the continuity of it being clear. If $s(g^{-1}) = s(g'^{-1})$, the elements $g$ and $g'$ differ
by a multiplication from the left by an element of $H$, by Lemma \ref{lemma: injectivity of symmetrization on orbits}. Writing
$g = hg'$ for $h \in H$, we get that $grg^{-1} =h(g'rg'^{-1})h^{-1}$. This does not change the $Ad_H$-orbit, so the map is injective.
%\qed
\end{proof}

%\begin{remark}
%Lemma \ref{lemma: orbits of Ad_H} do not follows directly from Theorem \ref{proposition: descent lemma}, as we take "anti-fixed points" instead of
%fixed points. But the proofs of the two are identical. In particular, the morphism we constructed in the proof is actually an isomorphism
%of topological spaces
%\[(Ad_{G}(r) \cap S)/Ad_{H} \cong \KerH{\theta}{Z_{G}(r)}{G}\]
%\end{remark}

As a result, we get 
\begin{theorem}
\label{the three are equivalent}
Let $(G,H,\theta)$ be a symmetric pair, and let $g\in G$. The following are equivalent:
\begin {enumerate}
\item $g$ is $H\times H$ semi-simple.
\item $s(g)$ is $Ad_{H}$ semi-simple.
\item $s(g)$ is $Ad_{G}$ semi-simple.
\end{enumerate}

\end{theorem}

\begin {proof}
$(1)\Leftrightarrow(2)$ is a consequence of Lemma \ref{lemma: interpretation of cohomology} The implication
$(2)\Rightarrow(3)$ is already proved in \cite[Proposition 7.2.1]{AG09},
 and $(3)\Rightarrow(2)$ is
a consequence of Lemma \ref{lemma: orbits of Ad_H}.
%\qed
\end {proof}

We conclude with an application to the stability of Riemannian pairs, a fact which both simple and well known, but still serve as a good example of the theory developed so far. 

\begin{theorem}
\label{theorem: riemannian pairs}
Let $(G,H,\theta)$ be a Riemannian symmetric pair. Then $(G,H,\theta)$ is stable.
\end{theorem}

\begin{proof}
It will suffice to show that $s(g)$ and $s(\sigma(g))$ are conjugate for every $g$ for which $s(g)$ is semi-simple.
But $G^\delta_0 = s(G)$ is a Riemannian symmetric space and hence a complete Riemannian manifold. It follows that $exp: \frk{g}^\sigma \to G^\delta_0$
is onto, and hence $s(g) = exp(X) = [exp(X/2)]^2$ for some $X \in G^\delta_0$. We deduce that $s(g) = s(exp(X/2))$ and therefore
$Ad_H(s(\sigma(g))) = Ad_H(exp(X)) = Ad_H(s(g))$.
%\qed
\end{proof}

\subsection{Cohomological Obstructions for Stability}

In the last section we reduced the problem of the stability of $g \in G$ to the problem of the $H$-conjugacy of $s(g)$ and $s(\sigma(g))$.
In this section we construct cohomological obstructions for their conjugacy. 
The constructions are both based on Theorem \ref{proposition: descent lemma}, 
and on an analysis of the relation between cohomologies of $G$ and of its descendants, in the following sense. 

\begin {definition}
Let $(G,H,\theta)$ be a symmetric pair, and $r$ a semi-simple element of $G^\sigma$. The pair
$(Z_{G}(r),Z_{H}(r), \theta)$ is a \textbf{descendant} of the pair $(G,H,\theta)$.
\end{definition}

Note that here we do not require $r$ to be in $G^\delta_0$, but we shall use this 
notion only in that case. 

Let $(G,H,\theta)$ be a symmetric pair. Let $r \in G^\delta_0$. We shall now define two cohomology 
classes associated with $r$. 

For the first one, we have to assume that $\mathbf{G}$ connected over $\bar{F}$, so that it is
stable over $\bar{F}$, by \cite[Corollary 7.1.4]{AG09}. 
Let $(G,H,\theta)$ be a symmetric pair with $G$ connected in the Zariski topology. 
Let $r = s(g)$ be a semi-simple symmetrization. 
The group $H$ act by conjugation on $Ad_\textbf{H}(s(r))$, and by the stability of the pair over 
$\bar{F}$, we have $\bar{r} := s(\sigma(g)) \in Ad_\textbf{H}(s(r))$. 

Since $Ad_\mathbf{H}(s(r)) \cap \mathbf{G}^\sigma(F) \cong \KerH{F}{Z_H(r)}{H}$, we can associate with 
$\bar{r}$ a unique class $[r]_1 \in \KerH{F}{Z_H(r)}{H}$. 

\begin{definition}
	The class $[r]_1$ is called the \textbf{first obstruction} of $r$. 
\end{definition}

Note that the first obstruction vanishes if and only if $r$ and $\bar{r}$ are H-conjugate, or in other words if and only if $r$ is stable. 
In practice, we will use only a second, less well known obstruction. The reason is that it is given 
by a simpler formula. 

Let $r \in G^\delta_0$. Then $g\bar{r}g^{-1} = r$ so $r = g \sigma(g)$ and $\bar{r} = \sigma(g) g$ are 
conjugate in $G$.  
Consider the action of $\sigma$ on $G$. The relation 
$\sigma(ghg^{-1}) = \theta(g) \sigma(h) \theta(g)^{-1}$ allows us to consider 
$G$ as a $\ZZ/2\ZZ$-set equipped with an action of the $\ZZ/2\ZZ$-group $G$. 
Inside $G$ we have the orbit of $r$, namely $Ad_G(r)$. By Theorem 
\ref{proposition: descent lemma} we deduce that 
\[Ad_G(r) \cap G^\sigma / Ad_H(r) \cong \KerH{\theta}{Z_G(r)}{G}.\] 
Since $\bar{r}$ gives an element of the left hand pointed set, we get an associated cocycle 
$[r]_2 \in \KerH{\theta}{Z_G(r)}{G}$, which vanishes if and only if $r$ is stable. 

\begin{definition}
	The class $[r]_2$ is called the \textbf{second obstruction} of $r$. 
\end{definition} 

We shall now identify the second obstruction as a very simple object. 

\begin{proposition}
\label{proposition: formula of second obstruction}
The class $[r]_2$ is the cohomology class represented by $r$. Namely, 
$[r]_2 = [r] \in \HH^1(\theta, Z_G(r))$. 	
\end{proposition}	

\begin{proof}
	Recall how the isomorphism $Ad_G(r) \cap G^\sigma / Ad_H(r) \cong \KerH{\theta}{Z_G(r)}{G}$ is defined. 
	Given a point $x \in Ad_G(r) \cap G^\sigma / Ad_H(r)$, we find 
	$y \in G$ such that $yry^{-1} = x$, and then the resulting cocycle is $\delta(y) = y^{-1} \theta(y)$. 
	In our case, since $g^{-1} r g = \bar{r}$, we get that the corresponding cocycle is 
	$\delta(g^{-1}) = g \sigma(g) = r$.   
\end{proof}

To summarize, we get the following result, which will be extensively used in the sequel.  

\begin{proposition}
\label{proposition: the centralizer criterion}

\label{proposition: symmetrization in centralizer}

\label{proposition: Galois invariant}
	
Let $(G,H, \theta)$ be a symmetric pair. 	
The following are equivalent: 

\begin{itemize}
	\item $r \in G^\delta_0$ is stable. 
	\item $[r]_2 = 1$ in $\HH^1(\theta,Z_G(r))$. 
	\item (if $\mathbf{G}$ is connected) $[r]_1 = 1$ in $\HH^1(F,Z_H(r))$.    
\end{itemize}
\end{proposition}

%Finally, we demonstrate the usefulness of the obstruction-theoretic approach to stability 
%via the following result, which allows to show that a pair is is not stable 
%without referring to a specific unstable element.  
%The criterion is based on the following notion. 

%\begin{definition}
%A torus $\textbf{A} \subseteq \textbf{G}$ is called \textbf{$\theta$-split} if  $\theta(x) = x^{-1}$ for every $x \in \textbf{A}$.
%\end{definition}

%As usual, if $\textbf{A}$ is defined over $F$ we denote its $F$-points by $A$, and refer to $A$ as a $\theta$-split F-torus.

%\begin{theorem}
%\label{theorem: cohomological criterion for stability}
%If a symmetric pair $(G,H,\theta)$ is stable and $A$ is a maximal $\theta$-split $F$-torus of $G$,
%then
%\[\ImH{\theta}{A}{Z_G(A)} \cap \KerH{\theta}{Z_{G}(A)}{G}  = 1.\]
%\end{theorem}

%\begin{proof}
% Assume that $(G,H,\theta)$ is stable, and let
%\[\alpha \in \ImH{\theta}{A}{Z_G(A)} \cap \KerH{\theta}{Z_{G}(A)}{G} .\]
%Let $r$ be a representative of $\alpha$ in $A$. Then $\alpha = [r] = [s^2r]$ for every $s \in A$, as $\theta|_A(x) = x^{-1}$
%so $\delta_x(r) = x^2r$ in $A$.
%We can choose $s \in A$ general enough such that $Z_{G}(s^2r) = Z_{G}(A)$.
%But then \[[r] = [s^2r] \in \HH^1(\theta, Z_{G}(A)) = \HH^1(\theta, Z_{G}(s^2r)).
%\]
%By Proposition \ref{proposition: the centralizer criterion} and the stability of $(G,H,\theta)$, we deduce that $\alpha$ is trivial.
%%\qed
%\end{proof}

\section{Open Orbits in Parabolic Quotients}
\label{section: open orbits in a parabolic quotient}

In this section we analyze the relation between the stability of a pair $(G,H,\theta)$ and another geometric property which we call p-stability, along with two other, less central for us, stability properties. The main notion in this context is the following.

\begin {definition}
\label{def: p-stable}
Let $(G,H,\theta)$ be a symmetric pair. We say that $(G,H,\theta)$ is \textbf{p-stable} if
$H$ has a single open orbit in each parabolic quotient $G/P$.
\end{definition}
It is clearly sufficient to check p-stability on the minimal parabolic subgroups.
\begin{remark}
The total number of $H$-orbits in $G/P$ is finite, see e.g. \cite[Corollary 6.16]{HW93}	
\end{remark}

This property is of interest for us since the decomposition of $G/P$ into $H$ orbits carries information on the relative representation theory of $H$ in $G$.
We shall return to the problem of linking p-stability and stability after recalling some machinery from the theory of involutions of
reductive groups.

\subsection {Preliminaries on Split Parabolic Subgroups and Split Tori}

We recall some facts and notation regarding parabolic subgroups and tori in symmetric pairs, most of which can be found in \cite{HW93}.
In general, by a parabolic subgroup of $G$ we mean the $F$-points of a parabolic subgroup of $\bd{G}$ defined over $F$. 

Recall the notion of a split torus over $F$. 
\begin {definition}
Let $\textbf{A}$ be an algebraic torus defined over $F$. Then $\textbf{A}$ is called $F$-split if the character lattice $X^*(\textbf{A})$ is a trivial $\Gal_{F}$ module.
\end {definition}
In this case, we also call $A=\bd{A}(F)$ an $F$-split torus. 

A reductive group $\textbf{G}$ defined over $F$ is called \textbf{$F$-split} if it has a maximal torus defined over $F$ which is $F$-split.
All the maximal $F$-split tori in $\textbf{G}$ are conjugate by $G$. Moreover, there exist maximal $F$-split tori which are $\theta$-stable,
and such tori exist inside every parabolic subgroup of $\bd{G}$ defined over $F$.

Let $\textbf{A} \subseteq \textbf{G}$ be a $\theta$-stable torus defined over $F$. We denote by $\textbf{A}^+, \textbf{A}^{-}$ its orthogonal and symmetric parts
respectively. Similarly, we denote by $A^+, A^-$ the orthogonal and symmetric parts of its $F$-rational points. Note that $A^+ \cap A^-$ is non-trivial, but it is always a 2-torsion finite group. 

\begin{definition}
A torus $\textbf{A}$ with an involution $\theta:\textbf{A}\to \textbf{A}$ is called $\theta$-split if $\textbf{A} = \textbf{A}^{-}$.
\end{definition}

Let $g \in G$. We would like to define what it means for $g$ to be "diagonalizable over $F$".

\begin {definition}
A semi-simple element $g \in G$ is called \textbf{$F$-split} if $g$ is contained in a maximal $F$-split torus of $G$.
\end{definition}

%\begin {lemma}
%Let $g \in G$ be semi-simple. Then $g$ is $F$-split if and only if $Z_G(g)$ contains a maximal $F$-split torus of $G$.
%\end{lemma}

%\begin{proof}
%There is a bijection between maximal tori containing $g$ and maximal tori of $Z_G(g)$.
%\qed
%\end{proof}

For us, the most important split elements are those contained in a maximal $(\theta, F)$-split torus.
Recall that a torus $A$ over $F$ is called $(\theta,F)$-split if $A$ is $F$-split and $\textbf{A}$ is $\theta$-split. 

\begin{definition}
A semi-simple element $g \in G$ is called $(\theta,F)$-split if $g$ is contained in a maximal $(\theta,F)$-split torus.
\end{definition}

Let $(G,H,\theta)$ be a symmetric pair.
In \cite{HW93} there is a detailed description of the $H$ orbits in parabolic quotients $G/P$ of $G$ by a minimal parabolic subgroup $P\subseteq G$, given in .
Since we are interested only in the open orbits, we recall only the part of the theory which is related to them.

\begin{definition}
Let $(G,H,\theta)$ be a symmetric pair over $F$.
Let $P$ be a parabolic subgroup of $G$. Then $P$ is called a $\theta$-\textbf{split} parabolic subgroup of $G$ if $P$ and $\theta(P)$ are opposite parabolic subgroups of $G$.
\end{definition}

\begin {lemma}
Let $P$ be a $\theta$-split parabolic subgroup of $G$. Then $HP$ is open in $G$.
\end {lemma}

\begin{proof}
It is sufficient to prove that $H \times P \to G$ is submersive at the identity, or that
$\frk{g} = \frk{h} + \frk{p}$. But as $P$ is $\theta$-split we have \[\frk{g} = \frk{p} + \theta(\frk{p}).\]
The result follows from the inclusion $\theta(\frk{p}) \subseteq \frk{p} + \frk{h}$.
\end{proof}

%\begin{remark}
%In fact, this is true also for $P$ a parabolic subgroup contained in a minimal $\theta$-split parabolic subgroup of $G$.
%This is essentially proved in \cite[Proposition 4.10]{HW93}.
%\end{remark}

We wish to classify minimal parabolic subgroups $P \subseteq G$ such that $HP$ is open.
In fact, every minimal parabolic subgroup of $G$ contains a $\theta$-stable maximal $F$-split torus (see \cite[Proposition 4.7]{HW93}),
unique up to conjugation by the unipotent radical $U_P$.
Fix such a torus $A$ and consider its symmetric part $A^- := A \cap G^\sigma$. Then
$\Phi(A,G)$ is a root system in $X^*(A)$ and
$\Phi(A^-,G)$ is a root system in $X^*(A^-)$ with Weyl group $N_G(A^-)/Z_G(A^-)$, by 
\cite[Proposition 5.9]{HW93}.   
We have a surjection 
\[\pi^- : \Phi(A,G) \to \Phi(A^-,G) \cup \{0\}\] given by restriction along the inclusion $A^-\hookrightarrow A$.

Since the involution $\theta$ stabilizes $A$, it acts naturally on 
$X^*(A)$ in a way which preserve the roots. Indeed, if 
$Ad_a(X) = \chi(a)X$ for $X \in \mathfrak{g}$, then 
\[Ad_{\theta(a)}(\theta(X)) = \theta(\chi(a)X) = \chi(a) \theta(X)\] 
so $\theta(X)$ is a root vector corresponding to the root $\theta(\chi)$. 
Every minimal parabolic subgroup $P \subseteq G$ containing $A$ defines a collection of positive roots $\Phi(A,P) \subseteq \Phi(A,G)$.
We can now state the classification of minimal parabolic subgroups $P$ for which $HP$ is open:

\begin{proposition}
\label{proposition: classification of generic parabolics}
Let $P$ be a minimal parabolic subgroup of $G$, and let $A$ be a $\theta$-stable maximal $F$-split torus of $G$ contained in $P$.
Then $PH$ is open in $G$ if and only if $(A^-)^0 = (A \cap G^\sigma)^0$ is a maximal $(\theta,F)$-split torus of $G$ and
$\pi^-(\Phi(A,P)) \backslash{0}$ is a collection of possitive roots for $\Phi(A^-,G)$
\end{proposition}

\begin{proof}
First, note that if $V$ is a linear space over $\RR$, $C \subseteq V$ is a convex polyhedral cone and $\pi : V \to U$ is a linear projection
then $\pi^-(C)$ is again a convex polyhedral cone.
Let $X^*(A)$ be the character lattice of $A$ and $C \subseteq X^*(A) \otimes \RR$ be the cone spanned by the roots $\Phi(A,P)$.
Since $\Phi(A,G) \subseteq C \cup -C$, we have
\[\Phi(A^-,G) \subseteq \pi^-(C) \cup \pi^-(-C).\]
Thus, since $\pi^-(C)$ is a polyhedral cone, it defines a collection of positive roots
in $\Phi(A^-,G)$ if and only if it does not contain antipodal roots.

We claim that the last condition is equivalent to the 
following one: for every $\alpha \in \Phi(A,P)$, 
either $\theta(\alpha) = \alpha$ or $\theta(\alpha)$ 
is negative. 

Indeed, via the identification 
\[X^*(A^-) \otimes \RR \cong (X^*(A) \otimes \RR)^{-} := 
\{v \in X^*(A) \otimes \RR : \theta(v) = -v\},\] 
the map $\pi^-$ corresponds to the map 
$v \mapsto (v - \theta(v))/2$. 
Thus, we have \[\pi^-(\Phi(A,P)) \backslash \{0\} = 
\{(\alpha - \theta(\alpha)) / 2: 
\alpha \in \Phi(A,P), \theta(\alpha) \ne 0\}.\] 

If $\alpha$ satisfies $\alpha \ne \theta(\alpha) > 0$, 
then $\alpha - \theta(\alpha)$ and $\theta(\alpha) - \alpha$ are two vectors in $\pi^-(C)$ and hence
$\pi^-(\alpha) \in \pi^-(C) \cap -\pi^-(C)$. 
Conversely, if the conditions on the roots are satisfied, 
then we claim that $\pi^-(C)$ is strictly convex.  
Indeed, $\pi^-(C)$ is spanned by the $\alpha_i - \theta(\alpha_i)$, which are, under the assumptions, 
contained in $C$. Since $C$ is strictly convex, it follows that the convex hall of the rays spanned by the 
$\alpha_i - \theta(\alpha_i)$ is convex as well in 
$(X^*(A) \otimes \RR)^{-}$, so that 
$\pi^-(C)$ is strictly convex.  

The condition that $HP$ is open in $G$ is equivalent to the infinitesimal condition
\[\mathfrak{h} + \mathfrak{p} = \mathfrak{g}.\]
Thus, we have to prove that $\mathfrak{h} + \mathfrak{p} = \mathfrak{g}$
if and only if $(A^-)^0$ is a maximal $(\theta,F)$-split torus of $G$, and
for every $\alpha \in \Phi(A,P)$, either
$\alpha = \theta(\alpha)$ or $\theta(\alpha)$ is a negative root.

Let $P = MU^+$  be the Levi decomposition of $P$ corresponding to $A$ for which
Then $M = Z_G(A)$ and $U^+$ is the unipotent radical of $P$.
Let $U^-$ be the unipotent radical of the opposite parabolic subgroup of $G$ containing $A$.
Let $<,>$ denote the Killing form of $\mathfrak{g}$.
With respect to $<,>$, we have
\begin{align*}
&\mathfrak{h}^\bot = \mathfrak{g}^\sigma \\
&\mathfrak{p}^\bot = \mathfrak{u}
\end{align*}
and therefore
\[(\mathfrak{h} + \mathfrak{p})^\bot = \mathfrak{u} \cap \mathfrak{g}^\sigma.\]
It follows that $HP$ is open in $G$ if and only if
$\mathfrak{u} \cap \mathfrak{g}^\sigma = 0$.

Assume first that $\mathfrak{u} \cap \mathfrak{g}^\sigma = 0$.
We have to prove that $A \cap G^\sigma$ is a maximal $(\theta,F)$-split torus and that the condition
on positive roots above is satisfied.

Let $Y \in \mathfrak{g}^\sigma$ be an element which commutes with $\mathfrak{a} \cap \mathfrak{g}^\sigma$.
Let $Y = Y_0 + \sum_{\alpha \in \Phi(A,G)} Y_\alpha$ be a decomposition of $Y$ into a sum of root vectors.
Since $Y$ commutes with $\mathfrak{a} \cap \mathfrak{g}^\sigma$,
$Y_\alpha = 0$ unless $\theta(\alpha) = \alpha$.
It follows that
\[Y = Y_0 + \sum_{\alpha \in \Phi(A,G)^\theta} Y_\alpha.\]
Since $Y \in \mathfrak{g}^\sigma$ we get that
\[Y_\alpha \in \mathfrak{g}_\alpha \cap \mathfrak{g}^\sigma \subseteq \mathfrak{u} \cap \mathfrak{g}^\sigma = 0.\]
It follows that $Y \in \mathfrak{m}$, but then
by the maximality of $\mathfrak{a}$ we must have
$Y \in \mathfrak{a} \cap \mathfrak{g}^\sigma$.
This proves the maximality of $A \cap G^\sigma$.

We next show that for every $\alpha \in \Phi(A,P)$ either $\theta(\alpha) = \alpha$ or
$\theta(\alpha) < 0$.
Indeed, if $\alpha \ne \theta(\alpha) > 0$ then
\[ 0 \ne (\mathfrak{g}_\alpha + \mathfrak{g}_{\theta(\alpha)})\cap \mathfrak{g}^\sigma \subseteq \mathfrak{u} \cap \mathfrak{g}^\sigma = 0\] so we get a contradiction.

Conversely, assume that $A^-$ is maximal and that no $\alpha \in \Phi(A,P)$ satisfies the conditions $\alpha \ne \theta(\alpha)  > 0$.
We have to prove that
$\mathfrak{u} \cap \mathfrak{g}^\sigma = 0$.

Clearly, $\mathfrak{u} \cap \mathfrak{g}^\sigma \subseteq \mathfrak{u} \cap \mathfrak{\theta(u)}$.
since $\mathfrak{u} \cap \mathfrak{\theta(u)}$ is spanned by the $\mathfrak{g}_\alpha$ for which
$\alpha$ and  $\theta(\alpha)$ are both positive,
by the assumption on the roots we deduce that
\[\mathfrak{u} \cap \mathfrak{g}^\sigma = \sum_{\alpha \in \Phi(A,P)^\theta} (\mathfrak{g}_\alpha \cap \mathfrak{g}^\sigma).\]
Thus, it is sufficient to prove that
$\mathfrak{g}_\alpha \cap \mathfrak{g}^\sigma = 0$ for every
$\alpha \in \Phi(A,P)^\theta$.

If $0 \ne Y \in \mathfrak{g}_\alpha \cap \mathfrak{g}^\sigma$ then as in
\cite[Lemma 7.1.11]{AG09}
we can extend $Y$ to an $\mathfrak{sl_2}$ triple
$(X,F,Y)$ such that $X \in \mathfrak{g}_{-\alpha} \cap \mathfrak{g}^\sigma$ and $F \in \mathfrak{h}$.
Since $\alpha = \theta(\alpha)$, $X$ and $Y$ commute with $A^-$.
But then $\mathfrak{a}^- + \Span\{X + Y\}$ is a $(\theta,F)$-split torus in $\mathfrak{g}$
porperly including $\mathfrak{a}^-$, contrary to the maximality of $\mathfrak{a}^-$.
\end{proof}

We finish with the following result of \cite{HW93}, which is useful for the classification of $H$-conjugacy classes of minimal $\theta$-split parabolic subgroups.
\begin{proposition} [{\cite[proposition 4.9]{HW93}}]
\label{proposition: conjugacy criterion for minimal theta-split parabolics}
Let $P,P'$ be minimal parabolic subgroups of $G$, and let $Q,Q'$ be minimal $\theta$-split parabolic subgroups containing $P,P'$ respectively.
If $g \in G$ satisfies $gPg^{-1} = P'$ then $gQg^{-1} = Q'$.
\end{proposition}

\subsection {Different Types of Stability}

There are some other notions, related to stability, that we shall consider.

\begin {definition}
Let $(G,H, \theta)$ be a symmetric pair. The pair is:
\begin {itemize}
 \item \textbf{s-stable} if every $g \in G$ such that $s(g)$ is $(\theta,F)$-split is stable.
 \item \textbf{p-stable} (see Definition \ref{def: p-stable})if $H$ has a unique open orbit in $G/P$ for every parabolic subgroup $P \subseteq G$.
 %\item \textbf{weakly p-stable} if and only if $H$ has a unique open orbit in $G/P$ where $P$ is a $\theta$-split parabolic.
 \item \textbf{t-stable} if all the $\theta$-stable maximal $F$-split tori in $G$ which contain a
        maximal $(\theta,F)$ split torus are conjugate by $H$.
 %\item \textbf{weakly t-stable} if and only if all the maximal $(\theta, F)$ split tori in $G$ are conjugate by $H$.
\end{itemize}

\end{definition}

All these properties are related. Even though not all of them are equivalent, we are able to prove the following scheme of implications:

%\vspace{0.5 cm}

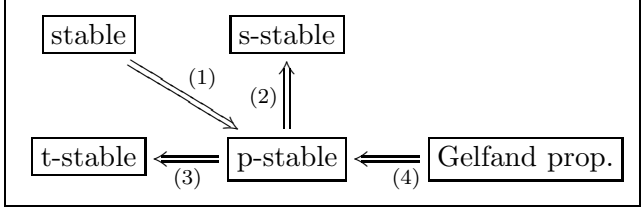
\fbox{
\label{diagram: implications}
 \xymatrix{
\fbox{stable} \ar@{=>}^{(1)}[rd] &\fbox{s-stable}                  \\
\fbox{t-stable}   &\fbox{p-stable} \ar@{=>}^{(2)}[u] \ar@{=>}^{(3)}[l] & \ar@{=>}^{(4)}[l] \fbox{Gelfand prop.}    \\
}
}

%\vspace{0.5 cm}

%\fbox{
% \xymatrix{
%  & \fbox{stable} \ar@{=>}^{\circled{1}}[r] &\fbox{s-stable}           &                  \\
%\fbox{t-stable} \ar@{<=>}^{\circled{2}}[dr] &       &     &\fbox{p-stable}       \\
% &\fbox{weakly t-stable}      &\fbox{weakly p-stable}\ar@{<=>}^{\circled{3}}[ur] \ar@{=>}^{\circled{4}}[uu] \ar@{=>}^{\circled{5}}[l] &  \\
%}
%}

\vspace{0.5 cm}

Implication (1) will be proved in the next section. Implication (4) will be proved in section 7.
Implications (2) and (3) will be proved here after some preparation.

\begin{proposition}[{\cite [Theorem 4.21] {BT65}}]
\label{proposition: BT tori}
Let $G$ be a reductive group over a local field $F$, then
all the maximal $F$-split tori in $G$ are conjugate in $G$.
\end{proposition}

\begin {proposition} [{\cite[Theorem 4.13] {BT65}}]
\label{proposition: BT minimal parabolics}
Let $G$ be a reductive group over a local field $F$. Then all the minimal parabolic subgroups of $G$ are conjugate.
\end {proposition}

\begin {proposition}
\label{proposition: t-stable iff weakly t-stable}
A pair $(G,H,\theta)$ is t-stable if and only if all the maximal $(\theta,F)$-split tori in $G$
are $H$-conjugate.
\end {proposition}

\begin {proof}
We first prove that t-stability implies the conjugacy of maximal $(\theta,F)$-split tori.
Let $(G,H,\theta)$ be a t-stable pair. Let $A, A'$ be two maximal $(\theta, F)$ split tori in $G$, we have to show that they are $H$-conjugate.
Let $T, T'$ be $\theta$-stable maximal $F$-split tori containing $A,A'$ respectively.
By the assumption, there is $h \in H$ such that $h T h^{-1} = T'.$
But then, as $h \in H$, we have \[h A h^{-1} = h T^{-} h^{-1} = (h T h^{-1})^{-} = T'^{-} = A'\]
and therefore $A,A'$ are $H$-conjugate.

Conversely, assume that all the maximal $(\theta,F)$-split tori are $H$-conjugate.
We shall show that $(G,H,\theta)$ is t-stable. Let $T, T'$ be $\theta$-stable maximal $F$-split tori in $G$,
with $T^-, T'^-$ maximal $(\theta, F)$-split. Let $A := T^-$ and $A' := T'^-$. By the assumption,
there is $h \in H$ such that $h A h^{-1} = A'$. Therefore, replacing $T$ by $hTh^{-1}$ we may assume that $A = A'$.
Set $K := Z_H(A)$, which, being a centralzer of a torus, is a reductive subgroup of $H$. Since $T^+$ and $T'^+$ are both maximal $F$-split tori in $K$, by Proposition 
\ref{proposition: BT tori}
there is $k \in K$ such that $k T^+ k^{-1} = T'^+$. But then $kTk^{-1} = T'$ and hence $T,T'$ are $H$-conjugate. %\qed
\end {proof}

We shall now discuss the p-stability of a pair. This require some preliminaries on $\theta$-split parabolic subgroups.

\begin {lemma}
\label{lemma: conjugacy of minimal split parabolics}
Let $(G,H,\theta)$ be a symmetric pair. The minimal $\theta$-split parabolic subgroups of $G$ are all conjugate by $G$.
\end{lemma}

\begin {proof}
Let $Q, Q'$ be minimal $\theta$-split parabolic subgroups of $G$, and let $P,P'$ be minimal parabolic F-subgroups of $G$ contained in $Q,Q'$ respectively.
By Proposition \ref{proposition: BT minimal parabolics} there is $g \in G$ such that $gPg^{-1} = P'$.
But then, by Proposition \ref{proposition: conjugacy criterion for minimal theta-split parabolics} we get that $gQg^{-1} = Q'$. %\qed
\end {proof}

\begin {lemma}
\label{lemma: uniqueness of theta-split torus}
Let $(G,H,\theta)$ be a symmetric pair and let $Q$ be a minimal $\theta$-split parabolic F-subgroup of $G$.
Then $Q$ contains a unique maximal $(\theta, F)$-split torus.
\end{lemma}

\begin{proof}
First note that by \cite[Proposition 4.7]{HW93} $Q$
contains a maximal $(\theta,F)$-split torus.
Let $A \subseteq Q$ be a maximal $(\theta, F)$-split torus. We have $Q \cap \theta(Q) = Z\low{G}(A)$ and since
$Q \cap \theta(Q)$ is reductive, its center is a torus. We then have by the maximality of $A$,
$A = Z(Q \cap \theta(Q))^-$, which depends only on $Q$. %\qed
\end {proof}

\begin{proposition}
\label{proposition: p-stable iff weakly p-stable}
A pair $(G,H,\theta)$ is p-stable if and only if all the minimal $\theta$-split parabolic subgroups of $G$ are conjugate by $H$.
\end{proposition}

\begin{proof}

First assume that the pair is p-stable. Since all the minimal $\theta$-split parabolic subgroups of $G$ are conjugate by $G$,
the variety of $\theta$-split parabolic subgroups is isomorphic as a $G$-variety to $G/Q$ for a minimal $\theta$-split parabolic
subgroup $Q \subseteq G$.
Let $P$ be a minimal parabolic subgroup of $G$ contained in $Q$. We have natural $H$-equivariant submersion
$\pi : G/P \to G/Q$. By the assumption, $H$ has a unique open orbit in $G/P$ and therefore it has an open dense
orbit $\mathcal{O} \subseteq G/P$. But then $\pi(\mathcal{O})$ is an open dense orbit of $H$ in $G/Q$.

We turn to the converse. Note that the p-stability can be checked only for the minimal parabolic F-subgroups,
as the projections $G/P \to G/P'$ are submersions for $P \subseteq P'$. Furthermore, as parabolic subgroups are self-normalizing,
the variety $G/P$ can be identified with the space of all conjugates of $P$.

Assume that all minimal $\theta$-split parabolic subgroups of $G$ are $H$-conjugate.
Let $P,P'$ be two minimal parabolic subgroups of $G$ such that $HP$ and $HP'$ are open on $G$.
By \cite[Proposition 9.2]{HW93}, there are minimal $\theta$-split parabolic F-subgroups $Q,Q'$ containing $P,P'$ respectively. 
By Lemma \ref{lemma: conjugacy of minimal split parabolics}, $Q,Q'$ are $G$-conjugate, hence both correspond to open $H$-orbits in $Ad_G(Q) \cong G/Q$.
By the assumption, there is $h \in H$ such that
$hQh^{-1} = Q'$. Replacing $P$ by $hPh^{-1}$, we may assume that $Q = Q'$.

Let $A, A'$ be $\theta$-stable maximal F-split tori of $G$ contained in $P, P'$ respectively. By Lemma \ref{lemma: uniqueness of theta-split torus},
$A^- = A'^-$. By the conjugacy of maximal split tori in $Z_{H}(A^-)$, there is $k \in Z\low{H}(A^-)$ such that $kAk^{-1} = A'$. Replacing $P$ by $kPk^{-1}$,
we may assume $A = A'$.

The problem is now reduced to a problem on the root system $\Phi(A, G)$. Let \[\pi^-: \Phi(A, G) \to \Phi(A^-, G)\] denote the natural projection.
Let $\Sigma, \Sigma'$ be the positive roots corresponding
to $P,P'$ in $X^*(A)$, respectively. Since both $HP$ and $HP'$ are open in $G$,
we have by Proposition \ref{proposition: classification of generic parabolics} that $A^-$ is a maximal $(\theta,F)$-split torus in $G$ and that $\pi^- (\Sigma) \backslash \{0\}$
and $\pi^- (\Sigma') \backslash \{0\}$ are
collections of positive roots in $\Phi(A^-,G)$. Clearly, both are the positive roots defined by $\Phi(A^-,Q)$, and in particular
they are the same choices of positive roots.

We claim that if $\alpha \in \Sigma$ and $\pi^-(\alpha) \ne 0$ then $\alpha \in \Sigma'$.
Indeed, let $\alpha \in \Sigma$. If $\alpha \ne \theta(\alpha)$
then $\pi^-(\alpha) \in \Phi(A^-,Q)$. But if $\alpha \notin \Sigma'$ then
$\alpha \in - \Sigma'$ which implies that $\pi^-(\alpha) \in - \Phi(A^-,Q)$,
a contradiction to the fact that $\Phi(A^-,Q)$ strictly contained in a half-plane.
By symmetry we have
\[\Sigma \cap (\pi^-)^{-1}(\Phi(A^-,Q)) = \Sigma' \cap (\pi^-)^{-1}(\Phi(A^-,Q)).\]
Let $\Sigma_H = \Sigma \cap Ker(\pi^-)$ and $\Sigma'_H = \Sigma \cap Ker(\pi^-)$.
We wish to prove that there exists an element $h \in N_H(A) \cap Z_H(A^-)$ such that
$h(\Sigma_H) = \Sigma'_{H}$. Then, we would get
$h(\Sigma) = \Sigma'$ since $h$ preserves the fibers of $\pi^-$.

Consider the projection
$\pi^+ : X^*(A) \otimes \RR \to X^*(A^+) \otimes \RR$. The map $\pi^+|_{Ker(\pi^-)}$ is an isomorphism,
so it will be sufficient of prove that we can find $h \in N_H(A)$ such that
$h(\pi^+(\Sigma_H)) = \pi^+(\Sigma'_H)$.

We first claim that $\pi^+(\Sigma_H) \subseteq \Phi(A^+,Z_H(A^-))$. Indeed,
if $\alpha \in \Sigma_H$ then $\mathfrak{g}_\alpha \subseteq Z_{\mathfrak{h}}(A^-)$
since $\mathfrak{u} \cap \mathfrak{g}^\sigma = 0$.

Next, we claim that $\pi^+$ maps $\Phi(A,Z_G(A^-))$ isomorphically into $\Phi(A^+,Z_H(A^-))$.
Indeed, $\Phi(A,Z_G(A^-)) = Ker(\pi^-)$ so $\pi^+$ is injective on $\Phi(A,Z_G(A^-))$.
Furthermore, if $\alpha \in \Phi(A^+,Z_H(A^-))$ and $Y$ is a root vector for $\alpha$ then
$Y$ is automatically a root vector for $A$ in $Z_G(A^-)$ corresponding to a root which lifts $\alpha$,
so $\pi^+$ is a surjection onto $\Phi(A^+,Z_H(A^-))$.

Finally, we claim that $\Sigma_H$ is a collection of positive roots in $\Phi(A,Z_G(A^-))$ and
therefore $\pi^+(\Sigma_H)$ is a collection of positive roots in $\Phi(A^+,Z_H(A^-))$.
To see this, note that $\Sigma_H$ is a subset of $\Phi(A,Z_G(A^-))$ which is
strictly contained in a half-space since $\Sigma$ is. Moreover,
\[\Phi(A,Z_G(A^-)) \subseteq (\Sigma \cup -\Sigma) \cap Ker(\pi^-) = \Sigma_H \cup -\Sigma_H.\]
This shows that $\Sigma_H$ is a collection of positive roots.
By a similar argument, $\pi^+(\Sigma'_H)$ is a collection of positive roots in $\Phi(A^+,Z_H(A^-))$.

Since $A^+$ is a maximal F-split torus in $Z_H(A^-)$,
$\Phi(A^+,Z_H(A^-))$ is a (maybe non-reduced) root system, and its Weyl group acts transitively on the Weyl chambers.
Since both $\Sigma_H$ and $\Sigma'_H$ are collections of positive roots in $\Phi(A^+,Z_H(A^-))$
there is a Weyl group element which maps one to the other.
In particular, there exists an $h \in N_H(A^+) \cap Z_H(A^-)$ such that $h(\Sigma_H) = \Sigma'_H$.
But then $h(\Sigma) = \Sigma'$ so $hPh^{-1} = P'$. It follows that the pair is p-stable.
\end {proof}

We now turn to the proof of the implications (2) and (3).

\begin{theorem} [Implication (3)]
Every p-stable pair is t-stable.
\end{theorem}

\begin{proof}

Let $(G,H,\theta)$ be a p-stable pair. Let $A, A'$ be two maximal $(\theta, F)$-split tori in $G$.
 Let $Q,Q'$ be minimal $\theta$-split parabolic F-subgroups of
$G$ containing $A,A'$ respectively. As $(G, H, \theta)$ is p-stable and by Proposition \ref{proposition: p-stable iff weakly p-stable},
there is $h \in H$ such that $hQh^{-1} = Q'$.
By Lemma \ref{lemma: uniqueness of theta-split torus}, $hAh^{-1} = A'$.
Since by Proposition \ref{proposition: t-stable iff weakly t-stable} the $H$-conjugacy of all the maximal $(\theta,F)$-split tori
is equivalent to $t$-stability, the pair is $t$-stable.
%\qed

\end{proof}

\begin{theorem} [Implication (2)]
Every p-stable pair is s-stable.
\end{theorem}

\begin{proof}

Assume that $(G,H,\theta)$ is p-stable. We have to prove that it is s-stable.
Let $r \in G^\sigma_0$ be a $(\theta, F)$-split element which is the symmetrization of an element $g$.
Let $A^{-}$ be the $F$-points of the identity component of the torus topologically generated by $r$. In particular, $A^-$ is $(\theta,F)$-split and  $Z_G(r) = Z_G(A^{-})$.
Let $P$ be a parabolic subgroup of $G$ corresponding to a choice of positive roots in $\Phi(A^{-}, G)$.
 $G/P$ is isomorphic as a $G$-variety to the variety of all conjugates of $P$ since $P$ is self-normalizing.

Let $P' = g^{-1} P g$.
We claim that both $P$ and $P'$ has open $H$-orbits in $G/P$.
Indeed, by construction $P$ is $\theta$-split, as $P \cap \theta(P) = Z_G(A^{-})$ is a Levi-subgroup.
But
\begin {align*}
\theta(P') = \sigma(g) \theta(P) \theta(g) = && (g \sigma(g) = r) \\
g^{-1} r \theta(P) r^{-1} g  = && (r \in P \cap \theta(P)) \\
= g^{-1} \theta(P) g.&&\\
\end {align*}
Therefore, $\theta(P') \cap P' = g^{-1} (P \cap \theta(P)) g$ is a Levi subgroup of $G$, so $P'$ and $\theta(P')$ are opposite.
We deduce that $HP'$ is open in $G$, so $P'$ has open $H$-conjugacy class, and in fact is $\theta$-split parabolic subgroup of $G$.

Since $(G,H,\theta)$ is p-stable and by Proposition \ref{proposition: p-stable iff weakly p-stable},
there is $h \in H$ such that $h P h^{-1} = P' = g^{-1} P g.$
It follows that $gh$ normalizes $P$. Since $P$ is self-normalizing, $gh \in P$.
Since $r = s(g) = s(gh)$ we may assume that $g \in P$.

Consider the semi-direct decomposition $P = M_P \cdot U_P$ where $M_P = Z_G(A^{-})$ is the Levi part of $P$ and $U_P$ is the unipotent radical.
We can decompose $g$ as $g = u m$ where $u \in U_P$ and $m \in M_P$ (as the product is semi direct the factors can be interchanged!).
We get \[r = s(g) = g \sigma(g) = u(m \sigma(m)) \sigma(u).\]
Both the left hand side and the right hand side are Bruhat decompositions of $r$, so by the uniqueness of Bruhat decomposition we must have $m \sigma (m) = r$.
But $m \in Z_G(A^{-}) = Z_G(r)$ so $r$ is a symmetrization of an element of its centralizer. By Proposition \ref{proposition: symmetrization in centralizer}
$r$ is stable.
\end{proof}

In some cases, we can prove that t-stability and s-stability together imply the p-stability of a pair. 
This however require extra assumptions on $(G,H,\theta)$.
Before we prove this, we shall state and prove a lemma about generic elements in split tori, which will be used several times in the rest of this paper. 

\begin{lemma}
	\label{lemma: generic representative in split torus}
	Let $F$ be a local field of characteristic $0$ and let $A\subseteq G$ be a split torus in a reductive group $G$ over $F$. 
	For every $r\in Z_G(A)$ and every natural number $k$ there exists $s\in A$ such that $Z_G(rs^k)\subseteq Z_G(A)$.  
\end{lemma}

\begin{proof}
	Let $r_s$ denote the semi-simple factor of $r$. By Jordan decomposition, since $r$ and $s$ commutes and $s$ is semi-simple, we have $Z_G(rs^k)\subseteq Z_G(r_ss^k)$ and hence we may assume without loss of generality that $r$ is semi-simple. 
	
	Let $\mathbf{T}$ be the Zariski closure of the subgroup of $\mathbf{G}$ generated by $\mathbf{A}$ and $r$. Then $\mathbf{T}$ is a torus defined over $F$. Let $T=\mathbf{T}(F)$, so that $A\subseteq T$. Moreover, $\mathbf{T}/\mathbf{T}^0$ is finite. Let $n=|\mathbf{T}/\mathbf{T}^0|$ and set $T^0=\mathbf{T}^0(F)$. Then, $\mathbf{A}^{kn} \subseteq \mathbf{T}^0$ and $r^n\in \mathbf{T}^0$ and moreover they together generate $\mathbf{T}^0$ topologically. Since $Z_G(r^ns^{kn})\supseteq Z_G(rs^k)$ we may replace $k$ by $kn$ and $r$ by $r^n$ and assume that $\mathbf{T}$ is a connected torus. Let \[X^*(\mathbf{T},\mathbf{A})=\{\psi \in X^*(\mathbf{T}): \psi|_\mathbf{A}\ne 1\},\] and let $U=\{rs^k:s\in A\}$. Then $U$ is a (non-algebraic) open Zariski-dense subset of $rA$. For every $\psi\in X^*(\mathbf{T},\mathbf{A})$ we have by construction that $\psi|_U$ is not identically $1$, and hence $U_\psi = \{x\in U : \psi(x)=1\}$ is closed with empty interior (note that $\psi|_{U}$ take values in a finite extension of $F$, which is itself a local field). Since $U$ is an open subset of $rA$, which is the $F$-points of an affine veriety over $F$, it is a Baire space, and since $X^*(\mathbf{T},\mathbf{A})$ is countable there is $x\in U$ which is not in any of the $U_\psi$-s. Fix such $x$, and let $\mathbf{T}'\subseteq \mathbf{T}$ be the Zariski closure of the subgroup generated by $x$. Since $\psi(x)\ne 1$ for every $\psi\in X^*(\mathbf{T})$ for which $\psi|_\mathbf{A}\ne 1$, we deduce that $\mathbf{A}\subseteq \mathbf{T}'$. It follows that $Z_G(x)\subseteq Z_G(A)$.       
\end{proof}

\begin {theorem}
Let $(G,H,\theta)$ be a symmetric pair, and assume that $G$ has a maximal torus which is $(\theta,F)$-split. 
If $(G,H,\theta)$ is t-stable and s-stable, then in is also p-stable.
\end {theorem}

\begin{proof}
First note that in this case the minimal $\theta$-split parabolic subgroups in this case are precisely the
$\theta$-split Borel sub-groups of $G$.
Thus,by Proposition \ref{proposition: p-stable iff weakly p-stable} it will suffice to prove that
every two $\theta$-split Borel subgroups of $G$ are $H$-conjugate.
Let $B$,$B'$ be $\theta$-split Borel subgroups.
Let $A,A'$ be $\theta$-stable maximal F-split tori of $G$ contained in $B,B'$ respectively.
Since the pair is t-stable we may assume that $A = A'$.

By the transitivity of the action of the Weyl group on the Weyl chambers,
there exists $w \in W_G(A)$ such that $w(\Phi(A,B)) = \Phi(A,B')$.
But then since both $B$ and $B'$ are $\theta$-split,
\[-\theta(w)(\Phi(A,B)) = \theta(w)(\theta(\Phi(A,B))) =\theta(w(\Phi(A,B))) = \theta(\Phi(A,B')) = -\Phi(A,B')\]
so
\[\theta(w)(\Phi(A,B)) =  \Phi(A,B') = w(\Phi(A,B))\]
and since the action of the Weyl group on the Weyl chambers is free we deduce that
$\theta(\omega) = \omega.$
Let $g \in N_G(A)$ be a representative of $w$. We have $r := s(g) \in Z_G(A) = A$
since $G$ is split and connected.
It follows that $s(g) \in A \cap G^\sigma = A^-$.

For every $y \in A^-$ we have  \[s(yg) = yg\theta(g)^{-1}y = yry = ry^2.\]
By Lemma \ref{lemma: generic representative in split torus} we can choose $y \in A^-$ for which
$Z_G(ry^2) = Z_G(A^-)$.
Replacing $g$ with $yg$ we may assume that
$Z_G(r) = Z_G(A^-)$.

Since the pair is s-stable and $r$ is $(\theta,F)$-split and a symmetrization of an element in $G$, 
we have $r = s(z)$ for $z \in Z_G(r) = Z_G(A^-)$.
The equality $s(g) = s(z)$ implies that $z = gh$ for some $h \in H$.
Replacing $B$ with $h^{-1}Bh$ we may assume that $g = z$,
and hence that $zBz^{-1} = B'$.

Since $\Phi(A,B)$ and $\Phi(A,B')$ are
conjugate by $Z_G(A^-)$, which preserves the fibers of the map
$\pi^-:\Phi(A,G) \to \Phi(A^-,G)$, we get that
that 
\[\Phi(A^-,B) = \Phi(A^-,B').\]
Since $B$ and $B'$ are $(\theta,F)$-split Borel subgroups containing the torus $A$, we get the equalities 
$\Phi(A,B) = (\pi^-)^{-1}(\Phi(A^-,B))$ and
$\Phi(A,B') = (\pi^-)^{-1}(\Phi(A^-,B')).$
Since
$\Phi(A^-,B) = \Phi(A^-,B')$, we deduce that
$\Phi(A,B) = \Phi(A,B')$. Finally, this implies that $B = B'$ and hence that the pair is $p$-stable.
\end{proof}

\subsection {Verification Methods for s,t, and p-Stability}

In this section we describe our method to verify the various stability properties. In the next section we will use this method to some specific cases.

We have the first and second obstructions, which in principle allows us to check stability.
In general this is done by considering the different subgroups (maybe non-split ones), which form the stabilizers
of semi-simple symmetric elements and computing their cohomologies. 
We are ready to state and prove the criterion of $s$-stability of symmetric pairs. 

\begin {theorem}
\label{theorem: criterion for s-stability}
A symmetric pair $(G,H,\theta)$ is s-stable if and only if for every maximal $(\theta,F)$-split torus $A\subseteq G$ ,
\begin{equation}
\ImH{\theta}{A}{Z\low{G}(A)} \cap \KerH{\theta}{Z\low{G}}{G} = 1.
\end{equation}
\end{theorem}

\begin {proof}
In one direction, suppose that the pair is s-stable. If
\[r \in \ImH{\theta}{A}{Z\low{G}(A)} \cap \KerH{\theta}{Z\low{G}}{G},\]
then since $[r]$ is trivial in $\HH^1(\theta,G)$, $r$ is a symmetrization in $G$ and the stability of the pair implies that $[r]$ is the trivial
cocycle of $Z_G(r)$.
Replacing $r$ by $rs^2$ for $s \in A$ has no effect on the cohomology class of $r$ in $Z_G(A)$. By Lemma \ref{lemma: generic representative in split torus} we can choose
$s$ such that $Z\low{G}(s^2r) = Z\low{G}(A)$. It follows that $[r]$ is the trivial cohomology class class already in $\HH^1(\theta, Z_G(A))$, and we got that
\[\ImH{\theta}{A}{Z\low{G}(A)} \cap \KerH{\theta}{Z\low{G}}{G} = 1.\]

The other direction is easier. Let $r$ be a $(\theta,F)$-split element and let $A$ be a maximal $(\theta,F)$-split torus containing $r$.
By the assumption $[r]$ is trivial in $\HH^1(\theta,Z_G(A))$, and since 
$Z_G(A) \subseteq Z_G(r)$ we get that $[r]$ is trivial in $\HH^1(\theta, Z_G(r))$.
%\qed
\end{proof}

We shall now discuss t-stability. In order to give a cohomological criterion for it, we
need a fact about the conjugacy of $(\theta, F)$-split tori.

\begin {proposition}[{\cite[Lemma 10.3]{HW93}}]
Let $(G, H, \theta)$ be a symmetric pair. All the $(\theta, F)$-split tori in $G$ are conjugate by elements of $G$.
\end{proposition}

\begin {theorem}
\label{theorem: conjugacy of split tori}
Let $(G,H, \theta)$ be a symmetric pair, and let $A$ be a maximal $(\theta, F)$-split torus in $G$.
The maximal $(\theta, F)$-split tori in $G$ are classified up to $H$-conjugacy be the set
\[\ImH{\theta}{Z_G(A)}{N_G(A)} \cap \KerH{\theta}{N\low{G}(A)}{G}.\]
In particular, the pair is $t$-stable if and only if
\begin{equation}
\label{equation: eqn criterion for t-stability}
\ImH{\theta}{Z_G(A)}{N_G(A)} \cap \KerH{\theta}{N\low{G}(A)}{G} = 1.
\end{equation}
\end{theorem}

\begin {proof}
Set $X = \ImH{\theta}{Z_G(A)}{N_G(A)} \cap \KerH{\theta}{N\low{G}(A)}{G}$.
Let $\mathcal{T}$ denote the set tori $A \subseteq G$ which conjugate to a maximal $(\theta,F)$-split torus in $G$.
Let $A$ be any maximal $(\theta,F)$-split torus in $G$ and take it to be the base point of $\mathcal{T}$.
By Theorem \ref{theorem: conjugacy of split tori}, $G$ acts transitively on $\mathcal{T}$ by conjugation.
We have a natural action of $\theta$ on $\mathcal{T}$ via $A \mapsto \theta(A)$.
By Theorem \ref{proposition: descent lemma} we have

\[\mathcal{T}^\theta / Ad_H \cong \KerH{\theta}{ N_G(A)}{G}\]
Under this identification, the torus
$A' = gAg^{-1}$ corresponds to the cocycle $g^{-1}\theta(g) \in Z^1(\theta,N_G(A))$.
We wish to determine those $g$ for which $gAg^{-1}$ is maximal $(\theta,F)$-split.
This is the case exactly when
$\theta(gag^{-1}) =ga^{-1}g^{-1}$ for every $a \in A$.
But $\theta(a) = a^{-1}$ for $a \in A$ so in this case
$\theta(g) a^{-1} \theta(g)^{-1} = g a^{-1} g^{-1}$
which means that $g^{-1} \theta(g)$ commutes with $a$ for every $a \in A$.
Thus, the cocycles corresponding with maximal $(\theta,F)$-split tori
are exactly these which live in $\ImH{\theta}{Z_G(A)}{N_H(A)}$.
\end{proof}

Finally, we treat the p-stability.

\begin{theorem}
\label{theorem: criterion for p-stability}
The $H$-conjugacy classes of minimal $\theta$-split parabolic subgroups of $G$ are
in 1-1 correspondence with the set
\[\KerH{\theta}{Z_G(A)}{G}\]
for a maximal $(\theta,F)$-split torus.
In particular, the pair $(G,H,\theta)$ is p-stable if and only if
\[\KerH{\theta}{Z_G(A)}{G} = 1.\]
\end{theorem}

\begin{proof}
Let $\mathcal{P}$ denote the set of minimal $\theta$-split parabolic subgroups of $G$.
Let $P$ be a minimal $\theta$-split parabolic subgroup of $G$.
Consider $G/P$ as the variety of conjugates of $P$.
Let $\theta$ acts on $G/P \times G/P$ by
$\tilde{\theta}(P',P'') = (\theta(P''),\theta(P'))$.
Let $\Delta_\theta = (G/P \times G/P)^\theta$.
Let $\mathcal{O} \subseteq G/P \times G/P$ denote the unique open $G$-orbit which consists of
pairs $(P',P'')$ of opposite parabolic subgroups.

By Lemma \ref{lemma: conjugacy of minimal split parabolics},
the minimal $\theta$-split parabolic subgroups are all conjugates
$P$. Let $\pi_1 : G/P \times G/P \to G/P$ be the projection onto the first factor.
We deduce that $\mathcal{P} \cong \pi_1(\Delta_\theta \cap \mathcal{O})$.
Since $\pi_1|_{\Delta_\theta}$ is injective and $H$-equivariant,
we have
\[\mathcal{P}/Ad_H \cong (\Delta_\theta \cap \mathcal{O}) /H \cong \mathcal{O}^\theta/H.\]

By Theorem \ref{proposition: descent lemma} we get
\[\mathcal{P}/Ad_H \cong \mathcal{O}^{\tilde{\theta}}/H \cong \KerH{\theta}{\Stb_{G}(P,\theta(P))}{G}.\]
But
\[\Stb_G((P,\theta(P))) = N_G(P) \cap N_G(\theta(P)) = P \cap \theta(P) = Z_G(A^-)\]
where $A^-$ is a maximal $(\theta,F)$-split torus contained in $P$, by
\cite[Lemma 4.6]{HW93}.
We deduce that
\[\mathcal{P}/H \cong \KerH{\theta}{Z_G(A)}{G}.\]
\end{proof}

At this point, we are ready to prove the remaining implication between stabilities, namely that stability implies p-stability.
Before we prove it, we need a lemma:

\begin{lemma}
\label{lemma: semi-simple rep}
Let $(G,H,\theta)$ be a symmetric pair, and let $x \in \HH^1(\theta,G)$. Then $x$ can be represented by a semi-simple element.
\end{lemma}

\begin{proof}
Let $r$ be a representative of $x$, and let $r = r_s r_u$ be a Jordan decomposition of $r$.
Let $\mathbf{U}$ be the Zariski closure of the subgroup of $G$ generated by $r_u$ and $U = \mathbf{U}(F)$. Then $U$ is an $F$-subgroup of $G$ containing $r_u$ which centralizes $r_s$.
Then, by \cite[Lemma  0.6]{HW93} we have $r_u = y^{-1}\theta(y)$ for some $y \in U$.
But then, as $y$ commutes with $r_s$ we get 
\[x = [r] = [r_s r_u] = [y^{-1}r_s\theta(y)]  = [r_s]\] in $\HH^1(\theta,G)$.
%\qed
\end{proof}

\begin{theorem} [{Implication (1)}]
\label{theorem: stable then p-stable}
A stable symmetric pair is p-stable.
\end{theorem}

\begin{proof}

Let $(G,H,\theta)$ be a stable pair.
By Theorem \ref{theorem: criterion for p-stability}),
it suffices to prove that,
for a maximal $(\theta,F)$-split torus $A\subseteq G$, the map $\HH^1(\theta,Z_{G}(A)) \to \HH^1(\theta, G)$ has trivial kernel.

Let $[r] \in \KerH{\theta}{Z_G(A)}{G}$. 
Choose a representative $r\in Z_G(A)^\sigma$. By Lemma \ref{lemma: generic representative in split torus}, we may find $s\in A$ such that $Z_G(rs^2)\subseteq Z_G(A)$. But then, since $(G,H,\theta)$ is stable, $[rs^2]$ is trivial in $\HH^1(\theta,Z_G(rs^2))$ and hence also in $\HH^1(\theta,Z_G(A))$. It follows that the pair is $p$-stable.  
\end{proof}

Empirically, it seems that in the Archimedean case, at least if $G$ is connected, the converse also holds.
\begin{conjecture}
If $(G,H,\theta)$ is a p-stable symmetric pair over $\RR$ and  $G$ is connected in the real topology,
then $(G,H,\theta)$ is stable.
\end{conjecture}

\begin{remark}
We hope that some special features of the cohomology of real pairs might help,
such as the results in \cite{Ad13}.
\end{remark}

We now illustrate the usage of the cohomological method by an example.
\begin {example}
\label{example: example 1}
Not every p-stable pair is stable. For example, let $p \equiv 3 \mod 4$
and consider the pair $(G,H,\theta)$ where $G$ is the quaternions
of norm 1 in the quaternion algebra 
\[\QQ_p[i,j]/(\{i^2 = p, j^2 = -1, ij = -ji\}) := \HH[p, -1],\] and \[\theta(x) = ixi^{-1}= \frac{ixi}{p}.\]
This pair is clearly p, s and t-stable because
there are no parabolic subgroups or split tori or split elements at all.
However, the pair is not stable as we shall prove in the next theorem.
\end{example}

\begin{theorem}
The pair $(G,H,\theta)$ in Example \ref{example: example 1} is unstable for every $p \equiv 3 \mod 4$.
\end{theorem}

\begin{proof}
Consider the maximal $\theta$-split torus $A = \{x \in G : x = a + bk\}$ where $k := ij$.
We claim that $\KerH{\theta}{A}{G} \ne 1$, and therefore by Proposition \ref{proposition: the centralizer criterion} the pair is not stable,
since it is straight forward to check that $Z_G(A) = A$.

Let $\HH = \HH[p, -1]$ denote the quaternion algebra, then we have an exact sequence
\[\SES{G}{\HH^\times}{\QQ_p^\times},\] where the right arrow is the reduced norm \[N_{\HH/\QQ_p}(x + iy + jz + kw) := x^2 + y^2 - p(z^2 + w^2).\]
Moreover, we have an exact sequence
\[\SES{A}{\QQ_p[k]^\times}{B}\] where $B$ is the subgroup of $\QQ_p^\times$ generated by the norms from $\QQ_p[k]$. These are the elements with the least significant digit a square mod $p$.
The two short exact sequences fit into a commutative diagram of sequences:
\[\xymatrix{
1 \ar[r]& A \ar[r] \ar[d] & \QQ_p[k]^\times \ar[r] \ar[d] & B \ar[d] \ar[r] & 1 \\
1\ar[r] & G \ar[r] & \HH^\times \ar[r] & \QQ_p^\times \ar[r] & 1 \\
}\]

Note that $(\HH^\times)^\theta = \QQ_p[i]^\times$.
Using the isomorphism in Theorem \ref{proposition: descent lemma} for both horizontal lines, where we let the middle groups act on the groups on the right, we obtain a commutative square with injective horizontal arrows:

\[\xymatrix{
 B/(\QQ_p^\times)^2 \ar@{^{(}->}[r] \ar[d] & \HH^1(\theta, A) \ar[d]\\
\QQ_p^\times / N_{\HH/\QQ_p}(\QQ_p[i]^\times) \ar@{^{(}->}[r] & \HH^1(\theta, G) \\
}\]

Namely, here the active groups are $\QQ_p[k]^\times$ and $\HH^\times$ and they act on
$B$ and $\QQ_p^\times$ respectively.

The upper horizontal map is an isomorphism, since by Hilbert's 90 Theorem
\[\HH^1(\theta, \QQ_p[k]^\times) = 1.\]

It is therefore sufficient to show that
$B/(\QQ_p^\times)^2$ is non-trivial and the map
\[B/(\QQ_p^\times)^2 \to \QQ_p^\times / N(\QQ_p[i]^\times))\] is trivial.
Indeed, by the diagram above and the fact that the upper arrow is an isomorphism,
$\KerH{\theta}{A}{G}$ is isomorphic to the kernel of the map
$B/(\QQ_p^\times)^2 \to \QQ_p^\times / N(\QQ_p[i])^\times$ which appears in the upper part of the diagram.

But it is immediate that
\[B/(\QQ_p^\times)^2 \cong \ZZ / 2\ZZ\]
and that the left hand side is generated by $-p$. Since $-p = N(i) \in N(\QQ_p[i]^\times)$, the mapping
\[B/(\QQ_p^\times)^2 \to \QQ_p^\times / N(\QQ_p[i]^\times))\] is trivial,
and the pair is unstable.
%\qed
\end{proof}

\section{Calculations}
\label{section: calculations}

In this section we always assume that $F = \RR$ or $F$ is non-Archimedean. 
We shall apply the methods developed in the previous sections to classify the stable pairs among several familes of classical symmetric pairs. For the pairs we consider, we check s-stability, p-stability and stability. We do not mention t-stability as it is unrelated to the Gelfand property to the extent of our knowledge. 
Note, however, that the p-stable pairs are automatically t-stable.

During the calculation, we will use extensively the various  cohomological criteria for the different stability properties, sometimes implicitly.

The results of this sections are summarized in the following table, while the rest of this 
section is dedicated to proving the validity of results presented in it. 
In the columns stable (resp. p-st., s-st.)
we list necessary and sufficient conditions for stability (resp. p-stability, s-stability.).
The columns labeled $\RR$ are for the real pairs while the columns labeled
"non-Archomedean" are for the non-Archimedean pairs.
The entries labeled $(X)$ and $(Y)$ are described below the table.
\newline
\newline

\begin{table}[h] 	
	\begin{tabular}
		{ |l|l l l|l|l|l| }
		\hline
		&\multicolumn{3}{|c|}{$\mathbb{R}$}&\multicolumn{3}{|c|}{non-Archimedean}\\
		\hline
		Pair & Stable \vline& s-st. \vline & p-st. &  Stable & s-st. & p-st. \\
		\hline
		{\scriptsize $SL(V),(GL(V^+)\times GL(V^-))\cap SL(V)$} & \multicolumn{6}{|c|}{$\dim(V^+) \neq \dim(V^-)$} \\
		\hline
		{\footnotesize $GL_F(V),GL_E(V)$} & \multicolumn{6}{|c|}{Always} \\
		\hline
		{\footnotesize $SL_F(V),SL_E(V)$} & \multicolumn{3}{|c|}{Always} &  \multicolumn{3}{|c|}{dim V = 1} \\
		\hline
		{\footnotesize $SL(V(E)),SL(V(F))$} & \multicolumn{6}{|c|}{$\dim(V)=2k+1$} \\
		\hline
		{\footnotesize $O(B),O(B^+)\times O(B^-)$} & \multicolumn{3}{|c|}{$B^+$ or $B^-$ is definite} &\multicolumn{3}{|c|} {\text{not inluded}}  \\
		\hline
		{\footnotesize $U(B),U(B^+)\times U(B^-)$} & \multicolumn{3}{|c|}{$B^+$ or $B^-$ is definite} & (X)  & (Y) & (X) \\
		\hline
		{\footnotesize $GL(V),O(B)$} & \multicolumn{3}{|c|}{$B$ is definite} & \multicolumn{3}{|c|}{$\dim(V)=1$} \\
		\hline
		{\footnotesize $GL(V),U(B)$} & \multicolumn{3}{|c|}{$B$ is definite} & \multicolumn{3}{|c|}{$\dim(V)=1$} \\
		
		\hline
		{\footnotesize $Sp_E(\omega),Sp_F(\omega)$} & \multicolumn{6}{|c|}{Never} \\
		
		\hline
		{\footnotesize $GSp_E(\omega),GSp_F(\omega)$} & \multicolumn{6}{|c|}{$\dim(V) = 2$} \\
		\hline
		{\footnotesize $Sp(\omega), GL(L)$} & \multicolumn{6}{|c|}{Never} \\
		\hline
		
	\end{tabular}
	\caption{List of Stable Pairs}
	\label{table:stable pairs}
\end{table}

For a quadratic form $B$ we denote by $\mu(B)$ the maximal dimension of a subspace on which $B$ vanishes.

\underline{(X)} : $min\{\dim(V^+), \dim(V^-)\} \le 1$.

\underline{(Y)} : $min\{\dim(V^+), \dim(V^-), \mu(B)\} \le 1.$

The pairs $(GL(V),GL(V^+) \times GL(V^-))$ and $GL(V(E)), GL(V(F))$ are known to be stable 
(\cite[Proposition 2.8.5 and Proposition 7.7.4]{AG09}) so the claculation for those pairs is 
omitted.

\subsection {Preliminaries}

In this section we describe some very classical results for classical 
groups, to be used in the calculations of obstructions to various stability 
conditions. 
None of the results presented here is new, however, some of the 
terminology and notation will be different from the classical ones, 
due to our special usage.

\subsubsection {Linear Algebra and Eigenvalues}

For each point $\lambda \in spec(F[x])$ there is a unique monic irreducible polynomial
$m_\lambda(x)$ generating the maximal ideal associated with the point $\lambda$.

Let $V$ be a linear space over a field $F$. Let $r$ be a semi-simple automorphism of $V$.
There is a decomposition \[V = \oplus_{\lambda \in spec(F[x])} V_\lambda(r)\]
where $V_\lambda(r) = \{v \in V : m_\lambda(r)v = 0\}$. We think of $\lambda$ as a point of the algebraic closure
$\bar{F}$, well defined only up to the action of $\Gamma_{\bar{F}/F}$.
The space $V_\lambda(r)$ will be referred to as primary space with primary-value $\lambda$.
The decomposition $V = \oplus_{\lambda \in spec(F[x])} V_\lambda(r)$ will be referred to as the primary decomposition of
$V$ with respect to $r$.

For $\lambda \in spec(F[x])$, let $F_\lambda \cong F[x] / (m_{\lambda}(x))$ denote the residue field of $spec{F[x]}$ at $\lambda$.
If we consider $\lambda$ as an element of $\bar{F}$ then this field is just $F[\lambda]$.
The spaces $V_\lambda(r)$ come endowed with a canonical structure of an $F_\lambda$ vector space,
given by $p(x) v := p(r) v$ for $p$ a polynomial over $F$, considered as an element of $F_\lambda$.
Under this identification $r$ is given by multiplication by $\lambda$ when we think of $\lambda$ as an element of $\bar{F}$.

If $h \in GL(V)$ is an element such that $hr = r^{-1} h$, then $h|_{V_\lambda}$ gives rise to an isomorphism
$V_\lambda \to V_{\lambda^{-1}}$. We have 3 options for the interrelation of $h$ and the primary subspaces:

\begin{itemize}
\item $\lambda$ and $\lambda^{-1}$ are different points of $spec(F[x])$. In this case
$h|_{V_\lambda(r)} : V_\lambda(r) \stackrel{\approx}{\rightarrow} V_{\lambda^{-1}(r)}$.
The set of primary values of this type will be called primary values of type $A$.

\item $\lambda$ and $\lambda^{-1}$ are different as elements of $\bar{F}$, but correspond to the same point of $spec(F[x])$. In this case
$V_\lambda = V_{\lambda^{-1}}$ and $h$ is an $F$-linear automorphism of it. We can consider $Ad_h$ as
an automorphism of $F_\lambda \cong F[r]$ (with the action of conjugating $r$), and then $h$ is a semi-linear automorphism of
$V_\lambda(r)$ over $F_\lambda$,
corresponding to the non-trivial automorphism of $F[\lambda]$ over $F[\lambda + \lambda^{-1}]$.
The set of primary values of this type will be called primary values of type $B$.

\item $\lambda = \pm 1$. In this case $h|_{V_\lambda(r)}$ commutes with $r|_{V_\lambda(r)}$.
The set of primary values of this type will be called primary values of type $C$.
\end{itemize}

We shall denote by $\mathcal{A}(r), \mathcal{B}(r), \mathcal{C}(r)$ the sets on primary values of type $A,B,C$ of $r$
respectively.
The set $\mathcal{A}(r)$ consists of pairs of the form $(\lambda,\lambda^{-1})$. We shall denote by
$\mathcal{A}^+(r)$ any choice of an element from each pair, and by $U_\lambda(r)$ the linear space
$V_\lambda(r) + V_{\lambda^{-1}}(r)$.

\begin{lemma}
\label{lemma: dimension formula}
Let $r \in GL(V)$ and let $h \in GL(V)$ such that $h^2 = Id_V$ and $hrh^{-1} = r^{-1}$.
Then

\[
\dim(V_1(h))
=\frac{1}{2} \left(\sum_{\lambda \in \mathcal{A}(r) \cup \mathcal{B}(r)} \dim(V_\lambda(r)) \right) +
\dim(V_1(r) \cap V_1(h)) + \dim(V_{-1}(r) \cap V_{1}(h))
\]

and

\[
\dim(V_1(rh))
=\frac{1}{2} \left(\sum_{\lambda \in \mathcal{A}(r) \cup \mathcal{B}(r)}\dim(V_\lambda(r)) \right) +
\dim(V_1(r) \cap V_1(h)) + \dim(V_{-1}(r) \cap V_{-1}(h))
\]

\end{lemma}

\begin{proof}
Since for each primary value $\lambda$ of $r$ the space $V_\lambda(r) + V_{\lambda^{-1}}(r)$ is $rh$-invariant and $h$-invariant,
it is enough to prove the statement of the lemma for each such subspace separately.
If $\lambda$ is of type $A$, then $h$ and $rh$ map $V_\lambda(r)$ isomorphically into $V_{\lambda^{-1}}(r)$ and vice versa.
Therefore, the maps $v \mapsto v + h(v)$ and $v \mapsto v + rh(v)$ give linear isomorphisms
of $V_\lambda(r)$ with $(V_\lambda(r) + V_{\lambda^{-1}}(r)) \cap V_1(h)$ and $(V_\lambda(r) + V_{\lambda^{-1}}(r)) \cap V_1(rh)$
respectively. It follows that
\[\dim((V_\lambda(r) + V_{\lambda^{-1}}(r)) \cap V_1(h)) = \dim(V_\lambda(r)) = \frac{1}{2}\dim(V_\lambda(r) + V_{\lambda^{-1}}(r)),\]
so the formula is correct in that case.

Next assume that $\lambda$ is of type $B$. Then $V_\lambda(r) = V_{\lambda^{-1}}(r)$ and
$V_\lambda(r)$ has a structure of a $F_\lambda$ vector space, with $h$ being a semi-linear automorphism,
and the same is true for $rh$.
It follows from Hilbert's 90 Theorem that $V_1(h) \cap V_\lambda(r) \otimes_{F[r + r^{-1}]} F[r] \cong V_\lambda(r)$
and since $F[r] / F[r + r^{-1}]$ is a quadratic extension it follows that
\[\dim(V_1(h) \cap V_\lambda(r)) = \frac{1}{2}\dim(V_\lambda).\]

The case where $\lambda$ is of type $C$ is immediate: if $\lambda = 1$ then $rh = h$ on $V_\lambda(r)$ and if
$\lambda = -1$ then $rh = -h$ on $V_\lambda(r)$.
%\qed
\end{proof}

\subsubsection {Quadratic and Hermitian Forms}

Many of the calculations related to symmetric pairs with $G = O(B)$ or $H = O(B)$ 
involve many considerations related to quadratic forms. Similarly, we shall use various classical facts regarding Hermitian forms. In this section we recall some of these classical facts. This section contains no new results, and is mainly intended to fix notation and recall results that will be used in the computations.

Over a non-Archimedean local field we have the following c
Let $V$ be a linear space over $F$, or over $E$ for $E/F$ a quadratic extension $E/F$. Let $\Quad(V)$ (resp. $\Her_F(V)$) denote the set of symmetric bilinear forms
$B : V \times V \to F$ (resp. Skew-symmetric forms $V \times V \to E$). We identify $B$ with its associated quadratic 
(resp. Hermitian) form $v \mapsto B(v,v)$ and write it simply as $B(v)$.
Let $\mathcal{QF}_n(F)$ (resp. $\mathfrak{Her}_n(F))$ denote the set of equivalence classes of pairs $(B,V)$ of a space with a quadratic (resp. Hermitian) form on it such that $\dim(V) = n$.
For every quadratic form $B \in \Quad(V)$, we let $[B]$ denote the class of $(V,B)$ in $\mathcal{QF}_{\dim(V)}(F)$, 
and similarly for $\Her_F(V)$. 
Let $\mathcal{QF}(F) = \cup_n \mathcal{QF}_n(F)$ and $\mathcal{H}er(E/F) = \cup_n \mathcal{H}er_n(E/F)$.

If $V$ is a vector space over $F$ and $B$ is a quadratic form on $V$
let $B_{E,c}(x,y) := B_E(x,c(y))$. It is an Hermitian form on $V(E)$. This gives a map
\[(\cdot)_{E,c}:\Quad(V) \to \Her_F(V(E)).\]
The image consisting of exactly those forms for which $c(B(c(x),c(y))) = B(x,y)$.

If $B$ is a quadratic form on $V$ and $g: V \to V$ is a linear map such that $B(gv,u) = B(v,gu)$ for every $v,u \in V$ then we denote by $B_g$ the quadratic form $B_g(u,v) = B(gu,v)$.

There is a natural onto map \[Q_n : (F^\times / (F^\times)^2)^n \to \mathcal{QF}_n(F)\]
given by $Q_n([a_1,...,a_n]) = (F^n, \sum_{i = 1}^n a_i x_i^2)$. 
Let $Q : \coprod_n F^\times / (F^\times)^2)^n \to \mathcal{QF}(F)$ be the union of these maps.

Similarly, we have an onto map 
$H_n: (F^\times / N_{E/F}(E^\times))^n \to \mathcal{H}er_n(E/F)$ and we denote by 
$H$ the union of those maps.

Let $\{a,b\}$ denote the Hilbert symbol of $a$ and $b$, defined by
\[\{a,b\} = \begin{cases}
1,  \quad \exists x,y \in F : ax^2 + by^2 = 1 \\
-1, \quad \text{otherwise} \end{cases}.\]
We have several invariants on quadratic forms. For $B \in \Quad(V)$ let:
\begin{itemize}
\item $\det(B) \in F^\times / (F^\times)^2$ be the determinant of $B$, represented as a matrix in some basis of $V$.
\item $\rank(B) = \dim(V)$.
\item $H(B)$ be the Hasse invariant of $B$, i.e. the product of all the Hilbert symbols of pairs of diagonal elements of $B$ in some orthogonal basis.
\item $\mu(B)$ the maximal dimension of a subspace $U \subseteq V$ such that $B|_U = 0$.
\end{itemize}

They clearly descent to give maps
\[\rank : \mathcal{QF}(F) \to \NN\]
\[\det : \mathcal{QF}(F) \to F^\times / (F^\times)^2\]
\[H : \mathcal{QF}(F) \to \{\pm 1\}\]
\[\mu : \mathcal{QF}(F) \to \NN\]
Define similarly on sequences of elements of $F^\times / (F^\times)^2$
\[\rank([a_1,...,a_n]) = n\]
\[\det([a_1,...,a_n]) = \prod_{i = 1}^n a_i\]
\[H([a_1,...,a_n]) = \prod_{1 \le i < j \le n}\{a_i,a_j\}\]
so that for every sequence $\ell$ of elements of $F^\times / (F^\times)^2$ we have
$\rank(\ell) = \rank(Q(\ell))$, $\det(\ell) = \det(Q(\ell))$ and $H(\ell) = H(Q(\ell))$.

For Hermitian forms we consider only the invariants 
$\det : \mathcal{H}er(E/F) \to F^\times / N_{E/F}(E^\times)$ 
and $\rank: \mathcal{H}er(E/F) \to \ZZ$, defined in the oBT65ious way. 
 
For non-Archimedean fields, it turns out to be a complete set of invariants for quadratic or Hernitian forms.

\begin{proposition}[{\cite[Theorem 2.3.7]{Ser73}}, {\cite[Theorem 3.1]{Jac62}}]
\label{proposition: classification of quadratic forms over local fields}
\label{proposition: classification of Hermitian forms}
A quadratic form over a local non-Archimedean field of characteristic $0$ is completely determined up to equivalence by its rank, determinant and
Hasse invariant.

A Hermitian form over a local non-Archimedean local field of characteristic 0 is completely classified by its rank and determinant.  
\end{proposition}

In the case $F = \RR$ this is not the case but Sylvester Theorem states that the rank and signature form a complete set of invariants for quadratic forms and of Hermitian forms.

The following classical result will be useful in some cohomology computations:

\begin{proposition}[Witt Cancellation Theorem]
\label{proposition: Witt cancellation theorem}
Let $B,B',C$ be three quadratic forms over $F$.
If $B \oplus C \equiv B' \oplus C$ then $B \equiv B'$.
In other words, $\cdot \oplus C : \mathcal{QF}(F) \to \mathcal{QF}(F)$ is injective.
\end{proposition}
We shall say that $B \le B'$ if there is a quadratic form $C$ such that $B' \equiv B \oplus C$.
In this case, by Witt Cancellation Theorem we can unambiguously define difference $[B] - [B']$ to be $[C]$.

Let $V$ be a linear space over $F$ and $V^*$ the dual space. We associate with $V$ the \textbf{hyperbolic form} $\mathcal{H}_V$ on
$V \oplus V^*$ given by $\mathcal{H}_V(v + \phi) = \phi(v)$. Its class depends only on $n = \dim(V)$
and we denote $\mathcal{H}_n = [\mathcal{H}_{F^n}]$. Clearly
$\mathcal{H}_n = \oplus_{i = 1}^n \mathcal{H}_1$.

For a quadratic form $B$, we say that a scalar 
$x \in F / (F^\times)^2$ is \textbf{represented by $B$} if there is $v \in V$ such that $B(v) = x$.
Let $\Rep(B)$ denote the set of $(F^\times)^2$ orbits of elements $x \in F$ represented by $B$. 
Clearly $\Rep(\mathcal{H}_V) = F / (F^\times)^2$.

The following result follows easily from the classical Witt Decomposition Theorem. 
\begin{proposition}
\label{proposition: the hyperbolic lemma}
Let $B$ be a quadratic form and $k$ the maximal integer such that $\mathcal{H}_k \le B$.
Then $k = \mu(B)$.
\end{proposition}

%Next, we discuss the relation between quadratic forms over $F$ 
%and over its field extensions. 

\subsection {Some Cohomology Computations}
\label{subsetion: some cohomology computations}
Here we shall calculate the different cohomology sets that we need for the stability calculations. Once again here, most of the results are classical and easy, but we include them for completeness of the exposition.

First, one easily check that the following formulas hold: 
\begin{proposition}
Let $(G,H,\theta)$ be a symmetric pair.
\label{proposition: cohomology of Fli91p}
\label{proposition: cohomology of conjugacy}
\label{proposition: cohomology commutative case}
\begin{itemize}
	\item If $G = H \times H$ and $\theta(x,y) = (y,x)$, then $\HH^1(\theta,G) = \{1\}$.  
	
	\item If $H = Z_G(h)$ and $\theta(x) = hxh^{-1}$ for $h^2 \in Z(G)$, then 
			$\HH^1(\theta,G)$ can be identified with the set $\{h' \in G : h'^2 = h^2\} / \{G-conjugacy\}$, via the map $r \mapsto rh$. The neutral element is then given by $h$.  
	\item If $G$ is commutative and $\theta(x) = x^{-1}$ then $\HH^1(\theta,G) \cong G/G^2$. 
\end{itemize}
\end{proposition}

These observations together with the classification of 
elements of order 2 in classical groups immediately gives us the following formulas for cohomologies.

\begin{proposition}
\label{proposition: cohomology of inner involutions}
In the following cases, $h\in G$ is an element of order 2. 

\begin{itemize} 
\item Let $G = GL(V)$. Then $\HH^1(Ad_h,GL(V)) \cong (\{0,...,\dim(V)\},\dim(V_1(h)))$. 
\item Let $G = SL(V)$. Then $\HH^1(Ad_h,SL(V)) \subseteq  \HH^1(Ad_h,GL(V))$ and via the 
identification of $\HH^1(Ad_h,GL(V))$ with $(\{0,...,\dim(V)\},\dim(V_1(h))$, 
$\HH^1(Ad_h,SL(V))$ correspond to the numbers of the same parity as $\dim(V_1(h))$. 
\item Let $B$ be a quadratic form on $V$ and $G = O(B)$. 
Then \[\HH^1(Ad_h, O(B)) \cong (\{B' \in \mathcal{QF}(F) : \exists C, B' \oplus C \equiv B\}, B|_{V_1(h)}).\] 
\item Let $V$ be a vector space over a quadratic extension $E$ of $F$, and let $B$ be a Hermitian form 
on $V$. Then 
$\HH^1(Ad_h,U(B)) \cong (\{B' \in \mathcal{H}er(E/F) : \exists C, B' \oplus C \equiv B\}, B_{V_1(h)})$
\item Let $V$ be a vector space and $\omega$ a symplectic on $V$. 
 Let $h \in Sp(\omega)$ be a symplectic involution of $V$.
Then $\HH^1(Ad_h,Sp(\omega)) \cong (\{0,2,4,...,\dim(V)\}, \dim(V_1(h))$. This identification is, moreover, 
compatible with the inclusion $Sp(\omega) \subseteq SL(V)$. 
  
\end{itemize}

\end{proposition}	

\begin{proof}
	The cohomology of the form $\HH^1(Ad_h,G)$ is 
	identified via twisting with the set of conjugacy 
	classes of elements of order 2 in $G$. 
	
	In the case $G = GL(V)$, involutions of $V$ are 
	classified up to conjugation by the dimension of their 
	1-aigen space. 
	
	In the case $G = SL(V)$, the determinant restriction 
	impose the parity constraint on this subspace, and it is 
	easy to see that no two conjugacy classes in $SL(V)$ 
	coincide in $GL(V)$. 
	
	In the case of quadratic or Hermitian spaces, 
	the involutions are classified up to conjugacy by 
	the restriction of the form to the 1-aigen space 
	$V_1(h)$. 
	The result follows easily from these 
	classical observations. 
\end{proof}

Next, we compute the cohomology of the involutions arising from 
the action of $\Gal_{E/F}$ via restriction of scalars from $E$ to $F$.
\begin{proposition}
\label{proposition: cohomology of semi-linear involutions}
Let $E/F$ be a quadratic extension. Let $c : E  \to E$ denote conjugation of $E$ over $F$.
The following holds:

\begin{itemize}
\item  $\HH^1(c,GL_n(E)) = 1$ 
\item  $\HH^1(c,SL_n(E)) = 1$ 
\item $\HH^1(c,SP_{2n}(E)) = 1$ 
\item $\HH^1(c,O_E(B_E)) \cong (\{B' \in \mathcal{QF}(F) : B'_E \equiv B_E\}, B)$ 
\item $\HH^1(c,U(B_{E,c})) = (\{B' \in \mathcal{QF}(F) : B'_{E,c} \equiv B\}, B)$
\end{itemize}
\end{proposition}

\begin{proof}
The first equality is Theorem Hilbert's 90. All the others follow from the first and Theorem \ref{proposition: descent lemma},
by considering the action of $GL(V)$ on $E^\times$ via $det$, its action on the spaces of anti-symmetric, symmetric and
Hermitian bilinear forms respectively. Note that up to equivalence there is a unique non-degerenrate anti-symmetric form.
%\qed
\end{proof}

\subsection{Maximal Split Tori of Certain Symmetric Pairs} 

Our aim in this section is to give a systematic calculation of the
maximal $(\theta,F)$-split tori in classical groups, 
together with the induced map  
This computation is important for the verification of 
s-stability, t-stability and p-stability of those pairs, 
since the obstructions for those properies live in the cohomology 
of $A$ and $Z_G(A)$.

The description of the maximal $(\theta,F)$-split tori 
of the classical pairs that we consider is given in terms of split subspaces 
of a linear space, or in a linear space endowed with a bilinear form. 

\begin{definition}
	Let $V$ be a linear space over a field $F$. 
	Let $h : V \to V$ be a map such that $h^2 = \lambda I_V$.  
	A subspace $U \subseteq V$ is called \textbf{$h$-split} 
	if $U \cap h(U) = \{0\}$. 
\end{definition}

Note that the maximal dimension of an $h$-split subspace is 
given by the minimum among $\dim(V_1(h))$ and $\dim(V_{-1}(h))$.  

If the space $V$ is endowed with a quadratic form $B$, we 
restrict the class of subspaces under consideration 
further: 

\begin{definition}
Let $(V,B)$ be a linear space with a non-degenerate bilinear form $B$, 
which is either symmetric or anti-symmetric. 
Let $h : (V,B) \to (V,B)$ be an linear map with scalar square. 
A subspace $U \subseteq V$ is called \textbf{$(h,B)$-split} 
if it is $h$ split and in addition $B|_U = 0$ but 
$B|_{U + h(U)}$ is non-degenerate. 
\end{definition}

Being $(h,B)$ split is equivalent to the existence of a basis 
of $V$ of the form $x_1,...,x_k,y_1,...,y_k,z_1,...,z_l$ in which 
$B$ takes the form  
\[\begin{pmatrix}
0  & A_{k \times k}  & 0 \\ 
A'_{k \times k}  & 0   & 0  \\
0  & 0  & C_{l \times l} 
\end{pmatrix}\] 

and $h$ takes the form 
\[\begin{pmatrix}
0  & I_{k}  & 0 \\
\lambda I_k  & 0   & 0 \\
0  & 0  & H_{l \times l} 
\end{pmatrix}\] 

Given a maximal $h$-split subspace 
$U\subseteq V$, a maximal split torus 
$T \subseteq GL(U)$ and an $h$-stable linear complement 
$W$ of $U + h(U)$ in $V$, we denote by 
$A(U,W,T) \cong T$ the torus given by 
the image of the map 
$j_T : T \to GL(V)$ given by 
$j_T(x) = x \oplus (hxh)^{-1} \oplus I_W$. 
More precisely we consider all those 
matrices which respect the decomposition 
$V = U \oplus h(U) \oplus W$, act as identity on 
$W$, and act via $x \in T$ and $(hxh)^{-1}$ on 
$U$ and $h(U)$ respectively.    
In the basis as above, this means that the elements of 
$A(U,W,T)$ takes the form 

\[\begin{pmatrix}
x  & 0  & 0 \\
0  & x^{-1}   & 0 \\
0  & 0  & I_{l} 
\end{pmatrix}\] 

The tori $A(U,W,T)$ are $Ad_h$-split tori. 
We shall show that all the maximal $Ad_h$-split tori 
of the corresponding pairs arise from this construction.

\begin{proposition}
	\label{proposition: maximal h-split tori GL(V)}
	In all the statements below, $h \in G$ denote an element such that 
	$h^2 \in Z(G)$. 
	\begin{itemize}
		\item The maximal $(Ad_h,F)$ split tori in  
		$GL(V)$ or $SL(V)$ are precisely the tori of the form $A(U,W,T)$ for $U$ a maximal 
		$h$-split subspace.    
		\item Let $B$ be a symmetric, anti symmetric or Hermitian form on a space $V$. 
		The maximal $Ad_h$-split tori in $G(B)$, the stabilizer of the form $B$ in $GL(V)$, 
		are exactly the tori of the form $A(U,(U + h(U))^\bot, T) \cap G(B)$ for $U$ a maximal 
		$(h,B)$-split subspace.  	
\end{itemize}
\end{proposition}

\begin{proof}
	Clearly the tori described in the proposition are all $(Ad_h,F)$-split, 
	and are maximal among tori of the form 
	$A(U,W,T)$ contained in $G$. 
	It remains to show that every $(\theta,F)$-split torus is contained in a torus 
	of the form $A(U,W,T)$. Let $A \subseteq G$ be an $(Ad_h,F)$-split torus. Let  
	$V = \sum_\chi V_\chi$ be a decomposition of $V$ into simultaneous aigen-spaces for 
	$A$. since $h x h^{-1} = x^{-1}$ for every $x \in A$, the set of $\chi$-s appearing in the decomposition. Moreover, in the case where we have a bilinear form as well, 
	the condition $V_\chi \bot V_\psi$ is satisfied unless $\chi = \psi^{-1}$. 
	Choosing a half-space $P^+ \subseteq X^*(A) \otimes \RR$ (i.e. a side of a hyper-plane) which contain no 
	$\chi$ in the decomposition on its boundary, we can now choose 
	$U = \oplus_{\chi \in P^+} V_\chi$ and $W = V_1$. 
	Denote by $T$ the torus generated by the restrictions of the elements of 
	$A$ to $U$. 
	We get immediately that $A \subseteq A(U,W,T)$. 
\end{proof}

In a similar fashion, one can classify maximal $(\theta,F)$-split tori in 
pairs of the form $(G(E),G(F))$ for $G = GL(V),SL(V)$ or  $G(B)$ for 
a symmetric or anti-symetric form $B$.

\begin{definition}
	Let $E/F$ be a quadratic extension. Let $V$ be a vector space over $F$. 
	Let $B$ be a non-degenerate bilinear form on $V$, defined over $F$. 
	Let $c : V(E) \to V(E)$ denote the action of the non-trivial element of 
	$\Gal_{E/F}$. 
	
	An $E$-linear subspace $U \subseteq V(E)$ is called 
	$c$-split if it is $c$-split as an $F$-linear subspace of $V(E)$. 
	It is $(c,B)$-split if it is $(c,B)$-split as an $F$-subspace.    
\end{definition}

 For a maximal $F$-split torus $T \subseteq GL_E(U(E))$, and a complement $W$ which is 
 the orthogonal complement to $U + c(U)$ in the case where we have a form, we denote by 
 $A(U,W,T)$ the torus consisting of matrices of the form 
 $x \oplus c(x)^{-1} \oplus Id_W, \quad x \in T$. All maximal 
 $(c,F)$-split tori in $G$ arise from this construction, 
 just like in the case of inner involutions.

\subsection{Non-s-Stable Pairs}
 
Using the fact that s-stability is implied by p-stabiliy, or by the stability, of a symmetric pair, 
we can use the cohomlogical criterion for s-stability in Theorem \ref{theorem: criterion for s-stability} to exclude the stability and 
p-stability of many pairs.
The strategy for falsifying the s-stability of all those pairs is very simple.
We compute the maximal $(\theta,F)$-split tori and for such pairs, and the induced map on cohomology.

\begin{theorem}
	\label{theorem: some non-s-stable pairs}
	The following pairs are not s-stable. In particular, 
	they are not stable or p-stable. 
	
	\begin{itemize}		
		\item The pair $(SL(V),Z_{SL(V)}(h),Ad_h)$ for $h : V \to V$ 
		an involution with $\dim(V_1(h)) = \dim(V_{-1}(h))$. 

		\item The pair $(SL(V(E))), SL(V(F),c)$ for $E/F$ 
		a quadratic extension, $V$ a vector space over $F$ 
		of even dimension and $c: V(E) \to V(E)$ corresponding to the conjugation of $E$ over $F$.

		\item $(GL(V), O(B), \theta)$ for $\theta(x) = x^{-t}$ if $F = \RR$ and 
		$B$ is not definite, or $F$ is non-Archimedean and $\rank(B) > 1$. 
		\item $(GL(V(E)), U(B), \theta)$ for $E/F$ quadratic extension, $V$ an $F$-vector space of dimension $>1$, $B$ a Hermitian form on $V(E)$ which is non-definite if $F = \RR$ and $\theta(x) = (x^*)^{-1}$, the inverse of the adjoint with respect to $B$.   
		\item $(Sp(\omega), GL(U), Ad_h)$ where $\omega$ is a symplectic form on a vector space $V$, $U \subseteq V$ is a Lagrangian subspace, $h : V \to V$ is an involution with $h^* \omega = -\omega$ and $h(U) = U$. 		
		\item $(O(B),GL(U),Ad_h)$ where $B$ is a quadratic form on $V$ and 
		$h^* B = -B$. 
		\item $(SL_F(V), SL_E(V), Ad_J)$ where $F$ in non-Archimedean local field, $V$ is a vector space over $F$, $\dim(V) > 1$,  
		$J : V \to V$ is a linear operator with 
		$J^2 = a \notin F^\times$ and $E = F[J]$. 
		\item $(O(B),O(B^+) \times O(B^-) , Ad_h)$ for 
		$F=\RR$ and both $B^+,B^-$ are not definite.     
		\item $(U(B),U(B^+) \times U(B^-) , Ad_h)$ for 
		$F=\RR$ and both $B^+,B^-$ are not definite, or 
		$F$ is non-Archimedean and the pair is of rank bigger than $1$. 
		\item The pair $(\mathbf{GSp}(\omega)(E), \mathbf{GSp}(\omega)(F),c)$ for $\omega$ a symplectic form on a vector space $V$ over $F$.
		\item The pair $(\mathbf{Sp}(\omega)(E), \mathbf{Sp}(\omega)(F),c)$ for $\omega$ a symplectic form on a vector space $V$ over $F$. 
\end{itemize}

\end{theorem}
The rest of this sub-section will be dedicated to the proof of this theorem. 

we start with the pair
$(SL(V),Z_{SL(V)}(h),Ad_h)$, for $\dim(V_1(h)) = \dim(V_{-1}(h))$. 
By Proposition \ref{proposition: maximal h-split tori GL(V)}, a maximal $(Ad_h,F)$-split torus in $SL(V)$ is of the form 
$A(U,W,T)$ for an $h$-split subspace $U$ of $V$ of maximal possible dimension and $h$-invariant complement linear $W$.. 
In this case we can find $U$ of dimension $\dim(V)/2$, so that 
$W = \{0\}$. 
Let $x_1,..,x_n$ be a basis of $U$. Then 
$x_1,...,x_n,h(x_1),...,h(x_n)$ is a basis of $V$ in which 
the torus $A = A(U,W,T)$ can be chosen to consist of diagonal matrices for a suitable choice of $T$.
More precisely, for such a choice we have  
\[A=\Big\{\begin{pmatrix}
a & 0 \\ 
0 & a^{-1}
\end{pmatrix}, a\in (F^\times)^n\Big\}
\]
and 
\[h = \begin{pmatrix}
0 & Id_n \\ 
Id_n & 0
\end{pmatrix}
\]

Moreover, $Z_{SL(V)}(A)$ is the standard maximal torus of $SL(V)$ in this basis. 
We claim that the map 
$\HH^1(Ad_h,Z_G(A)) \to \HH^1(Ad_h,SL(V))$ is the constant map, and that 
the map $\HH^1(Ad_h, A) \to \HH^1(Ad_h,Z_G(A))$ is not, so that necessarily 
$\ImH{Ad_h}{A}{Z_G(A)} \cap \KerH{Ad_h}{Z_G(A)}{G}$ is non-trivial. Consequently,  
by Theorem \ref{theorem: criterion for s-stability}, the pair is not s-stable. 

Indeed, via the identification 
$\HH^1(Ad_h,SL(V)) \cong \{t \in \{0,...,\dim(V)\} : t \equiv \dim(V)/2 \mod 2\}$, 
an element $[r] \in \HH^1(Ad_h,SL(V))$ is mapped to $\dim(V_1(rh))$. 
If $r \in Z_G(A)$ then $r$ is a diagonal matrix and therefore 
$rh$ is of the form 
\[\begin{pmatrix} 
0 & A \\ 
A^{-1} & 0 
\end{pmatrix}.\] 
Every involution of this form has $+1$ and $-1$-eigen spaces of the same dimension, and hence in this case $\dim(V_1(rh)) = \dim(V)/2$. This implies that $[r]$ is the neutral element of $\HH^1(Ad_h,SL(V))$. 

However, $\HH^1(Ad_h,A) \cong A/ A^2$ while $\HH^1(Ad_h,Z_G(A)) \cong F^\times / (F^\times)^2$, 
e.g. by considering the long exact sequence associated with the short exact sequence
\[Z_{SL(V)}(A) \to Z_{GL(V)}(A) \stackrel{\det}{\to} F^\times.\] 
More precisely, this identification is given by 
writing a representative of the cocycle $[r]$ as $r = \delta(x)$ for $x \in Z_{GL(V)}(A)$, 
and then sending it to $\det(x) \mod (F^\times)^2 $.  
The resulting map $A/ A^2 \to F^\times / (F^\times)^2$  is given by 
$x \mapsto \det(x|_U)$. Indeed, in $Z_{GL(V)}(A)$ we can write 
$x = \delta(x|_U \oplus Id_{h(U)})$ and the determinant of $x|_U \oplus Id_{h(U)}$ 
is the same as $\det(x|_U)$. Since this map is not constant, this shows that the pair $(SL(V), Z_{SL(V)}(h), Ad_h)$ is not s-stable. 

Next we consider the pair $(SL(V(E)), SL(V(F)),c)$. It is easy to see that
if $\dim(V(E))$ is even, then 
we can find a $c$-split subspace of $V(E)$ of dimension $\dim(V)/2$. 
Let $U$ be such a subspace, and consider the torus 
$A = A(U,\{0\},T)$, for some maximal $F$-split torus in $GL_E(U)$. 
By Hilbert's Theorem 90 (see Proposition \ref{proposition: hilbert 90}), $\HH^1(c,SL(V(E))) = 1$. Hence, to prove that the pair 
in not s-stable, it suffice to prove that $\ImH{c}{A}{Z_{G}(A)}$ in non-trivial. 
Note that $A \cong (F^\times)^n$ while $Z_G(A) \cong (E^\times)^n$ for 
$n = \dim(V)/2$. The action of $c$ on $Z_G(A)$ is given by 
\[c(x_1,...,x_n) = (c(x_1)^{-1},...,c(x_n)^{-1}),\] so that 
\[\HH^1(c, Z_G(A)) \cong (F^\times / N_{E/F}(E^\times))^n,\] 
while $\HH^1(c,A) \cong (F^\times / (F^\times)^2)^n$. 
The map $\HH^1(c,A) \to \HH^1(c,Z_G(A))$ is given by sending 
the class of $(x_1,...,x_n)$ modulo squares to the same class modulo norms from 
$E^\times$. This map is surjective and non-trivial for both 
$F =\RR$ and for $E/F$ a quadratic extension of local non-Archimedean fields of charcteristic $0$. 
We deduce that the pair $(SL(V(E)), SL(V(F)),c)$ is not s-stable. 

Consider now the pair $(GL(V), O(B), \theta)$. Choose a basis $\{x_1,...,x_n\}$ such that 
$B$ is diagonalized in this basis, and let $A$ be the torus of diagonal matrices in this basis. 
The torus $A$ is a maximal $(\theta,F)$-split torus, and in fact a maximal torus of $GL(V)$. 
It follows that $Z_{GL(V)}(A) = A$, and therefore to show that the pair is not s-stable, it suffice to show that 
$\KerH{\theta}{A}{GL(V)}$ in not trivial. 

We have \[\HH^1(\theta, A) = A/ A^2 \cong (F^\times / (F^\times)^2)^n,\] 
while $\HH^1(\theta,GL(V))$ is identified with the collection of quadratic forms on 
$V$ up to equivalence. More precisely, this identification is done as follows: if $\theta(r) = r^{-1}$ then 
$B(rv,u) = B(v,ru)$, and we associate with $r$ the quadratic form 
$B_r(v,u) = B(rv,u)$. 
If $B = Q([b_1,...,b_n])$ and $r = (a_1,...,a_n)$ then 
$B_r = Q([a_1b_1,...,a_nb_n])$. The problem is thus reduced to the following one. We need to determine for which quadratic forms $B = Q([b_1,...,b_n])$ there exist a non trivial element 
$(a_1,...,a_n) \in F^\times / (F^\times)^2$ such that 
$Q([a_1b_1,...,a_nb_n]) \equiv B$. 
We claim that this is possible if and only if $F = \RR$ and 
$B$ is non-definite, or $F$ is non-Archimedean and $\dim(V) > 1$. 

Indeed, if $F = \RR$ and $B$ is definite then the condition 
$Q([a_1b_1,...,a_nb_n]) \equiv B$ imply that $a_i > 0$ and hence that 
$(a_1,...,a_n) \in ((\RR^\times)^2)^n$.  
Conversely, if $B$ is not definite we may assume that 
$b_1 = 1$ and $b_2 = -1$. But then 
$(-1,1,1,...,1)$ is an example of a non-trivial element in 
$\KerH{\theta}{A}{GL(V)}$. 

If $F$ is non-Archimedean and $\dim(V) > 1$, then it suffices to find some 
$a_1,a_2$ not both squares such that $Q([a_1b_1,a_2b_2]) \equiv Q([b_1,b_2])$. 
By the classification of quadratic forms over $F$, as stated in Proposition \ref{proposition: classification of quadratic forms over local fields}, this condition is equivalent 
to $a_1 a_2 = 1 \mod(F^{\times})^2$ and 
$\{a_1b_1,a_2b_2\} = 1$. Namely, $a_1=a_2= \lambda$ and 
$\{\lambda b_1,\lambda b_2\} = \{b_1,b_2\}$. 
This equation is linear and homogenuous in $\lambda$, so a non-trivial solution 
exist because $dim_{\FF_2}(F^\times / (F^\times)^2) > 1$.  

Next consider the pair $(Sp(\omega), GL(U), Ad_h)$. 
It is easy to see that the set
$\HH^1(Ad_h,Sp(\omega))$ classifies pairs of transversal Lagrangian 
subspaces of $V$, and since every such pair is conjugate to any other pair, 
we have $\HH^1(Ad_h,Sp(\omega)) = 1$. However, we can choose a basis 
$\{x_1,...,x_n\}$ of $U = V_1(h)$ and a dual basis $\{y_1,...,y_n\}$ of 
$U' = V_{-1}(h)$. In the basis 
\[\{x_1 + y_1,...,x_n + y_n, x_1 - y_1,...,x_n - y_n\}\] the 
torus $A$ represented by matrices of the form 
$diag(\lambda_1,...,\lambda_n,\lambda_1^{-1},...,\lambda_n^{-1})$ 
is a maximal $(\theta,F)$-split torus. Note that this is a maximal torus of $Sp(\omega)$, so $Z_{Sp(\omega)}(A) = A$. 
Moreover, 
\[\KerH{A}{G}{\theta} = \HH^1(A) = \HH^1(Z_G(A)) = A/ A^2 \ne 1\] in this case. 
It follows that the pair is not s-stable. 
Exactly the same argument works also for the pair 
$(O(B),GL(U),Ad_h)$ where $h: V \to V$ is an involution with 
$h^*B = -B$ (and in particular $B$ must be hyperbolic as it contains 
an isotropic subspace $V_1(h)$ of half the dimension of $V$).       

Consider now the pair $(SL_F(V), SL_E(V), Ad_J)$.  
Recall that here $J^2=aId_V$,and choose a basis $\{x_1,...,x_n,y_1,...,y_n\}$ such that
\[J x_i = y_i, \quad J y_i = a x_i.\]
Then
\[A = \{g \in SL_F(V) : gx_i = \lambda_i x_i, \quad gy_i = \lambda_i^{-1} y_i, \quad \lambda_i \in F^\times\}\]
is a maximal $(Ad_J,F)$-split torus in $SL_F(V)$.
Let $T$ be the torus of all diagonal matrices in our basis.
Then $Z_{SL_F(V)}(A) = T$. To compute $\HH^1(Ad_J,T)$, first note that
since $A = T^-$ the map $\HH^1(Ad_J,A) \to \HH^1(Ad_J,T)$ induced by the inclusion is onto.
Thus, the pair is not s-stable if the map
$\HH^1(Ad_J,T) \to \HH^1(Ad_J,SL_F(V)) \cong F^\times/N_{E/F}(E^\times)$ has non-trivial kernel.
In order to compute $\HH^1(Ad_J,T)$, let $\tilde{T} = Z_{GL_F(V)}(A)$.
It is easy to see that
\[\HH^1(Ad_J,\tilde{T}) = 1.\]
We have an exact sequence of abelian $\ZZ/ 2\ZZ$ modules:
\[\SES{T}{\tilde{T}}{F^\times}\]
and as a result we get a (prtion of) a long exact sequence
\[\begin{CD} \det(\tilde{T}^+) @>>> F^\times @>>> \HH^1(Ad_J,T) @>>> 1  \end{CD}.\]
But $\det(\tilde{T}^+) = (F^\times)^2$ so
\[\HH^1(Ad_J,T) \cong F^\times / (F^\times)^2.\]
The induced map
\[F^\times / (F^\times)^2 \cong \HH^1(Ad_J,T) \to \HH^1(Ad_J,SL_E(V)) \cong F^\times / N_{E/F}(E^\times)\]
is the natural quotient map, which has non-trivial kernel 
if $F$ is non-Archimedean. 

We now consider the pair $(O(B),O(B^+) \times O(B^-) , Ad_h)$, 
for $F = \RR$ and $B^+,B^-$ both non-definite.   
We can choose a $(B,h)$-split subspace 
$U$ of $V$ such that the form 
$B_h(x,y) = B(x,h(y))$ on $U \oplus h(U)$ is non-definite as follows. Represent the quadratic form $B^+$ in a basis 
$\{x_1,...,x_n\}$ in which  
$B^+ = Q([-1,1,...])$. 
Consider the subspace $U = \Span(x_1 + h(x_2))$. 
It is a $(B,h)$-split one dimensional subspace, 
and $B_h|_U$ is represented by the matrix 
\[
\begin{pmatrix}
B(x_1 + h(x_2), h(x_1) + x_2) & B(h(x_1) + x_2, h(x_1) + x_2) \\ 
B(h(x_1) + x_2, h(x_1) + x_2) & B(h(x_1) + x_2, x_1 + h(x_2)) 
\end{pmatrix}
= 
\begin{pmatrix}
0 & 2 \\ 
2 & 0 
\end{pmatrix}
\]
Which is hyperbolic. 

Let $T = GL(U)$ and $A = A(U,T)$ be the corresponding $(\theta,F)$-split torus. 
Then $A$ is not necessarily maximal $(\theta,F)$-split, 
but if $A \subseteq \tilde{A}$ is an extension to a maximal 
$(\theta,F)$-split torus, the sequence of inclusions 
\[A \subseteq \tilde{A} \subseteq Z_G(\tilde{A}) \subseteq Z_G(A)
\subseteq G\] 
shows that it suffice to prove that 
$\ImH{\theta}{A}{Z_G(A)} \cap \KerH{\theta}{Z_G(A)}{G} \ne 1.$ 
Moreover, for the subspace $W = (U + h(U))^{\bot}$ we have 
$Z_G(A) \cong A \times O(B|_W)$ as a $\ZZ/2$-groups. 
The inclusion 
$Z_G(A) \to O(B)$ factor through $O(B|_{U + h(U)}) \times O(B|_W)$, so it suffices to show that 
\[\ImH{\theta}{A}{Z_G(A)} \cap O(B|_{U + h(U)})) \cap \KerH{\theta}{Z_G(A) \cap O(B|_{U + h(U)})}{O(B|_{U + h(U)})} \ne 1.\] 

Thus, the problem is reduced to the case 
\[
B = \begin{pmatrix}  
0 & 1 \\ 
1 & 0 
\end{pmatrix}, 
h = \begin{pmatrix}  
0 & 1 \\ 
1 & 0 
\end{pmatrix}, 
A =  
\Big\{\begin{pmatrix}  
\lambda & 0 \\ 
0 & \lambda^{-1} 
\end{pmatrix} : \lambda \in F^\times\Big\} 
 \]
 Take the element 
 $x = -Id_V \in A$. 
 By Proposition \ref{proposition: cohomology of conjugacy}, the class of $x$ in 
 $\HH^1(Ad_h,O(B))$ is represented by the involution 
 $xh = -h$ . 
 But
 \[\begin{pmatrix}
 -1 & 0 \\ 
 0 & 1 \\ 
 \end{pmatrix}(-h) \begin{pmatrix}
 -1 & 0 \\ 
 0 & 1 \\ 
 \end{pmatrix} =h\] 
  and hence  
 $x$ represent a non-trivial element in 
 $\ImH{Ad_h}{A}{Z_{O(B)}(A)} \cap \KerH{Ad_h}{Z_{O(B)}(A)}{ O(B)}$. 
 
 The case of the pair $(U(B),U(B^+) \times U(B^-) , Ad_h)$ for 
 $F = \RR$ is similar and we leave the details for this case for the 
 reader. 
 
 We are left with the pair  
 $(U(B),U(B^+) \times U(B^-) , Ad_h)$ for non-Archimedean $F$. 
 We have to prove that if the dimension of a maximal $(\theta,F)$-split 
 torus is bigger than $1$, then the pair is not s-stable. 
 We can choose a 2-dimensional 
 $(\theta,F)$-split torus of the form $A(U,T)$ for a 2-dimensional 
 $(B,h)$-split subspace $W$ of $V$, and as in the previous case 
 we can reduce the problem to 
 $U(B|_{W + h(W)})$. Thus, we can reduce to the case 
 \[B = 
 \begin{pmatrix}
 0 & 0 & 1 & 0 \\ 
 0 & 0 & 0 & 1 \\ 
 1 & 0 & 0 & 0 \\ 
 0 & 1 & 0 & 0 
 \end{pmatrix}, 
 h = 
  \begin{pmatrix}
 0 & 0 & 1 & 0 \\ 
 0 & 0 & 0 & 1 \\ 
 1 & 0 & 0 & 0 \\ 
 0 & 1 & 0 & 0 
 \end{pmatrix}, 
 A = 
   \Bigg\{\begin{pmatrix}
 \lambda_1 & 0 & 0 & 0 \\ 
 0 & \lambda_2 & 0 & 0 \\ 
 0 & 0 & \lambda_1^{-1} & 0 \\ 
 0 & 0 & 0 & \lambda_2^{-1} 
 \end{pmatrix}, 
 : \lambda_{1,2} \in F^\times
 \Bigg\}.\]  

In this case, 
$Z_{U(B)}(A)$ consist of matrices of the form 
\[   \bigg\{\begin{pmatrix}
\lambda_1 & 0 & 0 & 0 \\ 
0 & \lambda_2 & 0 & 0 \\ 
0 & 0 & c(\lambda_1)^{-1} & 0 \\ 
0 & 0 & 0 & c(\lambda_2)^{-1} 
\end{pmatrix}, \lambda_i \in E^\times\Bigg\}\]
. 

One computes that 
\[\HH^1(Ad_h,A) \cong F^\times / (F^\times)^2\] 
\[\HH^1(Ad_h,Z_G(A)) \cong (F^\times / N_{E/F}(E^\times))^2\]
and 
\[ 
\HH^1(Ad_h,U(B)) \subseteq Her_2(E/F) \cong 
F^\times / N_{E/F}(E^\times).\] Moreover, it is easy to see that the map $\HH^1(Ad_h,A) \to \HH^1(Ad_h,Z_{U(B)}(A))$ correspond to reduction mod $N_{E/F}(E^\times)$ in each coordinate and 
the map 
$\HH^1(Ad_h,Z_{U(B)}(A)) \to \HH^1(Ad_h,U(B))$ correspond to 
the product map 
$(F^\times / N_{E/F}(E^\times))^2 \to (F^\times / N_{E/F}(E^\times))$. 
It follows that every element of 
the form $diag(x,x,x^{-1},x^{-1})$ where 
$x$ is not a norm from $E$ represents a non-trivial class 
in $\ImH{Ad_h}{A}{Z_{U(B)}(A)} \cap \KerH{Ad_h}{Z_{U(B)}(A)}{U(B)}$. 
 
We shall treat now the pairs  
$(\mathbf{GSp}(\omega)(E), \mathbf{GSp}(\omega)(F),c)$ and 
$(\mathbf{Sp}(\omega)(E),\mathbf{Sp}(\omega)(F),c)$. 
First, note that by the uniqueness of a symplectic form over an $F$-vector space these pairs have trivial cohomology. Thus, to show that these pairs are not s-stable, it would suffice to show that $\ImH{c}{A}{Z_G(A)} \ne \{1\}$ for a maximal $c$-split torus  $A$. Choose a $(c,\omega)$-split subspace $L \subseteq V$ and a maximal torus $T \subset GL_E(L)$. Let 
$A = A(T,L)$ be the corresponding $(c,F)$-split torus 
of $GSp(\omega)$ and $A'$ its intersection with $Sp(\omega)$.
Then $A' \cong (E^\times)^n$ with the action of $c$ given by  
$c(a_1,...,a_n) = (c(a_1)^{-1},...,c(a_n)^{-1})$, 
and $Z_{Sp(\omega)(E)}(A') = A'$. Since 
\[\HH^1(c,A) \cong (F^\times / N_{E/F}(E^\times))^n \ne \{1\},\] we get that   
$(Sp(\omega)(E),Sp(\omega)(F),c)$ is not s-stable. 

The case of $(\textbf{GSp}(\omega)(E),\textbf{GSp}(\omega)(F), c)$ for $\dim(V)>2$ follows from the previous one and the exact sequence 
$A' \to A \stackrel{a_1b_1}{\to} E^\times$. The associated long exact seuqnece identifies 
$\HH^1(c,A)$ with the quotient of $\HH^1(c,A')$ by the image of 
$(E^\times)^c / im(A^c) = F^\times / N_{E/F}(E^\times)$. 
The boundary map is not surjective if $\dim(V) > 2$ since 
it embeds $F^\times / N_{E/F}(E^\times)$ diagonally in 
$\dim(V)/2$ copies of it. 
This completes the proof of Theorem \ref{theorem: some non-s-stable pairs} 

Note that for some of the pairs in the list, even if there are non-trivial conditions for its s-stability, this gives a complete classification of stable pairs among them. This is because Riemannian pairs are stable, by Theorem \ref{theorem: riemannian pairs}) and pairs with $G$ commutative are certainly stable. 

Specifically, by Theorem \ref{theorem: some non-s-stable pairs} it follows that the pairs
$(GL(V), O(B), \theta)$ and  $(GL(V(E)), U(B), \theta)$ 
are either non-s-stale, or they are Riemannian, or $G$ is commutative.

\subsection{Some Stable and p-stable Pairs}
\label{subsection: some stable and p-stable pairs}
In this section we finish to prove the results presented in table 
\ref{table:stable pairs} by showing that the conditions for 
s-stability, stability and p-stability presented there holds. 

Theorem \ref{theorem: some non-s-stable pairs} provide useful restrictions on the pairs we should consider: For all the pairs appearing in this theorem all the 3 stability conditions that we consider here are false. 
It remains only to consider the cases which  
are not falsified by Theorem \ref{theorem: some non-s-stable pairs}. 
We shall treat them one by one. 
The main tools we use are the cohomological criterion for stability 
in Proposition \ref{proposition: the centralizer criterion}, 
and for p-stability in Proposition \ref{theorem: criterion for p-stability}
We start with $(SL(V),(GL(V_1(h)) \times GL(V_{-1}(h))) \cap SL(V), Ad_h)$. 

By Theorem \ref{theorem: some non-s-stable pairs} 
it suffice to treat the case where
 $\dim(V_1(h)) \ne \dim(V_{-1}(h))$. 

We shall prove: 

\begin{theorem}
Let $V$ be a finite dimensional vector space over a local field $F$ of characteristic $0$. If $\dim(V_1(h)) \ne \dim(V_{-1}(h))$, then the pair 
\[(SL(V),(GL(V_1(h)) \times GL(V_{-1}(h))) \cap SL(V), Ad_h)\] 
is stable. 	  
\end{theorem} 

This concludes the verification for this pair, as 
stability implies s-stablitiy and p-stability. 

\begin{proof}
The main point in the proof is that if 
$\dim(V^+) \ne \dim(V^-)$, then for every 
symmetric element $r$ in $G$ at least one of the spaces 
$V_1(r)$ or $V_{-1}(r)$ in non-zero. 
Indeed, every primary value of type $B$ and every pair $(\lambda,\lambda^{-1})$ of primary values of 
type $A$ of $r$ contribute the same dimension to $V^+$ and $V^-$, so only the primary values of 
type $C$ can contribute to $\dim(V^+) - \dim(V^-)$.  

First, note that since the pair $(GL(V),GL(V^+) \times GL(V^-),Ad_h)$ is stable 
(\cite[Corollary 7.7.4]{ AG09}), if $r$ is a symmetrization in $SL(V)$ then
\[[r] \in \KerH{Ad_h}{Z_{SL(V)}(r)}{ SL(V)}\] and also 
\[[r] \in \KerH{Ad_h}{Z_{SL(V)}(r)}{ Z_{GL(V)}(r)}.\] 

Since $r$ has a primary value of type $C$, we shall assume  
that $V_1(r)$ is non-zero, the case $V_{-1}(r) \ne \{0\}$ being similar. At least one of 
$V_1(r) \cap V_1(h)$ or $V_1(r) \cap V_{-1}(h)$ is non-zero, and again we shall assume that
$V_1(r) \cap V_1(h)$ is non-zero, leaving the other case to the reader.

Consider the exact sequence 
\[1 \to Z_{SL(V)}(r) \to Z_{GL(V)}(r) \to \det(Z_{GL(V)}(r)) \to 1,\] 
from which we deduce that 
\[\KerH{Ad_h}{Z_{SL(V)}(r)}{Z_{GL(V)}(r)} \cong 
\det(Z_{GL(V)}(r)) / \det(Z_{GL(V)}(\{r,h\})).\]   

But since $Z_{GL(V)}(\{r,h\})$ contains $GL(V_1(r) \cap V_1(h))$ as a subgroup, 
and since the restriction of the character $\det$ to this subgroup agree with the 
usual determinant character on it, we see that $\det(Z_{GL(V)}(\{r,h\})) = F^\times$ so that 
$\KerH{Ad_h}{Z_{SL(V)}(r)}{Z_{GL(V)}(r)} = 1$ and $[r] \in \HH^1(Ad_h,Z_{SL(V)}(r))$ vanishes. It follows that $r$ is stable, and we deduce that the pair is stable. 
\end{proof}

We turn to the pair $(SL(V(E)), SL(V(F)), c)$. 

\begin{theorem}
	If $\dim(V)$ is odd, then $(SL(V(E)), SL(V(F)), c)$ is a stable symmetric pair. 
\end{theorem}

\begin{proof}
	It will suffice to prove that if 
	$\dim(V)$ is odd, then all the centralizers 
	of all the symmetric elements in $SL(V(E))$ are acyclic to $c$ (namely, they have vanishing first cohomology with respect ot $c$).  
	
	Let $r$ be a semi-simple symmetrization in 
	$SL(V(E))$. Let $R = E[r] \cong E[x]/m_r(x)$ for $m_r(x)$ the minimal polynomial of $r$. 
	Since $r$ is semi-simple, $R$ is a product of fields. We can extend the involution 
	$c$ to $R$ by $c(r) = r^{-1}$.   
	The space $V(E)$ admits a natural structure of an $R$-module, which is compatible with the action of 
	$c$. 
	
	We have $Z_{SL(V(E))}(r) \cong SL_R(V(E))$. By Proposition \ref{proposition: vanishing theorem for centralizers}, 
	$\HH^1(c,Z_{GL(V(E))}(r))=1$. Similarly to the previously considered pair, this implies that  
	\begin{align*}
	&\HH^1(c, Z_{SL(V(E))}(r)) \cong \det(Z_{GL(V(E))}(r))^{c} / \det(Z_{GL(V(F))}(r))= \\ 
	&(\det(Z_{GL(V(E))}(r))) \cap F / \det(Z_{GL(V(F))}(r)).
	\end{align*} 
	However, if $K/L$ is a finite extension of fileds then $\det(A|_L) = N_{K/L}(\det(A))$. 
	Moreover, 
	\[\det(Z_{GL(V(E))}(r)) = N_{R/E}(R^\times)\] and \[\det(Z_{GL(V(F))}(r)) = N_{R^c/F}((R^c)^\times).\] 
	It follows that 
	\[\HH^1(c,Z_{SL(V(E))}(r)) \cong (N_{R/E}(R^\times) \cap F^\times) / N_{R^c/F}((R^c)^\times)
	\cong \HH^1(c,R^\times).\] 
	
	Note that the last term is a cohomology of $\ZZ/2\ZZ$ with abelian coefficients, hence 
	it is an abelian group of exponent 2. On the other hand, since $\dim(V)$ is odd, 
	$R$ contains a direct summand which is a field extension of $E$ of odd degree. 
	Since every $d$-th power in $F^\times$ is a norm from $K$, and hence from $R^c$, 
	this means that $\HH^1(c,Z_{SL(V(E))})$ is of odd exponent. Combining these two facts we deduce that 
	$\HH^1(c,SL(V(E)))$ is trivial, so the pair is stable in this case. 
\end{proof}

Next we consider the pairs $(GL_F(V), GL_E(V), Ad_J)$ and $(SL_F(V), SL_E(V), Ad_J)$. 
Note that we already proved in Theorem \ref{theorem: some non-s-stable pairs} that the second 
pair is not s-stable if $F$ is non-Archimedean, unless $\dim(V) = 1$, in which it is stable since it is 
commutative. So it remains to check the case where $E = \RR$ for this pair. 

\begin{theorem}
	The pair $(GL_F(V), GL_E(V), Ad_J)$ is always stable.  
	The pair $(SL_F(V), SL_E(V), Ad_J)$ is stable if $F = \RR$. 
\end{theorem} 

\begin{proof}
	For the pair $(GL_F(V), GL_E(V), Ad_J)$, note that  
	the descendants of this pair are all products of pairs of the form: 
\begin{itemize} 
	\item $(G \times G, \Delta G, (x,y) \mapsto (y,x))$, one for each primary value of type $A$.
	\item $(GL(V(E_\lambda)), GL(V(F_\lambda)), c)$, one for each primary value of type $B$.
	\item $(GL_F(V), GL_E(V), Ad_J)$, one for each primary value of type $C$.
\end{itemize}

Since the cohomology of each of these pairs vanish, the pair is stable. 

We shall consider now the pair $(SL_\RR(V),SL_\CC(V), Ad_J)$. 
Using the triviality of $\HH^1(Ad_J,GL_\RR(V))$ and the sequence 
\[1 \to SL_\RR(V) \to GL_RR(V) \to \RR^\times \to 1\] 
we deduce that 
\[\HH^1(Ad_J,SL_\RR(V)) \cong \RR^\times / N_{\CC/\RR}(\CC^\times) \cong \RR^\times / (\RR^\times)^2\] 

A similar argument shows that, if $r$ is a semi-simple symmetric element then 
\[\HH^1(Ad_J, Z_{SL_F(V)}(r)) \cong \det(Z_{GL_\RR(V)}(r)) / \det(Z_{GL_\CC(V)}(r)).\] 
However, since $\det(Z_{GL_\CC(V)}(r)) = (\RR^\times)^2$ the map  
$\det(Z_{GL_\RR(V)}(r)) / \det(Z_{GL_\CC(V)}(r)) \to \RR^\times / (\RR^\times)^2$ is injective, 
so the obstruction to stability of $r$ is trivial and the pair is stable.         
\end{proof}

Finally, we consider the pairs of the form
$(O_{p,q}(\RR), O_{p_1,q_1}(\RR) \times O_{p_2,q_2}(\RR), Ad_h)$ 
and of the form $(U(B),U(B^+) \times U(B^-),Ad_h)$. 
For the second pair, we will see that the conditions for p-stability and s-stability are different, hence 
stability and p-stability does not agree in general. 

Since we treated the cases where one of the forms $B^+$ or $B^-$ are definite for both pairs in the 
Archimedean case and the case where $F$ is non-Archimedean and the rank of the pair is 1 for the second pair, we shall not consider these cases. 

\begin{theorem}
	The pairs $(O_{p,q}(\RR), O_{p_1,q_1}(\RR) \times O_{p_2,q_2}(\RR), Ad_h)$ and 
	$(U_{p,q}(\RR), U_{p_1,q_1}(\RR) \times U_{p_2,q_2}(\RR))$ 
	are stable if one of $p_1,q_1,p_2$ or $q_2$ vanish.   
	
	In the non-Archimedean case, the pair $(U(B), U(B^+) \times U(B^-), Ad_h)$ is p-stable only if 
	$\rank(B^+)=1$ or $\rank(B^-)=1$.  
\end{theorem}

\begin{remark}
	In the non-Archimedean case, this theorem concludes the classification of stable, p-stable and s-stable 
	pairs among the pairs of the form $(U(B), U(B^+) \times U(B^-), Ad_h)$. 
	Indeed, in the case where the rank of one of the forms $B^-$ or $B^+$ is 1, the pair is 
	known to be stable, hence also p-stable and s-stable (see e.g. 
	\cite[\S 5]{AGRS10} and \cite[\S 7]{Ser97}.)   
\end{remark}

\begin{proof}
	For the Archimedean case, we shall prove this for the first pair, the second one being similar. 
	Assume, without loss of generality, that $q_1 = 0$, so that $B^+$ is positive definite. 
	Let $r$ be a semi-simple symmetrization. We can write $V$ as a sum of primary subspaces of $r$. 
	
	Let $\lambda$ be a non-real primary value of $r$ of type $A$, and let $U_\lambda V_\lambda(r) + V_{\lambda^{-1}}(r)$. The restriction 
	$B|_{U_\lambda \cap V_1(rh)}$ is not definite, so in particular $B|_{V_1(rh)}$ in not definite and $r$ is not a symmetrization in $O(B)$, as the cohomology class of $r$ in $\HH^1(Ad_h,O(B))$ is determined by 
	the class of the form $B|_{V_1(rh)}$, which must be definite in order to be trivial. 
	
	Thus, primary values of type $A$ for $r$ must be real. Let $\lambda$ be such a real primary value.
	Let $v \in V_\lambda$, so that $h(v)$ is in $V_{\lambda^{-1}}$. Scaling properly we may assume that 
	$B(v,h(v)) \in \pm 1$. The option $-1$ is excluded because then $v + h(v)$ is an element of 
	$V_1(h)$ with $B(v + h(v)) < 0$, contrary to the assumption $q_1 = 0$. 
	It follows that $B(v,h(v)) = 1$.  It follows that the vector $\lambda v +  h(v)$ is in $V_1(rh)$ and 
	$B(\lambda v +  h(v)) = 2 \lambda$. So for $r$ to be a symmetrization we must have 
	$\lambda > 0$. 
	
	If $\lambda \in \mathcal{B}(r)$, then it is easy to see that $r|_{V_\lambda}$ is a square of an 
	element of $Z_{O(B|_{V_\lambda(r)})}(r) \cap G^\sigma$ so that it represents the trivial cocycle in 
	$Z_{O(B|_{V_\lambda(r)})}(r)$ and in particular also in $O(B|_{V_\lambda(r)})$, so that 
	$B|_{V_1(h) \cap V_\lambda(r)}$ and $B|_{V_1(rh) \cap V_\lambda(r)}$ are equivalent. 
	By the consideration above, the same hold for primary values of type $A$, because a possitive number have square root.  
	Finally, this clearly holds for the primary value $1$, so writing 
	$V = V_{-1}(r) \oplus V_{-1}(r)^\bot$, we see that 
	\[B|_{V_1(rh) \cap V_{-1}(r)^\bot} \equiv B|_{V_1(h) \cap V_{-1}(r)^\bot},\] 
	and the class $[r] \in \HH^1(Ad_h,Z_G(r) \cap O(V|_{-1}(r)^\bot))$ is trivial. 
	
	By Witt Cancellation Theorem, and using the fact that $B|_{V_1(h)} \equiv B|_{V_1(rh)}$, we 
	deduce that $B|_{V_1(rh) \cap V_{-1}(r)} \equiv B|_{V_1(h) \cap V_{-1}(r)}$ 
	so $[r]\in \HH^1(Ad_h,Z_{O(B|_{V_{-1}}(r))})$ is trivial. Finally, since 
	\[Z_{O(B)}(r) \cong Z_{O(B|_{V_{-1}(r)^\bot})(r|_{V_{-1}(r)^\bot})} \times O(B|_{V_{-1}(r)}),\] 
	we deduce that $[r] \in \HH^1(\theta,Z_{O(B)}(r))$ is trivial and the pair is stable. 
	
	We turn to the non-Archimedean case for the pair $(U(B), U(B^+) \times U(B^-)$. 
	We wish to show that the pair is p-stable only if one of $V^+$ or $V^-$ is 
	one dimensional. 
	By Theorem \ref{theorem: criterion for p-stability}, it is enough to chack that 
	$\KerH{\theta}{Z_G(A)}{G}$ is trivial for a  maximal $(\theta,F)$-split torus $A$. 
	
	Assume that $\dim(V^+) \ge 2$ and $\dim(V^-) \ge 2$. Assume without loss of generality that $\dim(V^-) \ge \dim(V^+)$. Since both $B^+$ and $B^-$ are of rank at least $2$, they represent all of $F^\times$, so by Proposition  \ref{proposition: classification of Hermitian forms} 
	we can find $v^+ \in V^+$ and $v^- \in V^-$ such that $B^+(v^+) = -B^-(v^-) = 1$. 
	The subspace $W' := \Span\{v^+ + v^-\}$ is then $(B,h)$-split, and hence the pair is of rank at least one.
	
	Let $A = A(W,B,T)$ be a maximal $(Ad_h,F)$-split torus. Then
	\[\HH^1(Ad_h,Z_{U(B)}(A)) \cong (F^\times / N_{E/F}(E^\times))^k \times \{C \in \m{H}er(E/F) : C \le B|_{(W + h(W))^\bot}\}\]
	for $k = \dim(W)$.
	If $k > 1$ then the pair is not s-stable, hence also not p-stable.
	As $B^+$ is of rank at least $2$, by Proposition \ref{proposition: classification of Hermitian forms} we have
	$B^+ \equiv H([x]) \oplus C$ for some Hermitian form $C$. As
	\[\dim((W + h(W))^\bot) = \dim(V) - 2 \ge \dim(V^+) > \rank(C),\]
	we have $C \le B|_{(W + h(W))^\bot}$. It follows that $(x,C)$ represents a non-trivial element of $ker((i_A)_*)$, and the pair is not $p$-stable. 	 
\end{proof}

\section{Applications to Representation Theory}
\label{section: applications to representation theory}

In this section we link our geometric results to the representation theory of symmetric pairs.

\subsection{Preliminaries on Unitary Parabolic Induction}
 
Recall that, for a locally compact group $L$, a character $\psi: L \to \CC^\times$ is called \bd{unramified} if $\psi$ is trivial on every compact subgroup of $L$. 
Denote by $X^*_{ur,un}(K)$ the collection of unitary unramified characters of $K$, which are of the form $k \mapsto \phi(k)^a$ for an unramified character $\phi: L\to \RR^\times_{\ge 0}$ and some $a\in i\RR$. 

Let $F$ be a local field of characteristic 0.
Let $\bf{G}$ be a connected modulo center reductive group over $\bar{F}$ defined over $F$, and let $P$ be a parabolic subgroup of $G$.  
Let $M_P$ denote a Levi factor of $P$, which is the quotient of $P$ by its unipotent radical. 
Thus, we have a roof of algebraic groups $M_P \stackrel{\pi_P}{\leftarrow} P \stackrel{i_P}{\rightarrow} G$. 

Let $\Delta_P$ denote the modular character of $P$, given by  
$\Delta_P(x) = |\det(Ad_x)|$. 
Let $\rho_P=\sqrt{\Delta_P(x)}$ denote the positive square root of $\Delta_P$. 

For $\psi \in X_{un,ur}^*(M_P)$
the \textbf{unitary parabolic induction} of $\psi$ is defined by 
\[I_P^{u}(\psi) =  \{f \in C^\infty(G,\CC) : f(pg) = \psi(p) \rho_P(p)^{-1} f(g)\}.\] 

The unitary induction is naturally a $G$-representation via the action induced from right multiplication $g.f(h) = f(hg)$. 
Moreover, we have a skew-bilinear map defined by point-wise multiplication 
\[I_P^{u}(\psi) \otimes I_P^{u}(\psi) \stackrel{f\otimes g \mapsto f\bar{g}}{\to} ind_P^G(\psi \otimes \bar{\psi} \otimes \Delta_P) \cong C^\infty(|\omega_{P \backslash G}|).\] 
When composed with the integration map 
\[C^\infty(P \backslash G,|\omega_{P \backslash G}|) \stackrel{\int_{P \backslash G}}{\to} \mathbb{C},\] we get a non-degenerate Hermitian form on 
$I^{u}_P(\psi)$, so $I^{u}_P(\psi)$ is unitary. 

The following is a straight forward consequence of Bruhat Irreducibility Theorem, first appeared in \cite{bruhat1956representations} (see Theorem \cite[Theorem 4.12]{kolk1996transverse}) in the real case. In the p-adic cases it is even simpler, and follows from a routine application of Bernstein and Zelevinski's Geometric Lemma (\cite[Lemma 2.12]{bernstein1977induced}).  
\begin{proposition}
	\label{proposition: irreducibility criterion}
	Let $G$ be a reductive, connected modulo center, algebraic group defined over $F$. Let $P$ be a parabolic subgroup of $G$. 
	Choose a maximal $F$-split torus $T$ of $P$ and a set of representatives $w_1,...,w_n$ of 
	$P\backslash G/P$ such that $w_iTw_i^{-1}=T$. 
	For $\psi \in X^*_{ur,un}(M_P)$, if $\psi|_T^{w_i}\neq \psi|_T$ for every $w_i \notin P$, then the representation $I_P^{u}(\psi)$ is irreducible. 
\end{proposition}

\subsection{p-stability and Multiplicity One for Principal Series Representations}

Let $(G,H,\theta)$ be a symmetric pair which is not p-stable. In this case, generically, one can associate a functional on the generalized principal series representation of $G$ with each open orbit of $H$ in $G/P$,where $P$ is a minimal $\theta$-split parabolic subgroup of $G$. In particular, a "generic" generalized principal series
has multiplicity more than 1. 
In fact, those results are essentially known in the p-adic case by \cite{BD08} and in the archimedean case by \cite{BvD94}, but for completeness we shall show how they results imply the following
\begin{theorem}
	\label{theorem: Gelfand then p-stable}
	Let $(G,H,\theta)$ be a symmetric pair. If $(G,H)$ is a Gelfand pair, then $(G,H,\theta)$ is p-stable.
\end{theorem}  

\begin{remark}
	In fact, we will show that "many" unitary representations have multiplicity at least the number of open $H$-orbits in a quotient of $G$ by a minimal $\theta$-split parabolic subgroup. 
\end{remark}

%The standard strategy for showing such a result is given by the following 3 steps:
%Let $P$ be a $\theta$-split parabolic subgroup. 

%\begin{itemize}
%	\item Show that for general enough unitary character $\psi$ on $M_P$, the representation 
%	$I^{u}_P(\psi)$ is irreducible. This is done by computing $End_G(I^{u}_P(\psi))$ using a generalization of Mackey theory. 
%	\item Show that, for each open orbit $O \subseteq G/P$ of $H$, there exist an $H$-invariant measure on $O$. This boils down to a computation of a modular character of the stabilizer of a point in $O$.  
%	\item Show that, under some assumptions on $P, H$ and $\psi$, the number of $H$-invariant functionals on $I_P(\psi)$ is bounded from bellow by the number of open $H$ orbits in $P$.       
%\end{itemize}

Before implementing this strategy, let us mention that the converse to Theorem 
\ref{theorem: Gelfand then p-stable} is \textbf{false}. In fact, there are symmetric pairs of compact p-adic groups which are not Gelfand pairs. 
We do not know of examples in which $F = \RR$ or with $G$ quasi-split, 
and the question whether stability and p-stability for those pairs are equivalent is still open to the extent of our knowledge. 

\begin {example}
\label{example 1}
Not every p-stable pair is a Gelfand pair. For example, let $p = 3 \mod 4$ and consider the pair $(G,H,\theta)$ where $G$ is the quaternions
of norm 1 in the quaternion algebra \[\frac{\QQ_p[i,j]}{\{i^2 = p, j^2 = -1, ij = -ji\}} := \HH[p, -1]\], and $\theta(x) = ixi^{-1}= \frac{ixi}{p}.$ This pair is clearly p-stable as there are no non-trivial parabolic subgroups in this case. However, the pair is not a Gelfand pair, 
as we shall now show.
\end{example}

\begin{proposition}
The pair $(G,H,\theta)$ in example \ref{example 1} is not a Gelfand pair for every $p \equiv 3 \mod 4$.
\end{proposition}

\begin{proof}
Note that the Gelfand property of a pair $(G,H)$ is preserved by taking homomorphic images, and by shrinking $G$ to an intermediate group $H \subseteq K \subseteq G$.

Let $G' = \bf{G}(\FF_p)$ be the reduction of $G$ mod $p$, and similarly $H' = \bf{H}(\FF_p)$. Since all the points of $G$ have integral coordinates, there is a reduction map $\pi: G \to G'$ and $\pi|_H : H \to H'$. So, it will be sufficient to prove that $(G',H')$ is not a Gelfand pair.
It will be easier, however, to factor out also the 2-element subgroup $\{\pm 1\}$, and replace $(G',H')$ by $(G'/\{\pm 1\}, H'/\{\pm1\})$. For simplicity call it again $(G',H')$. This pair of finite groups can be realized explicitly as follows.
Let \[A \subseteq G' = \{1 + ai + bk, \quad a,b \in \FF_p\}\] and \[K = G' \cap \FF_p[j]/\{\pm 1\} = \{r \in \FF_p[j] |: N_{\FF_p[j]/\FF_p}(r) = 1\}/\{\pm 1\}.\]
Then $G' = A K$ and in fact it is the semi-direct product of $K$ and $A$. We can identify $A$ with the additive group of $\FF_p[j]$, and under this identification we get $G' \cong K \rtimes \FF_p[j]$ with respect to the action \[x * y = x^2 y, \quad x \in K, y \in \FF_p[j].\]
Under this isomorphism, $H'$ corresponds to the line $\{1 + xi |: x \in \FF_p\}$.
To show that the last pair is not a Gelfand pair, shrink $G'$ to the intermediate group $L$ generated by $A$ and $j$. Then $L = A \rtimes \{1,j\}$, where \[jrj^{-1} = r^{-1}, \quad r \in A.\] It is sufficient to prove that $(L,H')$ is not a Gelfand pair.
Let $\xi: A \to \CC$ be some non-trivial character which is trivial on $H'$.
One can choose $\xi$ such that $\Ind_A^{L}(\xi)$ is irreducible. In fact, this is true for every character $\xi$ such that $\bar{\xi} \ne \xi$.

We obtain 
\begin{align*}
&\Res_{H'}^L \Ind_A^L(\xi) \cong
\Res_{H'}^A \Res^L_A \Ind_A^L(\xi) \cong \\
&\Res_{H'}^A (\xi \oplus Ad_j^*(\xi)) \cong
\Res_{H'}^A (\xi \oplus \bar{\xi}) \cong \\
&\CC \oplus \CC
\end{align*}
where the isomorphism between the first and second row is due to Mackey Theorem for finite groups.
Thus, the multiplicity of the irreducible representation $\Ind_A^L(\xi)$, when restricted to $H'$, is 2, and the pair is not a Gelfand pair.
\qed
\end{proof}

The rest of this section is devoted to the proof of Theorem \ref{theorem: Gelfand then p-stable}
Let $(G,H,\theta)$ be a symmetric pair over $F$. Let $P$ be a minimal $\theta$-split parabolic subgroup of $G$. Let $T \subseteq P$ be a $\theta$-stable maximal $F$-split torus and $A= T^-$ its symmetric part, which is a maximal $(\theta,F)$-split torus. The Levi component of $P$ can be identified with $M_P = Z_G(A)$. Identify $\mathcal{P} = G/P$ with the space of conjugates 
of $P$ in $G$.  

As a direct consequence of \cite[Proposition 13.4]{HW93}, the union of open $H$-orbits in 
$\m{P}$ is exactly the collection of minimal $\theta$-split parabolic subgroups of $G$.
By \cite[Proposition 4.11]{HW93}, if $gPg^{-1}$ is $\theta$-split then $g = g' p$ where $p \in P$ and 
$g'^{-1}\theta(g') \in N_{M_P}(T) = Z_{G}(A) \cap N_G(T)$. In particular, every open $H$-orbit has a point of the from 
$gPg^{-1}$ with $g \in Z_{G}(A) \cap N_G(T)$. 

\begin{proposition}
    \label{proposition: elements conjugating maximal split to maximal split}
    Let $(G,H,\theta)$ be a symmetric pair, let $A\subseteq P$ be a maximal $(\theta,F)$-split torus, and let $T$ be a $\theta$-stable maximal $F$-split torus containing $A$. Let $g\in G$.
	If $g^{-1} \theta(g) \in Z_G(A) \cap N_G(T)$ then 
	$gAg^{-1}$ is a maximal $(\theta,F)$-split torus of G and $gT^+g^{-1} = (gTg^{-1})^+$.  
\end{proposition}

\begin{proof}
	Clearly $gAg^{-1}$ is $F$-split. Since all maximal $(\theta,F)$-split tori are conjugate over $\bar{F}$, it will suffice to show that $gAg^{-1} \subseteq G^\sigma$. 
	This can be seen as follows. Let $a \in A$ and 
	$g$ as above, then 
	\begin{align*}
	&\theta(gag^{-1}) = \theta(g)\theta(a)\theta(g)^{-1} =g(g^{-1}\theta{g})a^{-1}(g^{-1}\theta(g))^{-1}g^{-1} =\\
	&=ga^{-1}g^{-1} = (gag^{-1})^{-1}
	\end{align*} 
	where the conjugation by $g^{-1}\theta(g)$ can be omitted because $g^{-1}\theta(g) \in Z_G(A)$.   
	
	Next, we claim that we also have $gT^+g^{-1}=(gTg^{-1})^+$. This follows from the orthogonality of $Ad_G$ with respect to some non-degenerate $\theta$-invariant form on $\mathfrak{g}$, which exists by \cite[Lemma 7.1.9.]{AG09}, as follows.We have an orthogonal decomposition 
	$\mathfrak{t} = \mathfrak{t}^+ \oplus \mathfrak{t}^-$ and 
	similarly $Ad_g(\mathfrak{t}) = Ad_g(\mathfrak{t})^+ \oplus Ad_g(\mathfrak{t})^-$ and then the condition 
	$Ad_g(\mathfrak{t}^-) = Ad_g(\mathfrak{t})^-$ imply by passing to orthogonal complements that    
	$Ad_g(\mathfrak{t}^+) = Ad_g(\mathfrak{t})^+$.
\end{proof}

With $A$ and $T$ as above, consider the spaces of unramified unitary characters $X^*_{u,ur}(A)$ and $X^*_{u,nr}(T)$. We have 
\[X^*_{u,ur}(A) \cong X^*_{u,ur}(T)^{\sigma},\]
with respect to the action of $\sigma$ induced on characters from the action on $T$.
We have inclusions \[A\subseteq T\subseteq Z_G(A)=M_P \subseteq P,\] and therefore we have a restriction map $r^{M_P}_T: X^*_{u,ur}(M_P)\to X^*_{u,ur}(T)$. 

\begin{lemma}
\label{lemma: relating characters of Levi and split torus}
	The map $r^{M_P}_A$ is injective with image those characters which are trivial on $T\cap [M_P,M_P]$. 
\end{lemma}  

\begin{proof}
	Injectivity is clear, because $T$ is a maximal $F$-split torus of $M_P$. 
	The characters in the image are trivial on $[M_P,M_P]$ because they come from characters of $M_P$. Finally, because the inclusion $T/T\cap[M_P,M_P]\to M_P/[M_P,M_P]$ is an inclusion of tori, we can extend unramified characters of $T/T\cap[M_P,M_P]$ to unramified characters of $M_P$.    
\end{proof}

Since $A \subseteq Z(M_P)$ the intersection $A \cap [M_P,M_P]$ is finite, and hence 
all the elements of $X^*_{u,ur}(A)$ are trivial on $[M_P,M_P]\cap A$. Thus, they can be regarded as characters of $M_P$ via $(res_T^{M_P})^{-1}$, and we will implicitly make this identification from now on. 
\begin{proposition}
	There exists a closed set with empty interior $Z\subseteq X^*_{u,ur}(A)$ such that for $\psi \notin Z$, the representation $I^{u}_P(\psi)$ of $G$ is irreducible.   
\end{proposition}

\begin{proof}
	Let $w_1,...,w_k$ be a set of representatives for the double-coset space $P\backslash G/P$. We can choose them such that $w_iTw_{i}^{-1} = T$. 
	By Proposition \ref{proposition: irreducibility criterion}, the repressentation $I^{u}_P(\psi)$ is irreducible if $\psi|_T^{w_i}(x):=\psi|_T(w_i^{-1}xw_i) \neq \psi|_T(x)$ for every $i$. 
	The $w_i$-s can be split into two types: 
	those for which $w_iPw_i^{-1}$ is $\theta$-split and those for which it is not. 
	
	If $w_i \notin P$ and $w_i P w_i^{-1}$ is $\theta$-split, then we can assume that $w_i$ stabilizes $A$, and that $w_iPw_i^{-1}$ correspond to a choice of positive roots in $\Phi(A,G)$ which is different from the one given by $P$. In particular, we may assume that $w_i$ acts non-trivially on $A$, since it acts non-trivially on a root system in $X^*(A)$. But then, the condition $\psi=\psi^{w_i}$ impose non-trivial linear relation on $X^*_{u,ur}(A)$ hence cut a non-trivial affine subspace of $X^*_{u,ur}(A)$, which is of empty interior. 
	The other option is that $w_iPw_i^{-1}$ is not $\theta$-split. In this case, we get that the action of $w_i$ on $T$ does not preserve the sub-torus $A$. But then for $\psi$ outside of a non-trivial Zariski closed subset of $X^*_{u,ur}(A)$, the twisted character $\psi^{w_i}$ no longer satisfies $\theta(\psi^{w_i})=(\psi^{w_i})^{-1}$ so clearly the equality $\psi=\psi^{w_i}$ can not hold.             
\end{proof}

Our next goal is to show that in fact for every $\psi \in X^*_{u,ur}$, the representation $I^{u}_P(\psi)$ admits at least as many independent invariant functionals as the number of open $H$-orbits in $G/P$. This, together with the generic irreducibility presented above, will show that the pair $(G,H)$ is not a Gelfand pair if it is not p-stable.

Recall that, for a homogenuous action of an $F$-group $K$ on an $F$-variety $X$ with a point $x\in X$, stabilizer $K_x$, and a character $\psi$ of $K_x$, by a $(K,\psi)$-equivariant distribution on $X$ we mean a $K\times K_x$-invariant distribution on $K$. In other words, a $(K,\psi)$-equivariant distribution on $X$ is a $K$-invariant functional on the small induction $ind_{K_x}^K(\psi)$.   

\begin{lemma}
\label{lemma: orbits unimodular}
	Let $(G,H,\theta)$ be a symmetric pair. Let $\psi \in X^*_{ur}(A)$ (not necessarily unitary). Let $P$ be a minimal $(\theta,F)$-split parabolic subgroup of $G$. For every open orbit $O\subseteq G/P$, there is an $(H,\psi)$-equivariant distribution on $O$.     
\end{lemma}

\begin{proof}
	By Frobenious reciprocity (see e.g. \cite[Theorem 2.5.7]{AG09}), 
	we have to verify that $(\psi \otimes \Delta_H)|_{H_x} \otimes \Delta_{H_x}^{-1}$ is trivial for some (and hence all) $x\in O$ . Choose the point $x$ to represent a parabolic subgroup $P_x$ containing $A$. Then $\psi|_{H_x}$ is trivial because $\theta(\psi)=\psi^{-1}$. Furthermore, $\Delta_H$ is trivial because $H$ is reductive, and $\Delta_{H_x}$ is trivial because $H_x=(M_{P_x})^\theta$ is reductive as well. To conclude, all the characters $\psi|_{H_x}, \Delta_H|_{H_x}, \Delta_{H_x}^{-1}$ are trivial and hence also their product.   
\end{proof}

Theorem \ref{theorem: Gelfand then p-stable} now follows from the following proposition.

\begin{proposition}
    Let $(G,H,\theta)$ be a symmetric pair, $P$ be a minimal $\theta$-split parabolic subgroup of $G$ and $A$ a maximal $(\theta,F)$-split torus of $G$ cintained in $P$. There exist an open dense  subset $V\subseteq X^*_{u,ur}(A)$ such that for every $\psi \in V$, the representation $I_P^{u}(\psi)$ is irreducible and admits at least as many $H$-invariant functionals as the number of open $H$-orbits in $P \backslash G$.      
\end{proposition}

\begin{proof}
    By Proposition \ref{proposition: irreducibility criterion} there is an open dense $V'\subseteq X^*_{u,ur}(A)$ such that the representation $I_P^{u}(\psi)$ is irreducible for $\psi \in V$. By \cite[Theorem 2.7]{BD08} in the non-Archimedean case and \cite[Theorem 5.1]{vdB88} in the Archimedean case, there is an open dense subset $V''\subseteq X^*_{u,ur}(A)$ such that for every $\psi \in V''$ and for every open $H$-orbit $O$ in $G/P$, every $(H,\psi \otimes \rho^{-1})$-invariant distribution on $O$ extend to an $H$-invariant functional on $I_P^{u}(\psi)$. Moreover, we can choose such extensions to vanish on other open orbits, and hence in particular we can choose them linearly independent. Taking now $V=V'\cap V''$ and observing that it is open and dense in $X^*_{u,ur}(A)$, we get the result.    
\end{proof}

\subsection{Stability and the Gelfand Property}

Throughout this section we assume that $F \ne \CC$. In the case $F = \CC$ stability holds for every connected pair so it is irrelevant for the verification of the Gelfand property.

As we already mentioned, for many pairs the stability of the pair is known to imply its Gelfand property.
A list of such pairs can be found in \cite{AG10, Aiz13}. Since we proved stability for many pairs, we reveal as a result several new Gelfand pairs.

First, let us present without proofs the main factors in the method used to show for a given pair
that stability $\Rightarrow$ Gelfand property.

Let $X$ be the $F$-points of an affine algebraic variety over $F$. Let $S^*(X)$ denote the space
of Schwartz distributions on $X$, as in \cite{AG09}. If $G$ acts on $X$, we denote by $X/G$ the set $\textbf{X}(F)/\textbf{G}(F)$.
Recall that $\mathfrak{g}^\sigma$ is the space of symmetric elements in the Lie algebra $\mathfrak{g}$.
\begin{definition}
	Let $(G,H,\theta)$ be a symmetric pair. An element $g \in G$ is called \textbf{admissible} if $Ad_g$ commutes with $\theta$ and
	$Ad_g|_{\mathfrak{g}^\sigma}$ stabilizes all the closed $H$-orbits in $\mathfrak{g}^\sigma$.
\end{definition}

More generally, if a linear algebraic group $L$ acts on an affine variety $X$ and $\tau: X \to X$ is an involution, we say that 
$\tau$ is $L$\textbf{-admissible} if $\tau^2 \in Im(L \to Aut(X))$, $\tau$ stabilizes the image of $L$ in $Aut(X)$ and $\tau$ preserve all closed $L$-orbits in $X$. 

Let $\textbf{L}$ be a reductive group defined over $F$ and let $(\pi,V)$ be a finite dimensional algebraic representation of $L = \textbf{L}(F)$.
We denote by $Q(V)$ the direct complement of the trivial component $V^L$ of the representation $V$, and by $\Gamma(V) \subseteq Q(V)$ the set of nilpotent elements, namely those elements $v \in V$ such that $0 \in \overline{\pi(G)v}$. Finally, we denote $R(V) = Q(V) - \Gamma(V)$.

\begin{definition}[{\cite[Definition 7.4.2]{AG09}}] A symmetric pair $(G,H,\theta)$ is called \textbf{regular} if for every admissible $g \in G$ for which
	\[S^*(R(\mathfrak{g}^\sigma))^{H \times H} \subseteq S^*(R(\mathfrak{g}^\sigma))^{Ad_g}\] we also have
	\[S^*(Q(\mathfrak{g}^\sigma))^{H \times H} \subseteq S^*(Q(\mathfrak{g}^\sigma))^{Ad_g}\]
\end{definition}

\begin{definition}[{\cite[Definition 8.1.2]{AG09}}]
	A pair $(G,H)$ is \textbf{GP2} if for every irreducible admissible representation $\pi$ of $G$,the following inequality holds.
	\[\dim(Hom_H(\pi,\CC)) \cdot \dim(Hom_H(\tilde{\pi},\CC)) \le 1.\]
\end{definition}

In simple terms, this means that either both $\pi$ and its contragredient representation admits a one dimensional space of $H$-invariant continuous functionals, or that one of them has no non-zero $H$-invariant functional at all. 

The following result shows that stability and regularity of the descendants suffices for checking the GP2 property. With slightly more we can actually upgrade it to the Gelfand property, based on the notion of tame element. 

\begin{theorem}[{\cite[Theorem 7.4.5]{AG09}}]
	\label{theorem: Dima Rami's theorem}
	Let $(G, H, \theta)$ be a stable symmetric pairs such that all its descendants are regular. Then $(G,H,\theta)$ is GP2.
\end{theorem}

The next step is to show that in many cases GP2 implies the Gelfand property. This is done based on a method to compare $\pi$ and its contragradient representation. The main tool is the existence of some special involutions of $G$. 

\begin{definition}
	Let $G$ be group. An anti-involution $\tau: G \to G$ is called $Ad_G$-admissible, if 
	$\tau$ preserves closed $G$-conjugacy classes in $G$. 
\end{definition}

The importance of these kind of anti-involution for us is the next result, which allows to deduce that a pair is a Gelfand pair given that it satisfies the property GP2 above.

\begin{proposition}[{\cite[Corollary 8.2.3]{AG09}}]
	Let $(G,H,\theta)$ be a symmetric pair. Assume that there is an $Ad_G$-admissible anti-involution 
	$\tau: G \to G$ such that $\tau(H) = H$. Suppose further that $(G,H)$ is GP2. Then $(G,H)$ is a Gelfand pair. 
\end{proposition} 

\begin{proposition}[{\cite[Corollary 8.2.3]{AG09}}]
	If $(G,H)$ is GP2 and there is an $Ad_G$-admissible anti involution $\tau : G \to G$ with
	$\tau(H) = H$ then the pair $(G,H)$ is a Gelfand pair.
\end{proposition}

\begin{corollary}
	\label{corollary: criterion for gelfand}
	Let $(G,H,\theta)$ be a symmetric pair. If $(G,H,\theta)$ is stable, all its descendants are regular,
	and there is an $Ad_G$-admissible anti-involution $\tau : G \to G$ with $\tau(H) = H$, then the pair
	$(G,H)$ is a Gelfand pair.
\end{corollary}

This corollary is the upshot of the method introduced in \cite{AG09} to verify the Gelfand property for symmetric pairs. 
Our goal now is to apply it to some pairs for which, essentially, the stability is the only remaining ingredient. In some cases, however, we need to slightly adapt the arguments of \cite{AG09} for the regularity and existence of anti-involution as above.  

\subsection{Verification of the Gelfand Property for Certain Pairs}

In this section we shall sketch the neccesary modifications of the method developed in \cite{AG09} to deduce the Gelfand property from the stability of a pair. Then, we use these modification to derive the main theorem of this paper (Theorem \ref{theorem: main theorem}). 

The main tool to show that a pair is regular is based on the observation that distributions on a linear space $V$ which are, together with their Fourier transform, supported on the same non-degenerate quadratic cone, must be of homogeneity degree $\dim(V)/2$. The proof of this fact is based on the Weil representation. For details, see \cite[\S 5]{AG09}. 

\begin{definition}
	Let $(G,H,\theta)$ be a symmetric pair. $(G,H,\theta)$ is called \textbf{very special} if, for every nilpotent 
	element $e \in \mathfrak{g}^\sigma$, and any completion of it to a graded $\mathfrak{sl}_2$-triple 
	$(f,h,e)$ in $\mathfrak{g}$, we have $tr(h|_{(\mathfrak{g}^\sigma)^e}) < \dim(Q(\mathfrak{g}^\sigma))$. 
\end{definition}

Very-speciality of a pair is stronger then its regularity. Namely, we have

\begin{proposition}[{\cite[Remark 7.4.3]{AG09}}]
\label{proposition: very special then regular}
A very special symmetric pair is regular.  
\end{proposition}

In practice, usually in order to prove that a pair is regular, one either shows that all the $\theta$-admissible elements in $G$ actually lies in $H$, or that the pair is very-special.  
We shall consider specifically pairs for which this alternative holds.

\begin{definition}
A symmetric pair $(G,H,\theta)$ is called \textbf{trivially regular} if it is either very-special, 
or every admissible element in $G$ is in $H$. 
\end{definition}

The advantage of trivial regularity over regularity is that it is stable under base-extension and "restriction of the center".  

\begin{lemma}
\label{lemma: trivially regular stable to base-change}
	Let $(G,H,\theta)$ be a symmetric pair. Let $E/F$ be a finite field extension. 
	Assume that $(\mathbf{G}(E),\mathbf{H}(E),\theta)$ is trivially regular. Then so is 
	$(G,H,\theta)$.
\end{lemma}

\begin{proof}
	Suppose first that $(\mathbf{G}(E),\mathbf{H}(E),\theta)$ is very-special. 
	For every nilpotent element $e \in \mathfrak{g}^\sigma(E)$, and for every 
	completion of it to a graded $sl_2$-triple $(f,h,e)$, we have 
	$Tr(ad_h(E)|_{(\mathfrak{g}^\sigma)^e(E)}) < dim_E(\mathfrak{g}^\sigma(E))$. In particular, by the stability of dimension and trace under field extension, we get 
	$Tr(ad_h|_{(\mathfrak{g}^\sigma)^e}) < \dim(\mathfrak{g}^\sigma)$ for every $e \in \mathfrak{g}^\sigma$ nilpotent and 
	every completion $(f,h,e)$ over $F$. This implies that $(G,H,\theta)$ is very special, hence regular. 
	
	Similarly, if every admissible element of $\mathbf{G}(E)$ lie in $\mathbf{H}(E)$, then 
	every admissible element of $G$ is in $\mathbf{H}(E) \cap G = H$. It follows again that $(G,H,\theta)$ is regular. 
\end{proof}

\begin{lemma}
	\label{lemma: restriction of center for trivially regular}
	Let $(G,H,\theta)$ be a symmetric pair. Let $G' \subseteq G$ be a $\theta$-stable subgroup of 
	$G$, and let $Z \subseteq Z(G)$ be a $\theta$-stable connected subgroup such that 
	$\theta$ is trivial on $Z$ and $G =G'Z$. If $(G,H,\theta)$ is trivially regular, so is 
	$(G',G'\cap H, \theta)$.  
\end{lemma}

\begin{proof}
	Suppose first that $(G,H,\theta)$ is very-special. Let $e \in \mathfrak{g}'^\sigma$, the symmetric part of the Lie algebra of $G$. Let $(f,h,e)$ be a completion of it to a graded $\mathfrak{sl}_2$-algebra in $\mathfrak{g}'$. On the level of Lie algebras, unless
	$Z \subseteq G'$, we have $\mathfrak{g} = \mathfrak{g}' \oplus Lie(z)$. But it is clear that 
	$Q(\mathfrak{g}'^\sigma) \cong Q(\mathfrak{g}^\sigma)$ in this case, hence the eigen-values of $h$ on the centralizer of $e$ in both cases is the same, and they have the same dimension. 
	By the definition of very-speciality, it follows that $(G',H',\theta)$ is very special in that case. 
	
	If every admissible $g \in G$ is in $H$, then clearly every admissible $g \in G'$ is in 
	$H \cap G' = H'$. We deduce that $(G',H',\theta)$ is trivially regular. 
\end{proof}

We now ready to use our stability results to classify the Gelfand pairs among several classical symmetric pairs. We will do this for the pairs in the following list.

\begin{List}
\label{list: pairs for Gelfand verification}
\begin{align*}
	&(SL(V), GL(V^+) \times GL(V^-) \cap SL(V), Ad_h), (O(B),O(B^+) \times O(B^-))\text{ over  }\RR,\\ 
	&(U(B), U(B^+) \times U(B^-))
, (SL(V(E)),SL(V(F))), (GL_F(V),GL_E(V))\text{ and } 
(SL_F(V),GL_E(V),Ad_J). 
\end{align*}
\end{List}

%For convenience, in all the results we shall state in the rest of this sub-section, we shall refer to those pairs are \textbf{"pairs from the list"}. 

The scheme of the proof is identical to all of them: In cases where they are not p-stable we deduce that the pair is not a Gelfand pair immediately from Theorem 
\ref{theorem: Gelfand then p-stable}. In cases where the pair is stable, we show it is very-special, mostly referring to previous work or slightly modifying known arguments. Then, we find an admissible anti-involution preserving $H$ for each of them.  

\begin{proposition}
	\label{proposition: pairs in the list have admissible anti-involution}
	All the pairs in List \ref{list: pairs for Gelfand verification} has an anti-involution $\tau$ for which $\tau(H)=H$ and such that $\tau$ stabilizes the closed conjugacy classes of $G$. 	
\end{proposition}

\begin{proof}
	We choose the following involutions for the pairs in the list: 
	\begin{itemize}
		\item $\tau(x)=x^t$ for  
		$(SL(V), GL(V^+) \times GL(V^-))$, 
		$(U(B), U(B^+) \times U(B^-))$,  $(SL(V(E)),SL(V(F)))$, $(GL_F(V),GL_E(V))$ and 
		$(SL_F(V),GL_E(V),Ad_J)$. Here we assume for all those pairs that the involution of the pair commute with transposition, This can always be achieved up to conjugation. 
		\item  $\tau(x)=x^{-1}$ for the pair $(O(B),O(B^+) \times O(B^-))$. 
	\end{itemize}

	We have now to verify that those anti-involutions are indeed admissible and preserve $H$. The fact that $H$ is preserved follows in the first case from the assumption that $\theta$ commutes with transposition.  
	For $\tau(x)=x^{-1}$ it is obviously true. To check that it is admissible, since all those anti-involutions clearly normalizes the action of $H$, it suffices to check that all closed closed conjugacy classes in $G$. It is well known that 
	transposition preserve closed conjugacy classes in 
	$SL(V)$ and $GL(V)$. 
	
	Consider the group $U(B)$. In this case, the compatibility of $\tau$ with 
	the unitary structure means that $\tau(x) = \bar{x}^{-1}$, where 
	$\bar{x}$ is the conjugate of $x$ for the quadratic extension $E/F$ over which $B$ is defined. 
	Now, if $x$ is semi-simple, its primary values come in pairs 
	$\lambda, \bar{\lambda}^{-1}$, with conjugate-dual eigen-spaces. It follows immediately that $\bar{x}^{-1}$ has the same primary decomposition as $x$, with the same pairing among its primary subspaces, hence $\bar{x}^{-1}$ is conjugate to $x$. 
	A similar argument shows the claim for $\tau$ in the case of $O(B)$.   
\end{proof}

Our next goal is to show that all the descendants of the pairs in the list are regular. 
This is mainly done in \cite{AG09} and \cite{AG10}. We will use freely the fact that in 
the proof of regularity for all the pairs in these papers, actually trivial regularity is proven.

\begin{proposition}
	\label{proposition: trivial regularity for pairs from the list}
	All the descendants of all the pairs in List \ref{list: pairs for Gelfand verification} are trivially regular. 
\end{proposition}

\begin{proof}
	For the pair $(O(B),O(B^+)\times O(B^-),Ad_h)$ this is shown in 
	\cite[ Theorem 3.0.5]{AG10}. 
	All the descendants of the pair $(GL(V),GL(V^+) \times GL(V^-), Ad_h)$ are trivially regular, by (\cite[Section 7.7]{AG09}), and by Lemma \ref{lemma: restriction of center for trivially regular} this holds also for $(SL(V),GL(V^+) \times GL(V^-) \cap SL(V),Ad_h)$. 
	
	The pairs 
	\begin{align*}
	&(U(B), U(B^+) \times U(B^-)),  (SL(V(E)),SL(V(F))), (GL_F(V),GL_E(V))\\
	&\text{ and }(SL_F(V),GL_E(V),Ad_J)    
	\end{align*}
	 are $F$-forms of the pairs 
	\begin{align*}
	&(GL(V),GL(V^+)\times GL(V^-)), (G \times G,G), (GL(V),GL(V^+)\times GL(V^-))\\
	&\text{and }(SL(V), GL(V^+)\times GL(V^-)\cap SL(V))
	\end{align*} respectively, so their descendants are $F$-forms of the descendants of those pairs. The trivial regularity of the descendants of the pairs $(U(B), U(B^+) \times U(B^-))$,  $(SL(V(E)),SL(V(F)))$, $(GL_F(V),GL_E(V))$ and 
	$(SL_F(V),GL_E(V),Ad_J)$ follows now from the trivial regularity of the descendants of the pairs of the form $(GL(V),GL(V^+)\times GL(V^-)), (G \times G,G), (GL(V),GL(V^+)\times GL(V^-))$ and $(SL(V), GL(V^+)\times GL(V^-)\cap SL(V))$ via Lemma \ref{lemma: trivially regular stable to base-change} (The only pair for which trivial regularity of the descendants still have to be discussed is $(G\times G,G)$. This is shown in \cite[\S 7.6]{AG09}). 
\end{proof}

As a corollary, we deduce the following result regarding the Gelfand property of the pairs we consider.
\begin{theorem}
	\label{theorem: stability imply Gelfand for pairs from the list}
	All the stable pairs among the pairs in List \ref{list: pairs for Gelfand verification} are Gelfand pairs. 
\end{theorem}

\begin{proof}
	By proposition \ref{proposition: pairs in the list have admissible anti-involution}, all pairs from the list admits an $Ad_G$-admissible anti-involution preserving $H$. 
	By proposition \ref{proposition: trivial regularity for pairs from the list}, all the descendants of pairs from the list are trivially regular, hence regular. 
	The result now follows from Corollary \ref{corollary: criterion for gelfand}. 
\end{proof}

We end up by explicitly classifying the pairs in List \ref{list: pairs for Gelfand verification} into Gelfand and non-Gelfand pairs, at least for the cases in which this classification was unknown before. 

\begin{theorem}[{Main Theorem}]
\label{theorem: main theorem of paper}
	$\newline$
	\begin{itemize}	
		\item The pair $(SL(V), (GL(V^+) \times GL(V^-) \cap SL(V)), Ad_h)$ is a Gelfand pair if and only if $\dim(V^+) \ne \dim(V^-)$.
		\item\footnote{The non-Archimedean version of this result is more involved, and hence ommited from this version of the paper. It appears in the previous arxiv version} Let $B = B^+ \oplus B^-$ be a non-degenerate quadratic form over $\RR$.
		The pair $(O(B),O(B^+) \times O(B^-))$ is a Gelfand pair if and only if either $B^+$ or $B^-$ is definite.
		\item Let $B = B^+ \oplus B^-$ be a Hermitian form over $\CC$. The pair
		$(U(B), U(B^+) \times U(B^-))$ is a Gelfand pair if and only if $B^+$
		or $B^-$ is definite.
		\item Let $B = B^+ \oplus B^-$ be a Hermitian form for a quadratic extension $E/F$ of non-Archimedean local fields of characteristic 0. The pair
		$(U(B), U(B^+) \times U(B^-))$ is a Gelfand pair if and only if $B^+$
		or $B^-$ is of rank 1. 
		\item 	Let $E/F$ be a quadratic extension of local fields of characteristic $0$.
		The pair $(SL_n(E),SL_n(F))$ is a Gelfand pair if and only if $n$ is odd.
		\item  Let $E/F$ be a quadratic extension of local fields of characteristic $0$. The pair $(GL_F(V),GL_E(V),Ad_J)$ is a Gelfand pair.
		\item 	Let $E/F$ be a quadratic extension of local fields of characteristic $0$. The pair $(SL_F(V),SL_E(V),Ad_J)$ is a Gelfand pair if and only if $\dim(V) = 1$ or $F$ is Archimedean.		
\end{itemize}

\end{theorem}

\begin{proof}
	This is just an amalgamation of results which we already stated in this and the previous section. By theorem \ref{theorem: Gelfand then p-stable}, the non-p-stable among them are not Gelfand pairs. By Theorem \ref{theorem: stability imply Gelfand for pairs from the list}, the stable pairs among them are Gelfand pairs. 
	By the results summarized in table \ref{table:stable pairs} in Section \ref{section: calculations} one sees that all the pairs in List \ref{list: pairs for Gelfand verification} are either stable or non-p-stable. It follows that stability and the Gelfand property for them are equivalent. The result now follows from the criteria for stability for those pairs summarized in the same table.  
\end{proof}
%The non-Archimedean version of this result is more involved, and hence ommited from this version of the paper. There are still few interesting cases where this pair is not stable but is p-stable. They probably have to be treated using other methods.

\bibliographystyle{unsrt}
\bibliography{ref}

\begin{thebibliography}{10}

\bibitem{vD86}
G.~van Dijk.
\newblock On a class of generalized gelfand pairs.
\newblock {\em Math. Z.}, 193(4):581--593, 1986.

\bibitem{Fli91}
Yuval~Z. Flicker.
\newblock On distinguished representations.
\newblock {\em J. Reine Angew. Math.}, 418:139--172, 1991.

\bibitem{BvD94}
E.~P.~H. Bosman and G.~van Dijk.
\newblock A new class of gelfand pairs.
\newblock {\em Geom. Dedicata}, 50(3):261--282, 1994.

\bibitem{JR96}
Herv\'e Jacquet and Stephen Rallis.
\newblock Uniqueness of linear periods.
\newblock {\em Compositio Math.}, 102(1):65--123, 1996.

\bibitem{AGRS10}
Avraham Aizenbud, Dmitry Gourevitch, Stephen Rallis, and G\'erard Schiffmann.
\newblock Multiplicity one theorems.
\newblock {\em Ann. of Math. (2)}, 172(2):1407--1434, 2010.

\bibitem{AG09}
Avraham Aizenbud and Dmitry Gourevitch.
\newblock Generalized {H}arish-{C}handra descent, {G}elfand pairs, and an
  {A}rchimedean analog of {J}acquet-{R}allis's theorem.
\newblock {\em Duke Math. J.}, 149(3):509--567, 2009.
\newblock With an appendix by the authors and Eitan Sayag.

\bibitem{AG10}
Avraham Aizenbud and Dmitry Gourevitch.
\newblock Some regular symmetric pairs.
\newblock {\em Trans. Amer. Math. Soc.}, 362(7):3757--3777, 2010.

\bibitem{SZ12}
Binyong Sun and Chen-Bo Zhu.
\newblock Multiplicity one theorems: the {A}rchimedean case.
\newblock {\em Ann. of Math. (2)}, 175(1):23--44, 2012.

\bibitem{AS12}
Avraham Aizenbud and Eitan Sayag.
\newblock Invariant distributions on non-distinguished nilpotent orbits with
  application to the {G}elfand property of {$(GL_{2n}(\mathbb
  R),\,Sp_{2n}(\mathbb R))$}.
\newblock {\em J. Lie Theory}, 22(1):137--153, 2012.

\bibitem{Aiz13}
Avraham Aizenbud.
\newblock A partial analog of the integrability theorem for distributions on
  {$p$}-adic spaces and applications.
\newblock {\em Israel J. Math.}, 193(1):233--262, 2013.

\bibitem{Zha10}
Lei Zhang.
\newblock Gelfand pairs {$({\rm Sp}_{4n}(F),{\rm Sp}_{2n}(E))$}.
\newblock {\em J. Number Theory}, 130(11):2428--2441, 2010.

\bibitem{HW93}
A.~G. Helminck and S.~P. Wang.
\newblock On rationality properties of involutions of reductive groups.
\newblock {\em Adv. Math.}, 99(1):26--96, 1993.

\bibitem{BD08}
Philippe Blanc and Patrick Delorme.
\newblock Vecteurs distributions {$H$}-invariants de repr\'esentations
  induites, pour un espace sym\'etrique r\'eductif {$p$}-adique {$G/H$}.
\newblock {\em Ann. Inst. Fourier (Grenoble)}, 58(1):213--261, 2008.

\bibitem{vdB88}
E.~P. van~den Ban.
\newblock The principal series for a reductive symmetric space. {I}.
  {$H$}-fixed distribution vectors.
\newblock {\em Ann. Sci. \'Ecole Norm. Sup. (4)}, 21(3):359--412, 1988.

\bibitem{Ser97}
Jean-Pierre Serre.
\newblock {\em Galois cohomology}.
\newblock Springer-Verlag, Berlin, 1997.
\newblock Translated from the French by Patrick Ion and revised by the author.

\bibitem{Ber10}
Gr\'egory Berhuy.
\newblock {\em An introduction to {G}alois cohomology and its applications},
  volume 377 of {\em London Mathematical Society Lecture Note Series}.
\newblock Cambridge University Press, Cambridge, 2010.
\newblock With a foreword by Jean-Pierre Tignol.

\bibitem{BJ05}
Armand Borel and Lizhen Ji.
\newblock Compactifications of locally symmetric spaces.
\newblock {\em J. Differential Geom.}, 73(2):263--317, 2006.

\bibitem{BT65}
Armand Borel and Jacques Tits.
\newblock Groupes r\'eductifs.
\newblock {\em Inst. Hautes \'Etudes Sci. Publ. Math.}, (27):55--150, 1965.

\bibitem{Ad13}
Jeffrey Adams.
\newblock Galois cohomology of real groups, 2013.

\bibitem{Ser73}
J.-P. Serre.
\newblock {\em A course in arithmetic}.
\newblock Springer-Verlag, New York-Heidelberg, 1973.
\newblock Translated from the French, Graduate Texts in Mathematics, No. 7.

\bibitem{Jac62}
Ronald Jacobowitz.
\newblock Hermitian forms over local fields.
\newblock {\em Amer. J. Math.}, 84:441--465, 1962.

\bibitem{bruhat1956representations}
Fran{\c{c}}ois Bruhat.
\newblock Sur les repr{\'e}sentations induites des groupes de lie.
\newblock {\em Bulletin de la Soci{\'e}t{\'e} Math{\'e}matique de France},
  84:97--205, 1956.

\bibitem{kolk1996transverse}
Johan~AC Kolk and VS~Varadarajan.
\newblock On the transverse symbol of vectorial distributions and some
  applications to harmonic analysis.
\newblock {\em Indagationes mathematicae}, 7(1):67--96, 1996.

\bibitem{bernstein1977induced}
IN~Bernstein and Andrey~V Zelevinsky.
\newblock Induced representations of reductive p-adic groups. i.
\newblock In {\em Annales scientifiques de l'{\'E}cole normale sup{\'e}rieure},
  volume~10, pages 441--472, 1977.

\end{thebibliography}

%\begin{thebibliography}{}
%\bibitem[Ad00] {adams} J. Adams, {\it Galois Cohomology of Real Groups}, (2000). arXiv:1310.7917[math.GR].

%\end{thebibliography}

\end{document}